\documentclass[aos]{imsart}
\pdfoutput=1
\newcommand{\ARXIV}{} 
\RequirePackage{silence}
\RequirePackage[english]{babel}
\RequirePackage[ascii]{inputenc}
\RequirePackage[T1]{fontenc}
\RequirePackage{microtype,amsmath, amsthm, amsfonts,amssymb,mathtools,braket,bm,xcolor,float}
\RequirePackage[numbers]{natbib}
\RequirePackage[bookmarksnumbered,colorlinks,citecolor=blue,urlcolor=blue,linkcolor=blue,anchorcolor=green,breaklinks=true]{hyperref}
\RequirePackage{graphicx}
\RequirePackage[capitalize]{cleveref}
\RequirePackage{caption}
\RequirePackage{subcaption}
\usepackage{makecell}
\usepackage{dirtytalk}

\startlocaldefs
\newtheorem{theorem}{Theorem}[section]
\newtheorem*{theorem*}{Theorem}

\newtheorem{corollary}[theorem]{Corollary}
\newtheorem{prop}[theorem]{Proposition}
\crefname{prop}{Proposition}{Propositions}
\newtheorem{lemma}[theorem]{Lemma}
\newtheorem{remark}[theorem]{Remark}
\newtheorem{remark*}{Remark}
\newtheorem{fact}[theorem]{Fact}
\newtheorem{problem}[theorem]{Problem}
\newtheorem*{problem*}{Problem}
\theoremstyle{definition}
\newtheorem{definition}[theorem]{Definition}
\newtheorem{example}[theorem]{Example}
\floatstyle{boxed}\newfloat{Algorithm}{ht}{alg}
\AddToHook{env/corollary/begin}{\crefalias{theorem}{corollary}}
\AddToHook{env/prop/begin}{\crefalias{theorem}{prop}}
\AddToHook{env/lemma/begin}{\crefalias{theorem}{lemma}}
\AddToHook{env/remark/begin}{\crefalias{theorem}{remark}}
\AddToHook{env/fact/begin}{\crefalias{theorem}{fact}}
\AddToHook{env/problem/begin}{\crefalias{theorem}{problem}}
\AddToHook{env/definition/begin}{\crefalias{theorem}{definition}}
\AddToHook{env/example/begin}{\crefalias{theorem}{example}}
\crefname{Algorithm}{Algorithm}{Algorithms}
\numberwithin{equation}{section}
\DeclareMathOperator{\poly}{poly}
\DeclareMathOperator{\op}{op}
\DeclareMathOperator{\vol}{vol}
\DeclareMathOperator{\cut}{cut}
\DeclareMathOperator{\ch}{ch}

\DeclareMathOperator{\Mat}{Mat}
\DeclareMathOperator{\GL}{GL}
\DeclareMathOperator{\tr}{Tr}

\DeclareMathOperator{\PD}{PD}
\DeclareMathOperator{\SSPD}{SPD}
\DeclareMathOperator{\vect}{vec}
\DeclareMathOperator*{\argmin}{arg\,min}

\DeclarePairedDelimiter{\abs}{\lvert}{\rvert}
\DeclarePairedDelimiter{\norm}{\lVert}{\rVert}

\newcommand{\R}{{\mathbb{R}}}
\renewcommand{\P}{{\mathbb{P}}}
\newcommand{\C}{{\mathbb{C}}}
\renewcommand{\H}{{\mathbb{H}}}
\newcommand{\G}{{\mathbb{G}}}

\newcommand{\N}{{\mathbb{N}}}

\newcommand{\otheta}{\overline{\Theta}}
\newcommand{\htheta}{\widehat{\Theta}}
\newcommand{\wttheta}{\widetilde{\Theta}}

\newcommand{\ot}{\otimes}

\newcommand{\E}{\mathbb{E}}
\newcommand{\eps}{\varepsilon}
\newcommand{\cN}{\mathcal{N}}

\newcommand{\smallSym}{S}
\newcommand{\SPD}{\mathbb{P}}
\newcommand{\samp}{x}
\newcommand{\rv}{x}
\newcommand{\ef}{f}

\newcommand{\dFR}{d_{\operatorname{FR}}}
\newcommand{\DF}{D_{\operatorname{F}}}
\newcommand{\Dop}{D_{\operatorname{op}}}
\newcommand{\DKL}{D_{\operatorname{KL}}}
\newcommand{\DTV}{D_{\operatorname{TV}}}
\newcommand{\dop}{d_{\operatorname{op}}}

\def\dmin{d_{\min}}
\def\dmax{d_{\max}}

\ifdefined\ARXIV
\newcommand{\CREFsupp}[2]{\cref{#1}}
\else
\newcommand{\CREFsupp}[2]{\cite[#2]{FORW25supp}}
\fi

\newcommand{\DRAFT}{} 

\ifdefined\DRAFT

\fi

\endlocaldefs

\begin{document}

\begin{frontmatter}
\title{Near optimal sample complexity for matrix and tensor normal models via geodesic convexity}
\runtitle{Near optimal sample complexity for matrix and tensor normal models}

\begin{aug}
\author[A]{\fnms{Cole}~\snm{Franks}\ead[label=e1]{franks@mit.edu}},
\author[B]{\fnms{Rafael}~\snm{Oliveira}\ead[label=e2,mark]{rafael@uwaterloo.ca}},
\author[C]{\fnms{Akshay}~\snm{Ramachandran}\ead[label=e3,mark]{aramach@cs.ubc.ca}}
\and
\author[D]{\fnms{Michael}~\snm{Walter}\ead[label=e4]{michael.walter@lmu.de}}
\address[A]{Department of Mathematics, Massachusetts Institute of Technology, \printead{e1}}
\address[B]{Cheriton School of Computer Science, University of Waterloo, \printead{e2}}
\address[C]{Computer Science Department, University of British Columbia, \printead{e3}}
\address[D]{Ludwig-Maximilians-Universit\"at M\"unchen, \printead{e4}}
\runauthor{C.\ Franks, R.\ Oliveira, A.\ Ramachandran, M.\ Walter}
\end{aug}

\begin{abstract}
The matrix normal model, i.e., the family of Gaussian matrix-variate distributions whose covariance matrices are the Kronecker product of two lower dimensional factors, is frequently used to model matrix-variate data.
The tensor normal model generalizes this family to Kronecker products of three or more factors.
We study the estimation of the Kronecker factors of the covariance matrix in the matrix and tensor normal models.

For the above models, we show that the maximum likelihood estimator (MLE) achieves \emph{nearly optimal nonasymptotic sample complexity} and \emph{nearly tight error rates} in the Fisher-Rao and Thompson metrics.
In contrast to prior work, our results do not rely on the factors being well-conditioned or sparse, nor do we need to assume an accurate enough initial guess.
For the matrix normal model, all our bounds are minimax optimal up to logarithmic factors, and for the tensor normal model our bounds for the largest factor and for overall covariance matrix are minimax optimal up to constant factors provided there are enough samples for any estimator to obtain constant Frobenius error.
In the same regimes as our sample complexity bounds, we show that the flip-flop algorithm, a practical and widely used iterative procedure to compute the MLE, converges linearly with high probability.

Our main technical insight is that, given enough samples, the negative log-likelihood function is \emph{strongly geodesically convex} in the geometry on positive-definite matrices induced by the Fisher information metric.
This strong convexity is determined by the expansion of certain random quantum channels.
\end{abstract}

\begin{keyword}[class=MSC2020]
\kwd[Primary ]{62F12}
\kwd[; secondary ]{62F30}
\end{keyword}

\begin{keyword}
\kwd{Covariance estimation}
\kwd{matrix normal model}
\kwd{tensor normal model}
\kwd{maximum likelihood estimation}
\kwd{geodesic convexity}
\kwd{operator scaling}
\kwd{quantum expansion}
\end{keyword}

\end{frontmatter}
\ifdefined\ARXIV\tableofcontents\fi

\section{Introduction}\label{section: intro}

Covariance matrix estimation is an important task in statistics, machine learning, and the empirical sciences.
We consider covariance estimation for centered matrix-variate and tensor-variate Gaussian data, that is, when individual data points are matrices or tensors.
Matrix and tensor-variate data arise naturally in numerous applications, such as gene microarrays, clinical trials, spatio-temporal data, signal processing and brain imaging (see \cite{mardia1993spatial,brown2001bayesian, mitchell2004learning, werner2008estimation} and references therein).
A significant challenge in this setting is that the dimensionality of these problems is much higher than the number of samples, making estimation information-theoretically impossible without structural assumptions.

To remedy this issue, matrix-variate data is commonly assumed to follow the \emph{matrix normal distribution} \citep{mardia1993spatial,dutilleul1999mle,werner2008estimation}.
Here the matrix follows a multivariate Gaussian distribution and the covariance between any two entries in the matrix is a product of an inter-row factor and an inter-column factor.
In spatio-temporal statistics this is referred to as a separable covariance structure \cite{mardia1993spatial}.
Formally, if a matrix normal random variable~$X$ takes values in the space of~$d_1\times d_2$ matrices, then its covariance matrix $\Sigma$ is a $d_1d_2\times d_1 d_2$ matrix that is the Kronecker product~$\Sigma_1 \ot \Sigma_2$ of two positive-semidefinite matrices~$\Sigma_1$ and~$\Sigma_2$ of dimensions~$d_1\times d_1$ and~$d_2\times d_2$, respectively.
This naturally extends to the \emph{tensor normal model}, where $X$ is a $k$-dimensional array, with covariance matrix equal to the Kronecker product of $k$ many positive semidefinite matrices~$\Sigma_1, \dots, \Sigma_k$.
Hence, a centered tensor normal distribution is denoted by $\cN(0, \Sigma_1 \otimes \cdots \otimes \Sigma_k)$.
In this work, we study the estimation of the covariance factors $\Sigma_1, \dots, \Sigma_k$ or (equivalently) the precision factors $\Theta_{1} := \Sigma_{1}^{-1}, \dots, \Theta_{k} := \Sigma_{k}^{-1}$ from $n$ samples of $\cN(0, \Sigma_1 \otimes \cdots \otimes \Sigma_k)$.
We emphasize that the goal is to estimate \emph{each of the factors}, rather than estimating the overall product $\Theta = \Theta_{1} \otimes ... \otimes \Theta_{k}$ or $\Sigma := \Sigma_{1} \otimes ... \otimes \Sigma_{k}$
by an arbitrary precision or covariance matrix (that may not be of tensor product form).

This problem falls into the field of estimation theory: for a family $\mathcal{P} := \{p_{\Theta}\}_{\Theta \in \P}$ of distributions with parameter space $\P$, given samples from an unknown distribution $X_{1}, ..., X_{n} \sim p_{\Theta}$, compute an estimate $\hat{\Theta} \approx \Theta$ of the true parameter value.
The quality of our estimate depends on some \emph{error measure}, chosen based on the downstream application of the estimation problem.
Our parameter space $\P$ is the set of Kronecker products of $k$ precision matrices, each of dimension $d_i$, which will be taken from the space of positive definite matrices (denoted~$\PD(d_i)$).

\phantomsection\label{phantom:distance}
The error measures in our work will be given by the \emph{Fisher-Rao} and \emph{Thompson} metrics.
These are the relevant error metrics for statistical applications, as they are intimately tied to error measures for the corresponding distributions, such as total variation and relative entropy.
Further theoretical justification is given by Chentsov's Theorem~\cite[Theorem 3]{cencov1978algebraic}, which states that
for smooth parameter manifolds,
the Fisher information metric\footnote{The Fisher-Rao distance is the distance function arising from the Fisher information metric.} is the unique Riemannian metric that preserves all relevant information with respect to parameter estimation.
We refer the reader to \CREFsupp{app:rel-error}{Section~A} in the supplement for further details on these metrics, as well as their connection to other natural metrics used for the matrix and tensor normal models.

\begin{definition}[Fisher-Rao and Thompson distances]\label{definition: fisher-rao and thompson}
The Fisher-Rao distance for centered Gaussians parameterized by their precision matrices is given by
\begin{align}\label{eq:fisher rao}
  \dFR(\htheta, \Theta) = \frac{1}{\sqrt{2}} \norm{\log (\Theta^{-1/2} \htheta \Theta^{-1/2})}_F.
\end{align}
The Thompson distance is given by
\begin{align}\label{eq:def d_op}
  \dop(\htheta, \Theta) := \norm{\log (\Theta^{-1/2} \htheta \Theta^{-1/2})}_{\op}.
\end{align}
\end{definition}

We have the following simple relation between the two metrics, that follows directly from the same relation between the operator and Frobenius norms.

\begin{fact} \label{f:dopvsdFR}
For $A, B$ positive definite matrices of dimension $d$, i.e. $A, B \in \PD(d)$, the Fisher-Rao and Thompson metrics are related by
    \[ \dop(A,B) \leq \sqrt{2} \cdot \dFR(A,B) \leq \sqrt{d} \cdot \dop(A,B) .    \]
\end{fact}

Now that we are equipped with our error measures, we can formally ask the foundational questions on the parameter estimation problem for the tensor normal model.\footnote{Since the matrix normal model is a special case of the tensor normal model (when $k = 2$), we will refer to our model as the tensor normal model whenever we treat the general case.}
We begin with the sample complexity questions.

\begin{problem}[Sample complexity upper bound]\label{problem: sample complexity upper bound}
Let $\varepsilon > 0$ be an error parameter and $\delta \in (0,1)$ be a failure parameter.
Given sample access to an unknown tensor normal distribution $\cN(0, \Theta_1^{-1} \otimes \cdots \otimes \Theta_k^{-1})$, how many samples $n(\eps,\delta)$ are \textbf{\emph{sufficient}} for the existence of estimator $\htheta_a$ satisfying, with probability $1 - \delta$,
\[ \dFR(\htheta_a, \Theta_a) \leq \varepsilon, \text{ for all } a \in [k]? \]
\end{problem}

In practical settings often the number of samples $n$ is fixed, so many results in the literature give bounds on the error $\eps$ and failure probability $\delta$ for fixed value of $n$.
The first consideration for such a result is its \textbf{\emph{sample threshold}}: this is the number of samples $n_{0}$ that is required in order for the proposed estimator to give any non-trivial guarantees, i.e. better than an arbitrary guess in $\P$.
The second consideration is the \textbf{\emph{error rate}} achieved by the proposed estimator, that is, how fast the error decreases as the number of samples grows.

\Cref{problem: sample complexity upper bound} is only concerned with \emph{upper bounds} on the number of samples needed to obtain good enough estimates for the true precision factors.
It is natural to ask what is the \emph{optimal} upper bound on the number of samples, that is, the minimum number of samples required to estimate the precision factors.
This leads us to the following problem:

\begin{problem}[Sample complexity lower bound]\label{problem: sample complexity lower bound}
    Let $\varepsilon > 0$ be an error parameter and $\delta \in (0,1)$ be a failure parameter.
    How many samples from a distribution $\cN(0, \Theta_1^{-1} \otimes \cdots \otimes \Theta_k^{-1})$ are \textbf{\emph{necessary}} for existence of estimator $\htheta_a$ such that, with probability $1 - \delta$
    \[ \dFR(\htheta_a, \Theta_a) \leq \varepsilon, \text{ for all } a \in [k]? \]
\end{problem}

\begin{remark} \label{rem:Minimax}
The above notion of sample complexity lower bound can be used to derive a minimax lower bound as follows: if $n \geq n(\eps,\delta)$ samples are required to achieve $\dFR(\hat{\Theta}_{a}, \Theta_{a}) \leq \eps$ error with probability at least $1-\delta$, then given $n < n(\eps,\delta)$ samples,
\[ \inf_{\hat{\Theta}} \sup_{\Theta \in \P} \E \left[ \max_{i \in [k]} \dFR(\htheta_a, \Theta_a) \right] \geq \delta \cdot \eps ,   \]
where $\inf$ is over all possible estimators~$\hat{\Theta}$, the $\sup$ is over the parameter space $\P$, and the expectation is over the distribution corresponding to parameter $\Theta$.
\end{remark}

A complete solution to the sample complexity problem requires one to prove tight \emph{upper} and \emph{lower} bounds on the number of samples to estimate the factors of the covariance matrix for a given error and probability guarantee.

The above questions are concerned with the mathematical existence of an estimator with a prescribed number of samples which accurately estimates the true precision factors.
However, a more relevant question for practical purposes is whether the estimator proposed for \cref{problem: sample complexity upper bound} can be computed efficiently.
More succinctly, one can ask whether there is a gap between statistical estimation versus computational estimation.
This is captured by the following computational variant of \cref{problem: sample complexity upper bound}

\begin{problem}\label{problem: computational estimation}
Let $\varepsilon > 0$ be an error parameter and $\delta \in (0,1)$ be a failure parameter.
Given sample access to an unknown tensor normal distribution $\cN(0, \Theta_1^{-1} \otimes \cdots \otimes \Theta_k^{-1})$, how many samples from the above distribution are \emph{sufficient} for there to exist estimators $\htheta_a$ that are \textbf{\emph{efficiently computable}} and satisfy, with probability $1 - \delta$
\[ \dFR(\htheta_a, \Theta_a) \leq \varepsilon, \text{ for all } a \in [k]? \]
Moreover, give an algorithm to compute this estimator which runs in polynomial time and achieves the above error bounds and success probability.
\end{problem}

\noindent This work fully addresses the three problems above for the matrix and tensor normal models.

Our solution to \cref{problem: sample complexity upper bound} comes from the analysis of the most natural candidate: the \emph{maximum likelihood estimator} (MLE).
Informally, we give the following sample complexity bounds for this estimator:


\begin{theorem*}[Sample complexity, tensor normal model]
    Let $\cN(0, \Theta_1^{-1} \otimes \cdots \otimes \Theta_k^{-1})$ be a tensor normal distribution with $k \geq 2$, where each $\Theta_i$ is a positive definite matrix of dimension~$d_i$, and let $D := \prod_{i=1}^k d_i$.
    Given a number of samples $n$ respecting the sample threshold $n \gtrsim \frac{k^{2} d_{\max}^{3}}{D}$, the MLE achieves minimax optimal error rate in Fisher-Rao distance
    \[ \dFR(\hat{\Theta}, \Theta) \lesssim \sqrt{ \frac{k d_{\max}^{2}}{n}} , \qquad  \dFR(\hat{\Theta}_{a}, \Theta_{a}) \lesssim \sqrt{ \frac{k d_{a} d_{\max}^{2}}{n D}}   \]
    with high probability.
    Further, for the matrix normal model (i.e., $k=2$), the sample threshold is improved to $n \gtrsim \frac{d_{\max}^{2} \log^{2} d_{\min}}{D}$, and the error can be bounded in the Thompson metric as
    \[ \dop(\hat{\Theta}_{a}, \Theta_{a}) \lesssim \sqrt{ \frac{d_{a}^{2} \log^{2} d_{\min}}{n D}}  .  \]
\end{theorem*}

Our estimation guarantees are \emph{distribution independent}, in particular the above bounds hold regardless of condition number or sparsity or other properties of the true precision matrix. This means that they apply to the most general model where the precision factors are allowed to be arbitrary positive definite matrices with no restrictions.

By \cref{f:dopvsdFR}, the $\dop$ bound for the matrix normal model recovers the $\dFR$ error rate for the tensor normal bound up to logarithmic factors; furthermore, it implies strong estimation guarantees in the operator norm, which are useful in spectral applications (see \cite{BL08}).

The above guarantees are \emph{tight} compared to classical lower bounds (see \cref{cor:relative-lower}), matching the sample complexity lower bounds even for the simpler $k=1$ setting.
The $\dFR$ error rate for the full precision matrix as well as the largest tensor factor are tight up to the factor $\sqrt{k}$.
And the sample threshold matches the lower bound for estimating the largest tensor factor up to a single $d_{\max}$ factor.
In the $k=2$ matrix normal model, the error rate is tight in the more refined $\dop$ metric, matching the classical lower bound for estimating a single tensor up to $\log$ factors.
The sample threshold matches the classical lower bound up to $\log$ factors.

We solve \cref{problem: sample complexity lower bound} by extending the lower bound for the unstructured Gaussian estimation problem to the matrix and tensor normal model.
While the above results are near-optimal for estimation of the largest tensor factor (via the classical lower bound), one could hope for better results for the smaller tensor factors,\footnote{In certain applications, such as brain fMRI, one is interested only in the smaller factor, whereas the larger factor is treated as a nuisance parameter.} as they intuitively receive more information from each tensor data.
Our next contribution is a stronger sample complexity lower bound which shows this is not the case.

\begin{theorem}[Lower bound for matrix normal models]\label{thm:matrix-lower-intro}
Let $\htheta_1$ be any estimator for~$\Theta_1$ given $n$ samples $X_{1}, ..., X_{n} \sim \cN(0, \Theta_1^{-1} \ot \Theta_2^{-1})$.
For $d_1 \leq d_2$, there exist $\Theta_{1} \in \PD(d_{1})$ and~$\Theta_{2} \in \PD(d_{2})$ such that
\[ \dFR(\htheta_{1}, \Theta_{1})  \gtrsim \sqrt{ \frac{d_{1}^{2}}{n \cdot \min\{ nd_{1}, d_{2}\}}} , \qquad
\dop(\htheta_{1}, \Theta_{1}) \gtrsim \sqrt{ \frac{d_{1}}{n \cdot \min\{ nd_{1}, d_{2}\}}}  \]
with constant probability.
\end{theorem}

When $nd_{1} \ll d_{2}$, our lower bound is significantly stronger than the classical lower bound for estimating~$\Theta_1$ \emph{assuming~$\Theta_2$ is known}, namely~$\sqrt{d_{1}^{2}/n d_{2}}$ for~$\dFR$ and $\sqrt{d_{1}/n d_{2}}$ for~$\dop$.
Our result generalizes naturally to the tensor normal model, as we discuss further in \cref{sec:lower}.
This implies that the matrix and tensor estimation problems are strictly harder than separate instances of the classical Gaussian estimation problem.
We are also able to show that a simple modification of the MLE obtains a matching upper bound for the matrix normal model.

Lastly, our solution to \cref{problem: computational estimation} comes from analyzing the \emph{flip-flop algorithm} to compute the MLE.
This is the \emph{first rigorous convergence analysis} of the flip-flop algorithm, which was proposed in the independent works~\cite{mardia1993spatial,dutilleul1999mle, brown2001bayesian} and is widely used in practice.

\begin{theorem*}[Computational estimation, informal]
    With high probability, the MLE can be computed efficiently. Namely, the Flip-Flop algorithm enjoys exponential convergence rate $\log(1/\delta)$ to achieve a $\delta$ approximation to the MLE.
\end{theorem*}

For a full comparison and relation between our results above and previous works, we refer the reader to \CREFsupp{appendix: previous works}{Section~B} in the supplement.

\paragraph*{Technical Contributions and Overview}
We now discuss the main conceptual ideas and principles behind our results.
In the matrix and tensor normal models (i.e. $k \geq 2$ case), the MLE is a solution to an explicit optimization problem over the space of tensor products of positive definite matrices, which we denote by $\P$.
When we endow the parameter space $\P$ with a natural Riemannian metric induced by the Fisher information, the negative log-likelihood becomes a \emph{geodesically convex} function of the parameter space (first observed in \cite{wiesel2012geodesic}).
In this work, we use \textit{geodesic convexity} of the negative log-likelihood function to show that the MLE for the tensor normal model indeed recovers all the benefits of the unstructured Gaussian setting ($k=1$).
Our strategy, as we outline in \cref{subsec:proof-sketch}, proceeds as follows: provided one is given enough samples, we prove that the negative log-likelihood function is \emph{strongly geodesically convex}, and the gradient at the true precision matrix is small.
With these two facts, we are able to conclude our bounds via a generalization of the usual argument that with a strongly convex function, any point with a small enough gradient (in our case the true precision matrix) is close to the optimizer (the MLE).

The global geodesic perspective is also key when analyzing algorithms to compute the MLE.
Inspired by recent research in computer science~\cite{GGOW19,burgisser2017alternating,burgisser2018efficient,burgisser2019towards}, we view the flip-flop algorithm as a natural geodesic extension of the block-coordinate geodesic gradient descent method, which is a standard convex optimization method.
Once we establish strong geodesic convexity of the negative log-likelihood function, we can show that the iterates of the Flip-Flop algorithm converge exponentially quickly to the MLE once the gradient of our current guess is sufficiently small.
Our proof generalizes to any descent method with reasonable guarantees.

This geodesic geometry perspective induces a natural error metric under which our analysis becomes linearly-invariant, and this allows us to prove sample complexity and error bounds that are independent of condition number.
Furthermore, by using global geodesic convexity of the negative log-likelihood function, we are able to decouple our analysis of the estimator from our algorithm to compute the MLE, and therefore we are able to remove the initial guess assumption from our error bounds.
The bounds we achieve are tight in general, as we show in \cref{sec:lower}, and our bounds even improve upon the previous results in the sparse setting as soon as the condition number or initialization error becomes moderately large (square root of the maximum dimension of the Kronecker factors).
For detailed comparison of our bounds with prior work, we point the reader to \CREFsupp{appendix: previous works}{Section~B} in the supplement.

We believe that the strength of the derived bounds, along with the principled analysis of a very simple and practical algorithm, make strong arguments in favor of the geodesic perspective for understanding the tensor normal model.
We now present the formal definitions of our problems and state our main results.

\subsection{Formal definitions and our results}\label{subsection: formal definitions}

We write $\Mat(d)$ for the space of real $d\times d$ matrices and $\PD(d)$ for the convex cone of $d\times d$ real symmetric positive definite matrices; $\GL(d)$ denotes the group of real invertible $d\times d$ matrices.
We write~$\succeq$ for the L\"{o}wner order.
For matrices~$A$ and~$B$, $\norm{A}_{\op}$ denotes the operator norm, $\norm{A}_F = (\tr A^T A)^{\frac12}$ the Frobenius norm, and $\braket{A,B} = \tr A^T B$ the Hilbert-Schmidt inner product.
We say $A$ is a traceless matrix if $\tr A = 0$.
We denote by $\kappa(A) = \norm{A}_{\op} \norm{A^{-1}}_{\op}$ the condition number of~$A$.
For functions~$f,g\colon S \to \R$ on any set $S$, we say $f = O(g)$ if there is a constant $C > 0$ such that~$f(x) \leq C g(x)$ for all~$x \in S$, and similarly $f = \Omega(g)$ if there is a constant $c > 0$ such that~$f(x) \geq c g(x)$ for all $x \in S$.
If~$f = O(g)$ and~$g = O(f)$ we write $f = \Theta(g)$.
In case~$C,c$ depend on another parameter $\lambda$, we write $O_\lambda$ and $\Omega_\lambda$, respectively.
We abbreviate~$[k]=\{1,\dots,k\}$ for $k\in\N$.
All other notation is introduced in the remainder of the text as needed.

We can now formally define the tensor normal model, of which the matrix normal model is a particular case.

\begin{definition}
For dimensions $d_{1}, \dots, d_{k} \in \mathbb{N}$, the \emph{tensor normal model} is the family of centered multivariate Gaussian distributions with covariance matrix given by a Kronecker product $\Sigma = \Sigma_1 \ot \dots \ot \Sigma_k$ of positive definite matrices, with $\Sigma_{a} \in \PD(d_{a})$, $a \in [k]$, that is, the distributions $\cN(0, \Sigma_1 \ot \dots \ot \Sigma_k)$.
For $k=2$, this is known as the \emph{matrix normal model}.
\end{definition}

Note that each $\Sigma_a$ is a $d_a\times d_a$ matrix and $\Sigma$ is a $D\times D$-matrix, where $D=d_1 \cdots d_k$.
Our goal is to estimate $k$ Kronecker factors $\widehat{\Sigma}_1, \dots, \widehat{\Sigma}_k$ such that $\widehat{\Sigma}_{a} \approx \Sigma_{a}$ for each $a \in [k]$ given access to $n$ i.i.d.\ random samples $x_1, \dots, x_n \in \R^D$ drawn from the model. A weaker requirement is to only approximate the full covariance, that is, $\widehat{\Sigma}_1 \ot \cdots \ot \widehat{\Sigma}_k \approx \Sigma$.

One may also think of each random sample $x_j$ as taking values in the set of $d_1 \times \dots \times d_k$ arrays of real numbers.
There are $k$ natural ways to ``flatten'' $x_j$ to a matrix:
for example, we may think of it as a matrix with $d_1$ rows and $D/d_{1}$ columns, where a column is indexed by a tuple $(i_2 \in [d_{2}],\dots, i_k \in [d_{k}])$ and given by the vector in $\R^{d_1}$ with $i_1^{\text{st}}$ entry equal to~$(x_j)_{i_1, \dots, i_k}$.
In the tensor normal model, the $d_2d_3\cdots{}d_k$ many columns are each distributed as a Gaussian random vector with covariance proportional to~$\Sigma_1$.
In an analogous way we may flatten it to a $d_2 \times d_1d_3\cdots{}d_k$ matrix, and so on.
As such, the columns of the $a^{\text{th}}$ flattening can be used to estimate~$\Sigma_a$ up to a scalar.
However, doing so na\"ively (e.g., using the sample covariance matrix of the columns) can result in an estimator with very high variance.
This is because the columns of the flattenings are not independent.
In fact they may be so highly correlated that they effectively constitute only one random sample rather than $d_2\cdots{}d_k$ many.
The MLE attempts to decorrelate the columns to obtain rates such as those one would obtain if the columns were independent.

The MLE is easier to describe in terms of the precision matrices, which we now define.

\begin{definition}[Precision matrices] \label{def:precisionMatrices}
For a $D\times D$-covariance matrix~$\Sigma$ arising in the tensor normal model, we refer to $\Theta = \Sigma^{-1}$ as the \emph{precision matrix}.
We also define the \emph{Kronecker factor precision matrices} $\Theta_1, \dots, \Theta_k$ as the unique positive-definite matrices such that $\Theta = \Theta_1 \ot \cdots \ot \Theta_k$ and $(\det \Theta_a)^{1/d_a}$ is the same for each $a \in [k]$.
In other words, we choose $\Theta_a = \lambda \Theta'_a$ where $\det \Theta'_a = 1$ and $\lambda>0$ is a constant (not depending on $a\in [k]$). We make this choice because the Kronecker factors of $\Theta$ are determined only up to a scalar.
\end{definition}

Let~$\P$ denote the parameter space of all precision matrices $\Theta$ for the tensor normal model with fixed dimensions $d_1,\dots,d_k$, i.e.,
\begin{align*}
  \P &= \bigl\{ \Theta = \Theta_1 \ot \cdots \ot \Theta_k \;:\; \Theta_a \in \PD(d_a) \bigr\}.
\end{align*}

Given a tuple $x$ of samples $\samp_1,\dots,\samp_n\in\R^D$, the following function $f_x : \P \to \R$ is proportional to the negative log-likelihood:
\begin{align}\label{eq:neg log likelihood}
  \ef_\samp(\Theta)
=  \frac{1}{nD}\sum_{i = 1}^n \samp_i^T \Theta \samp_i -  \frac{1}{D}\log\det\Theta.
\end{align}
The \emph{maximum likelihood estimator (MLE)} for $\Theta$ is then defined as
\begin{align}\label{eq:mle}
  \widehat{\Theta} := \underset{\Theta \in \P}{ \argmin} f_x(\Theta)
\end{align}
whenever the minimizer exists and is unique.
We write $\widehat\Theta = \widehat\Theta(x)$ when we want to emphasize the dependence of the MLE on the samples~$x$, and we say $(\htheta_1, \dots, \htheta_k)$ is \emph{an} MLE for~$(\Theta_1, \dots, \Theta_k)$ if $\otimes_{a = 1}^k \htheta_a = \htheta$.
Note that $\P$ is \emph{not} a convex domain under the Euclidean geometry on the $D\times D$ matrices.

To state our results, and throughout this paper, we write $\dmin = \min_{1 \leq a \leq k} d_a$, $\dmax = \max_{1 \leq a \leq k} d_a$, and $D = \prod_{i=1}^k d_a$. Recall that we identify factors $\Theta_1,\dots,\Theta_k$ from $\Theta$ using the convention $\det\Theta_1^{1/d_1}=\dots=\det\Theta_k^{1/d_k}$, and likewise for the MLE $\htheta$.

\subsection{Results on sample complexity \& error bounds}
We begin with our result on the sample complexity for the tensor normal model.

\newcommand{\TensorFrob}{%
There are universal constants~$C, c_1, c_2>0$ such that the following holds.
Suppose that $t\geq 1$ and
\begin{equation}\label{eq:eps sqr assm}
  n \geq C k^2 \frac{\dmax^3}{D} t^2.
\end{equation}
Then, with probability at least
$
  1 - k e^{-c_1 t^2 \dmax } - k^2 \left( \frac{\sqrt{nD}}{k \dmax} \right)^{-c_2 \dmin},
$
the MLE~$\htheta$ for $n$ independent samples of the tensor normal model with precision matrix~$\Theta$ is unique and satisfies
\begin{equation*}
  \dFR(\htheta,\Theta) = O\left(\frac{\sqrt{k} \, \dmax }{\sqrt{n}} t\right) \quad \text{ and } \quad
    \dFR(\htheta_a,\Theta_a) = O\left(\frac{\sqrt{kd_a} \, \dmax }{\sqrt{nD}} t\right), \; \text{ for all } a \in [k].
\end{equation*}
}

\begin{theorem}[Tensor normal model sample complexity upper bounds]\label{thm:tensor-frobenius}
\TensorFrob
\end{theorem}

Our error guarantees are tight for both the full precision matrix and the largest factor, as they match the lower bound for the simpler Gaussian estimation problem described in \cref{cor:relative-lower} up to the factor $\sqrt{k}$.
Also note that the parameter $t$ allows a trade-off between error guarantees and probabilistic guarantees. In particular, choosing $t^{2} \approx \log n$ guarantees vanishing failure probability as $n \to \infty$.

For the matrix normal model $k=2$, we obtain a stronger result:\footnote{The key technical tool we use for our matrix normal model result is a sophisticated analysis of operator scaling from \cite{KLR19}. In order to lift this to the tensor normal model, we would need a similar analysis of the tensor scaling problem. This is significantly more difficult, as is discussed in more detail in e.g. \cite{burgisser2018efficient}.}
firstly, we improve the \emph{sample threshold} by a polynomial factor; secondly, we are able to bound the \emph{error rate} for the individual factors in the tighter Thompson metric;
and finally we improve the dependence on the \emph{failure probability} from polynomial to exponential.
Recall that we identify~$\Theta_1, \Theta_2$ from~$\Theta$ using the convention $\det\Theta_1^{1/d_1}=\det\Theta_2^{1/d_2}$.

\newcommand{\MatrixSpec}{%
There are universal constants~$c, C>0$ with the following property.
Suppose $t\geq 1$ and
\begin{align*}
  n \geq C \frac{\dmax}{\dmin} \max\{\log \dmax, t^2 \log^2 \dmin \}.
\end{align*}
Then the MLE $\htheta = \htheta_1 \ot \htheta_2$ \ for $n$ independent samples from the matrix normal model with precision matrix $\Theta = \Theta_1 \ot \Theta_2$ satisfies
\begin{align*}
  \dop(\widehat{\Theta}_1, \Theta_1) = O\left(t \sqrt{\frac{d_1}{nd_2}} \log \dmin \right)
\quad\text{and}\quad
\dop(\widehat{\Theta}_2, \Theta_2) = O\left(t \sqrt{\frac{d_2}{nd_1}} \log \dmin \right)
\end{align*}
with probability at least  $1 - e^{ - c \dmin t^2}$.}

\begin{theorem}[Matrix normal model sample complexity upper bounds]\label{thm:matrix-normal}
\MatrixSpec
\end{theorem}

We again note that the parameter $t$ allows for a trade-off between error and probabilistic guarantees, so in particular we can achieve vanishing failure probability as $n \to \infty$ by choosing e.g. $t^{2} \approx \log n$.
Further, we emphasize that the above error guarantees are tight for \emph{both tensor factors}, matching the classical Gaussian lower bound in \cref{cor:relative-lower} for each individual tensor factor up to $\log d_{\min}$ factors.

Recalling \cref{f:dopvsdFR}, we see that this stronger $\dop$ guarantee recovers the optimal $\dFR$ error rate for the tensor normal model up to $\log d_{\min}$ factors.
Further, the sample threshold is also tight up to $\log d_{\min}$ factors, matching the known lower bound for Gaussian estimation.
Finally, the guarantee in the Thompson metric gives much stronger accuracy for spectral applications such as PCA (see e.g. \cite{BL08}).

In applications such as brain fMRI, one is interested only in $\Theta_1$, and $\Theta_2$ is treated as a nuisance parameter.
If the nuisance parameter $\Theta_2$ were known, we could compute $(I \ot \Theta_2^{1/2} )X$, which is distributed as $nd_2$ independent samples from a Gaussian with precision matrix $\Theta_1$.
In this case, one can estimate $\Theta_1$ in operator norm with an RMSE rate of $O(\sqrt{ d_1/ n d_2})$ no matter how large $d_2$ is.
One could hope that this rate holds for $\Theta_1$ even when $\Theta_2$ is not known.
In \cref{sec:lower} we show a new lower bound for the matrix normal model that implies this better rate cannot hold.
Thus, for $d_2 > n d_1$, it is impossible to estimate $\Theta_1$ as well as one could if $\Theta_2$ were known.
Note that in this regime there is no hope of recovering $\Theta_2$ even if $\Theta_1$ is known.
As the random variable $Y_i$ obtained by ignoring all but $d_2' \approx nd_1$ columns of each $X_i$ is distributed according to the matrix normal model with covariance matrix $\Sigma_1 \ot \Sigma_2'$ for some $\Sigma_2' \in \PD(d_2')$, the MLE for $Y$ obtains a matching upper bound.

\begin{corollary}[Estimating only $\Theta_1$] \label{cor:EstimateOnlyOne}
There is a universal constant~$C>0$ with the following property.
Let $\Theta_1 \in \PD(d_1), \Theta_2 \in \PD(d_2)$, $X$ be distributed according to $\cN(0, \Theta_1^{-1} \ot \Theta_2^{-1})$, and suppose that $1 < d_1 \leq d_2$ and $t\geq 1$.
Let $Y = (Y_1, \dots, Y_n)$ be the random variable obtained by removing all but
\begin{align*}
  d_2' = \min \left\{ d_2, \frac{nd_1}{C \max \{ \log n, t^2 \log^2 d_1 \}} \right\}
\end{align*}
columns of $X_i$ for each~$i \in [n]$.
Then the MLE $\htheta = \htheta_1 \ot \htheta_2$ for $Y$ satisfies
\begin{align*}
  \dop(\widehat{\Theta}_1, \Theta_1) = O\left(t \sqrt{\frac{d_1}{nd_2'}} \log d_1\right) ,
\end{align*}
with probability $1 - e^{ - \Omega( d_1 t^2)}$.
This rate is tight up to factors of $\log d_1$ and $t^2 \log^2 d_1$.
\end{corollary}

\Cref{table: main results intro} provides a high-level comparison of the above results and previous works. For clarity we consider the simplified setting where all dimensions of the Kronecker factors are equal to $d$, all precision matrices are sparse with row sparsity $r$ (which implies total sparsity~$s \leq rd$), and all condition numbers of precision factors are upper bounded by $\kappa$.
A detailed comparison with all relevant parameters can be found in \CREFsupp{table: main results,table: guesses effect}{Tables~B.3 and~B.4}.

\begin{table}
\caption{Worst-case sample requirements and error rates of estimators.}
\label{table: main results intro}
\resizebox{1\columnwidth}{!}{
\begin{tabular}{|p{2.3cm}|p{1.4cm}|c|c|c|}
\hline
Work & Setting & Sample Threshold & Error Rate (above sample threshold)  \\
 \hline
\cite[Theorem 3]{tsiligkaridis2013convergence}
    & \makecell[t]{general, \\ $k=2$}
    & $\max\left\{1, \tfrac{\kappa^{2}}{d} \right\} \kappa^{2} \min\{ \kappa, d \} d \log d$
    & $ \frac{\|\hat{\Theta}^{(3)} - \Theta\|_{F}}{\|\Theta_{a}\|_{\op}} \lesssim \kappa^{2} \sqrt{\frac{d^{2} \log d}{n}} $
    \\
    \cite[Theorem 4]{tsiligkaridis2013convergence}
    & \makecell[t]{$s \lesssim d$, \\ $k=2$}
    & $\displaystyle \max\left\{1, \tfrac{\kappa^{2}}{d} \right\} \kappa^{2} \min\{ \kappa, d \} \log d$
    & $\frac{\|\hat{\Theta} - \Theta\|_{F}}{\|\Theta\|_{\op}} \lesssim \kappa^{2} \sqrt{\frac{d \log d}{n}} $
    \\
    \cite[Theorem 3.1]{zhou2014gemini}
    & \makecell[t]{$k=2$, \\ $s \leq d^{2}$}
    & $\displaystyle \max\left\{1, \tfrac{\kappa^{2}}{d} \right\} \kappa^{2} \min\{ \kappa, d \} \frac{(s+1) \log d}{d}$
    & $\displaystyle\frac{\|\hat{\Theta} - \Theta\|_{\op} }{\|\Theta\|_{\op}} \lesssim \kappa^{2} \sqrt{\frac{(s+1) \log d}{n}}$
    \\
    \makecell{ \\ \cite{Lyu2020Tlasso}}
    & \makecell[t]{$k \geq 2$ \\ $s \leq d^2$}
    & \makecell{$\displaystyle k^{2} \left( \min\{ \kappa, d \} \right)^{k-1} \max\left\{1, \tfrac{\kappa^{2}}{d} \right\} \kappa^{2} \frac{(s + d) \log d}{d^{k-1}}$}
    & $\frac{ \|\hat{\Theta}_{a} - \Theta_{a}\|_{F}}{\|\Theta_{a}\|_{F}} \lesssim \kappa \sqrt{\frac{d (s + d) \log d}{nd^{k}}} $
    \\
    \cref{thm:matrix-normal}
    & \makecell[t]{general, \\ $k=2$}
    & $\log^{2} d$
    & $ \dop(\hat{\Theta}_{a}, \Theta_{a}) \lesssim \sqrt{\frac{\log^{2} d}{n}} $
    \\
    \cref{thm:tensor-frobenius}
    & \makecell[t]{general,\\ $k \geq 3$}
    & $ \frac{k^{2} d^{3}}{d^{k}}  $
    & $ \dFR(\hat{\Theta}_{a}, \Theta_{a}) \lesssim \sqrt{ \frac{k d^{3}}{n d^{k}}} $
    \\
    \hline
\end{tabular}
}
\end{table}

As can be seen from the table, our sample threshold and error rates are independent of condition number factors, and our error measures $\dFR$ and $\dop$ are tighter than those used in previous works, as can be seen in \CREFsupp{remark: relative error and fisher-rao and thompson}{Remark~A.4} and \CREFsupp{prop:absRelation}{Proposition~A.8}.
While prior works are able to give improved guarantees for sparse inputs, we note that they also have polynomial dependence on condition number.
This becomes significant even for moderate values of condition number (e.g. $\kappa = d^{2}$), and so our estimator gives improved gaurantees in the most general setting.

\subsection{Results on the flip-flop algorithm for MLE estimation}\label{subsec:flip-flop}
The MLEs for the matrix and tensor normal models can be computed by a natural iterative procedure that is known as the \emph{flip-flop algorithm}.
In \cref{alg:flip-flop matrix} below, we describe it for the matrix normal model ($k=2$), where the samples $\samp_i$ can be viewed as $d_1\times d_2$ matrices~$X_i$.
The general flip-flop algorithm is described in \cref{alg:flip-flop} in \cref{sec:flip-flop}.

\begin{Algorithm}
\begin{description}
\item[\hspace{.2cm}\textbf{Input}:] Samples $X = (X_1, \ldots, X_n)$, where $X_i \in \R^{d_1 \times d_2}$, initial guess $\widetilde{\Theta} \in \SPD$.
Parameters $T\in\N$ and $\delta>0$.

\item[\hspace{.2cm}\textbf{Output}:]
An estimate $\otheta = \otheta_1 \ot \otheta_2 \in \SPD$ of the MLE.

\item[\hspace{.2cm}\textbf{Algorithm}:]
\end{description}
\begin{enumerate}
\item
Set $\otheta_1 = \widetilde{\Theta}_1$ and $\otheta_2 = \widetilde{\Theta}_2$.

\vspace{10pt}

\item
For $t=1,\dots,T$, repeat the following:

\vspace{5pt}

\begin{itemize}
\item
If $t$ is odd, set $a=1$ and $\Upsilon = \frac1{nd_2}\sum_{i=1}^n X_i \otheta_2 X_i^T$.
If $t$ is even, set $a=2$ and $\Upsilon = \frac1{nd_1}\sum_{i=1}^n X_i^T \otheta_1 X_i$.
\item If $t>1$ and $\norm{\nabla_a f_x(\otheta)}_F \leq \delta$, return $\otheta$
\vspace{3pt}
\item Update $\otheta_a \leftarrow \Upsilon^{-1}$
\end{itemize}
\end{enumerate}
\caption{Flip-flop algorithm for the matrix normal model.}\label{alg:flip-flop matrix}
\end{Algorithm}

We can motivate the flip-flop algorithm by noting that if in the first step we already have $\overline{\Theta}_2 = \Theta_2$ (the true precision factor), then $\frac{1}{n d_2} \sum_{i = 1}^n X_i \overline{\Theta}_2 X_i^T$ is simply a sum of outer products of $nd_2$ many independent random vectors with covariance $\Sigma_1 = \Theta_1^{-1}$; as such the inverse of the sample covariance would be a good estimator for $\Theta_1$.
As we don't know $\Theta_2$, the flip-flop algorithm instead uses $\overline{\Theta}_2$ as our current best guess, with the hope that each iteration will improve the next guess.

For the general tensor normal model (\cref{alg:flip-flop}), in each step the flip flop algorithm chooses one of the dimensions $a \in [k]$ and uses the $a^\text{th}$ flattening of the samples~$x_i$ (which are just $X_i$ and $X_i^T$ in the matrix case) to update $\overline{\Theta}_a$.

The advantage of flip-flop over other estimators are twofold: it directly converges to the MLE, as opposed to regularization approaches that trade-off accuracy for speed; and it has small iteration complexity.
Each iteration of flip-flop is extremely fast to compute (one matrix inversion), whereas (most) other works have expensive complexity per iteration (solving a convex program).
See details in \CREFsupp{subsec:complexity-previous-work}{Section~B.5}.

Our next results show that the flip-flop algorithm can efficiently compute the MLE when the hypotheses of \cref{thm:tensor-frobenius} or \cref{thm:matrix-normal} hold.
We state our result for the tensor normal model and then give an improved version for the matrix normal model.

\newcommand{\TensorFlop}{There are universal constants~$C,c, c_1, c_2>0$ such that the following holds.
Suppose $x=(x_1,\dots,x_n)$ are $n \geq C k^2 \dmax^3 / D$ independent samples from $\cN(0, \Theta^{-1})$, where $\Theta = \Theta_1 \ot \cdots \ot \Theta_k$.
Then, with probability at least
\begin{align*}
  1 - k \, e^{-c_1 \frac{n D}{k^2\dmax^2}} - k^2 \left( \frac{\sqrt{nD}}{k \dmax} \right)^{-c_2 \dmin},
\end{align*}
the MLE $\htheta$ exists, and for any~$0<\delta<\frac{c}{\sqrt{(k+1)\dmax}}$, the number of iterations $T$ needed for \cref{alg:flip-flop} to output $\otheta$ with $\dFR(\otheta_a, \htheta_a) \leq \sqrt{2 d_a} \cdot \delta$ for all $a\in[k]$, is:

\begin{enumerate}
    \item when the initial guess is $\wttheta$ with $\nabla_0 f_x(\wttheta) = 0$,
    \begin{align*}
    T = O\left( k^2 \dmax \cdot \dop(\wttheta, \htheta) + k \log \frac1\delta \right)
    \end{align*}
    \item when the initial guess is $\wttheta$ with $\nabla_0 f_x(\wttheta) = 0$ and $\dop(\wttheta, \htheta) = O\left( \frac{1}{k \dmax} \right)$,
    \begin{align*}
    T = O\left( k \log \left( \frac{\sqrt{k \dmax} \cdot \dop(\wttheta, \htheta)}{\delta} \right) \right)
    = O\left( k \log \frac1\delta \right)
    \end{align*}
    \item without any initial guess (and starting from $\frac{1}{f_x(I_D)} \cdot I_D$),
    \begin{align*}
    T = O\left( k^2 \dmax \bigl( 1 + \log \kappa(\Theta) \bigr) + k \log \frac1\delta \right)
    \end{align*}
\end{enumerate}
}

\begin{theorem}[Tensor normal flip-flop]\label{thm:tensor-flipflop}
\TensorFlop
\end{theorem}

\newcommand{\MatrixFlop}{There are universal constants~$C,c, c_1>0$ such that the following holds.
Let $1 < d_1, d_2$.
Suppose $x_1,\dots,x_n \in \R^{d_1d_2}$ are
\begin{align*}
  n \geq C \frac{\dmax}{\dmin} \max \left\{ \log \dmax, \log^2 \dmin \right\}
\end{align*}
independent samples from $\cN(0, (\Theta_1 \ot \Theta_2)^{-1})$.
With probability at least $1 - \exp\left(-\frac{c_1 \cdot n \dmin^2}{\dmax \log^2 \dmin}\right)$, the MLE $\htheta$ exists, and for every $0 < \delta < \frac{c}{\sqrt{\dmax}}$, the number of iterations $T$ needed for \cref{alg:flip-flop matrix,alg:flip-flop} to output $\otheta$ with $\dFR(\otheta_a, \htheta_a) = O(\sqrt{d_a} \delta)$ for $a\in\{1,2\}$, is:

\begin{enumerate}
    \item when the initial guess is $\wttheta$ with $\nabla_0 f_x(\wttheta) = 0$,
    \begin{align*}
    T = O\left( \dmax \cdot \dop(\wttheta, \htheta) + \log \frac1\delta \right)
    \end{align*}
    \item when the initial guess is $\wttheta$ with $\nabla_0 f_x(\wttheta) = 0$ and $\dop(\wttheta, \htheta) = O\left( \frac{1}{\dmax} \right)$,
    \begin{align*}
    T = O\left( \log \left( \frac{\sqrt{\dmax} \cdot \dop(\wttheta, \htheta)}{\delta} \right) \right) = O\left( \log \frac1\delta \right)
    \end{align*}
    \item without any initial guess (and starting from $\frac{1}{f_x(I_D)} \cdot I_D$),
    \begin{align*}
      T = O\left( \dmax \bigl( 1 + \log \kappa(\Theta_1 \ot \Theta_2) \bigr) + \log \frac1\delta \right)
    \end{align*}
\end{enumerate}
}

\begin{theorem}[Matrix normal flip-flop]\label{thm:matrix-flipflop}
\MatrixFlop
\end{theorem}

Plugging in the error rates for the MLE from \cref{thm:tensor-frobenius,thm:matrix-normal} into \cref{thm:tensor-flipflop,thm:matrix-flipflop} (with $t = 1$) shows that the output of the flip-flop algorithm with $O\left(k^2\dmax(1 + \log \kappa(\Theta)) + k \log(n)\right)$ iterations is an efficiently computable estimator with the same statistical guarantees as we have shown for the MLE.

\Cref{table: estimator performance no assumptions intro} summarizes the iteration complexity of previous works and of the above theorems, in the most general setting where one is not given any assumptions about the initial guess.
We give a detailed comparison of performance in \CREFsupp{appendix: previous works}{Section~B}. Note that, while the number of iterations of the flip-flop algorithm is larger than in previous works, each iteration is much faster in our case (matrix inversion) than in previous works (which need to solve a convex program).
This justifies the better performance of flip-flop in practical settings.

\begin{table}
\caption{Performance of estimators without any assumptions, from initial guess $\wttheta$}
\label{table: estimator performance no assumptions intro}
\begin{tabular}{|p{3cm}|c|c|}
\hline
\makecell{Work} & Setting &  Main subroutine \\
\hline

\cite[Theorem 3]{tsiligkaridis2013convergence} & \makecell[t]{$k=2$, \\ general} & \makecell[t]{matrix inversion} \\

\hline

\cite[Theorem 4]{tsiligkaridis2013convergence} & \makecell[t]{$k=2$, \\ $s_a \lesssim d_a$} & convex program \\

\hline

\cite[Theorem 3.1]{zhou2014gemini} & \makecell{$k=2$, \\ general $s_a$} & convex program \\

\hline

\cite[Theorem 3.3]{zhou2014gemini} & \makecell{$k=2$, \\ $r_{s,a} \lesssim \sqrt{d_a}$} & linear program \\

\hline

\cite{XZG17} & \makecell{$k \geq 4$, \\ general $r_{s,a}$} & \makecell{truncated \\ gradient descent} \\

\hline

\cite{Lyu2020Tlasso} & \makecell{$k \geq 2$,\\ general $s_a$} & convex program \\

\hline

\cref{thm:matrix-flipflop,thm:tensor-flipflop} & $k \geq 2$ & matrix inversion  \\
\hline
\end{tabular}
\end{table}

A key contribution of this work is that our estimator, the MLE, is well-defined independent of any additional information.
In particular, we have decoupled our sample complexity analysis from the algorithmic analysis of our estimator.
Thus, our initial guess assumption only affects the runtime of the algorithm, and not the sample complexity.

In the above, we see that the iteration complexity of the flip-flop algorithm depends on the condition number of the precision matrix, when we do not have any assumption on the initial guess (case 3 in \cref{thm:tensor-flipflop,thm:matrix-flipflop}).
However, if we assume that we have an initial guess which is "close to the true precision matrix" (case 2) we show that \cref{alg:flip-flop,alg:flip-flop matrix} achieve much faster convergence to the MLE.
Note that we state in the above theorems that the initial guess is close enough to the MLE, but in the sample regimes of the above theorems, \cref{thm:matrix-normal,thm:tensor-frobenius} tell us that the MLE is very close to the true precision matrix.
This allows us to do a full comparison between the performance of flip-flop and other proposed estimators in several previously considered settings.
For details, see \CREFsupp{subsec:complexity-previous-work}{Section~B.5}.

\section{Geodesic convexity, sample complexity \& error bounds}\label{sec:sample-complexity}

We now explain how we use geodesic convexity, following a strategy similar to \cite{FM20}, to prove \cref{thm:tensor-frobenius}.
The detailed proofs of all results in this section can be found in \CREFsupp{app:tensor}{Section~D}.

\subsection{Geodesic convexity}\label{subsec:geom}
The negative log-likelihood for the tensor normal model, i.e. \cref{eq:mle}, is an optimization problem over the parameter space $\P$, which is a subset of the space $\PD(D)$ of positive-definite real symmetric $D\times D$ matrices.
As we have discussed in the previous section, we will consider the Riemannian metric on $\PD(D)$ that arises from the Fisher information metric on centered Gaussians parametrized by their covariance matrices \citep{skovgaard1984riemannian}.\footnote{This is the same as the metric arising from the Hessian of the log-determinant \citep[Chapter 6]{bhatia2009positive}.}
When we endow $\PD(D)$ with this metric, we see that the geodesics starting at a point $\Theta \in \PD(D)$ are of the form $t \mapsto \Theta^{1/2} e^{Ht} \Theta^{1/2}$ for~$t \in \R$ and a symmetric matrix~$H$.
Moreover, if $A$ is an invertible matrix, the transformation $\Theta \mapsto A\Theta A^T$ is an isometry with respect to this metric, i.e., it preserves the geodesic distance.
This invariance is natural and desirable, as changing a pair of precision matrices in this way does not change the statistical relationship between the corresponding Gaussians; in particular the total variation distance, Fisher-Rao distance, and Kullback-Leibler divergence are unchanged.

Another very useful property is that our domain $\P$ is a \emph{totally geodesic submanifold} of $\PD(D)$: for any two points $A,B \in \P$, the entire geodesic between $A$ and $B$ remains in our domain $\P$.
Thus, the negative log-likelihood is truly an optimization problem over the Riemannian manifold $\P$ under the Fisher information metric.

As $\P$ is a totally geodesic submanifold of $\PD(D)$, the invariance properties described above for $\PD(D)$ are directly inherited by $\P$.
The manifold~$\P$ carries a natural action by the group
\begin{align*}
  \G =  \{ A = A_1 \ot \cdots \ot A_k \;:\; A_a \in \GL(d_a) \}
\end{align*}
Namely, if $\Theta \in \P$ and $A \in \G$ then $A \Theta A^T \in \P$.
Thus, as discussed above, the map $\Theta \mapsto A\Theta A^T$ is an isometry of the Riemannian manifold~$\P$, thereby preserving statistical relationship between the corresponding Gaussians.

As observed by \cite{wiesel2012geodesic}, the negative log-likelihood function (\cref{eq:mle}) is convex when restricted to geodesics of the Fisher information metric.
In other words, the negative log-likelihood is \emph{geodesically convex} on our manifold~$\P$.
To see this fact, we will now formally describe the structure of the manifold $\P$ and define geodesic convexity.

In the manifold $\P$, the tangent space at any point $\Theta \in \P$ is given by
\begin{align*}
    \mathfrak{p} := \left\{ \sum_{i=1}^k I_{d_1} \otimes \cdots \otimes I_{d_{i-1}} \otimes \log(\Gamma_i) \otimes I_{d_{i+1}} \otimes \cdots I_{d_k} \quad \mid \quad \Theta^{1/2} \Gamma \Theta^{1/2} \in \P \right\}
\end{align*}
which can be identified with the real vector space
\begin{align*}
  \H &= \bigl\{ (H_0; H_1,\dots,H_k) \;:\; H_0 \in \R, \; H_a \text{ a symmetric traceless $d_a \times d_a$ matrix} \, \forall a \in [k] \bigr\},
\end{align*}
equipped with the following  inner product and norm:
\begin{align*}
  \braket{H,K} := H_0 K_0 + \sum_{a=1}^k \tr H_a^T K_a, \quad
  \norm{H}_F := \braket{H, H}^{1/2}.
\end{align*}

The direction $(1; 0, \dots, 0)$ changes $\Theta$ by an overall scalar, and tangent directions supported only in the $a^{th}$ component for $a \in [k]$ only change~$\Theta_a$ (subject to its determinant staying fixed).
In order to make this inner product agree with the natural Frobenius inner product on the tangent space $\mathfrak{p}$, we parametrize the exponential map as in the following definition.

\begin{definition}[Exponential map and geodesics]
The \emph{exponential map} $\exp_\Theta \colon \H \to \P$ at~$\Theta=\Theta_1\ot\cdots\ot\Theta_k\in\P$ is defined by
\begin{align*}
  \exp_\Theta(H)
&= e^{H_0} \cdot ( \Theta_1^{1/2} e^{\sqrt{d_1} H_1} \Theta_1^{1/2}) \ot \cdots \ot (\Theta_k^{1/2} e^{\sqrt{d_k} H_k} \Theta_k^{1/2}).
\end{align*}
By definition, the \emph{geodesics} through $\Theta$ are the curves $t \mapsto \exp_\Theta(t H)$ for $t\in\R$ and $H\in\H$.
Up to reparameterization, there is a unique geodesic between any two points of~$\P$.
\end{definition}

The geodesics on $\P$ defined above are simply the geodesics of the Fisher information metric on~$\PD(D)$, reparametrized in terms of the identification of the tangent space $\H$ given above.

We take the convention that the geodesics have unit speed if $\norm{H}_F^2 = 1$.
The geodesic distance $d(\Theta,\Theta')$ between two points $\Theta$ and $\Theta'=\exp_\Theta(H)$ is therefore equal to~$\norm H_F$, which can also be computed as~$D^{-1/2} \norm{\log \Theta^{-1/2} \Theta' \Theta^{-1/2}}_F$, which we will take to be our notion of geodesic distance.
To summarize:

\begin{definition}[Geodesic distance and balls]
The \emph{geodesic distance} $d(\Theta,\Theta')$ between two points $\Theta$ and $\Theta'$ of $\P$ is given by
\begin{align}\label{eq:geodesic distance}
  d(\Theta,\Theta') := \frac1{\sqrt D} \norm{\log \Theta^{-1/2} \Theta' \Theta^{-1/2}}_F = \sqrt{\frac{2}{D}} \cdot \dFR(\Theta, \Theta').
\end{align}
where $\log$ denotes the matrix logarithm and $\dFR$ is the Fisher-Rao distance defined in \cref{eq:fisher rao}.

The closed \emph{(geodesic) ball} of radius~$r>0$ about~$\Theta$ is defined as
\begin{align*}
  B_r(\Theta) = \bigl\{ \exp_\Theta(H) : H \in \H, \norm H_F \leq r \bigr\},
\end{align*}
\end{definition}

The manifold $\PD(D)$, and hence $\P$, is a \emph{Hadamard manifold}, i.e., a complete, simply connected Riemannian manifold of non-positive sectional curvature \citep{bacak2014convex}. Thus geodesic balls are \emph{geodesically convex} subsets of~$\P$, that is, if $\gamma(t)$ is a geodesic such that~$\gamma(0),\gamma(1) \in B_r(\Theta)$ then $\gamma(t) \in B_r(\Theta)$ for all $t\in[0,1]$.

The definition of geodesics yields the following notion of geodesic convexity of functions.

\begin{definition}[Geodesic convexity]
Given a geodesically convex domain~$\Gamma \subseteq \P$, a function $f$ is \emph{(strictly) geodesically convex} on~$\Gamma$ if, and only if, the function $t \mapsto f(\gamma(t))$ is (strictly) convex on~$[0,1]$ for any geodesic $\gamma(t)$ with $\gamma(0),\gamma(1)\in \Gamma$.

The function $f$ is $\lambda$-\emph{strongly} geodesically convex if $t \mapsto f(\gamma(t))$ is $\lambda$-strongly convex along every unit-speed geodesic $\gamma$ with endpoints in $\Gamma$.

For a twice differentiable function $f\colon \P \to \R$, we say that it is $\lambda$-strong geodesically convex \emph{at} $\Theta$ if $\partial^2_{t=0} f(\exp_\Theta(tH)) \geq \lambda \norm H_F^2$ for all~$H\in\H$, and we say it is $\lambda$-strong geodesically convex on $\Gamma$ if it is $\lambda$-strong geodesically convex for every $\Theta \in \Gamma$.
\end{definition}

\begin{example}\label{exa:usual-likelihood} It is instructive to consider the case $k = 1$, or $\P = \PD(D)$.
The geodesics through $\Theta$ are the curves $t \mapsto \sqrt{\Theta} e^{\sqrt{D} \cdot H t} \sqrt{\Theta}$ where $H\in\H$.
As an example of a geodesically convex function, consider the likelihood for the precision matrix of a Gaussian with data $x_1, \dots, x_n$.
Let $\rho := \frac1{nD}\sum_i x_i x_i^T$ denote the matrix of ``second sample moments'' of the data.
Then we can rewrite the objective function~\eqref{eq:neg log likelihood} as
\begin{align*}
  f_x(\Theta) = \tr \rho \, \Theta - \frac1D \log \det \Theta.
\end{align*}
We claim that $f_x(\Theta)$ is always geodesically convex, and in fact \emph{strictly} geodesically convex whenever $\rho$ is invertible. Indeed,
$$ \partial^2_{t = 0} f_x(\sqrt{\Theta} e^{\sqrt{D} \cdot t H} \sqrt{\Theta}) = D \cdot \tr \sqrt{\Theta} \rho \sqrt{\Theta} H^2 \geq 0 $$
with strict inequality whenever $\rho$ is invertible (and~$H$ nonzero).
\end{example}

The computation in the example easily generalizes to the tensor normal model, which allows us to prove geodesic convexity of the negative log-likelihood function in our setting.

We now formally define the \emph{Riemannian} gradient and Hessian.

\begin{definition}[Gradient and Hessian]\label{def:hess grad}
Let $f\colon \P \to \R$ be a differentiable function and $\Theta \in \P$.
The \emph{(Riemannian) gradient}~$\nabla f(\Theta)$ is the unique element in $\H$ such that
\begin{align*}
  \braket{\nabla f(\Theta), H}
= \partial_{t=0} f(\exp_\Theta(tH))
\qquad \forall H \in \H.
\end{align*}
If $f$ is twice-differentiable, the \emph{(Riemannian) Hessian}~$\nabla^2 f(\Theta)$ is the unique linear operator on~$\H$ such that
\begin{align*}
  \braket{H, \nabla^2 f(\Theta) K}
= \partial_{s=0} \partial_{t=0} f(\exp_{\Theta}(sH + tK))
\qquad \forall H,K \in \H.
\end{align*}
We abbreviate $\nabla f = \nabla f(I_D)$ and $\nabla^2 f = \nabla^2 f(I_D)$ for the gradient and Hessian, respectively, at the identity matrix, and we write $\nabla_a f$ and $\nabla^2_{ab}f$ for the components.
As block matrices,
\begin{align*}
  \nabla f = \left[\begin{array}{c} \nabla_0 f \\ \hline \nabla_1 f \\ \vdots \\ \nabla_k f \end{array}\right],
  \qquad
  \nabla^2 f = \left[\begin{array}{c|ccc}
  \nabla_{00}^2 f & \nabla_{01}^2 f \dots & \nabla_{0k}^2 f \\
    \hline\nabla_{10}^2 f & \nabla_{11}^2 f \dots & \nabla_{1k}^2 f \\
  \vdots  & \vdots & \ddots & \vdots \\
  \nabla_{k0}^2 f & \nabla_{k1}^2 f \dots & \nabla_{kk}^2 f \\
  \end{array}\right].
\end{align*}
Here, $\nabla_0 f \in \R$ and each $\nabla_a f(\Theta)$ is a $d_a \times d_a$ traceless symmetric matrix.
Similarly, for~$a, b \in [k]$ (i.e., for the blocks of the submatrix to the lower-right of the lines) the components $\nabla_{ab}^2f(\Theta)$ of the Hessian are linear operators from the space of traceless symmetric $d_b\times d_b$ matrices to the space of traceless symmetric $d_a \times d_a$ matrices, while $\nabla_{a0}f$ is a linear operator from $\R$ to the space of traceless symmetric $d_a\times d_a$ matrices (hence can itself be viewed as such a matrix), $\nabla_{0a}f$ is the adjoint of this linear operator, and $\nabla^2_{00} f(\Theta)$ is a real number.
\end{definition}

We note that the Hessian is symmetric with respect to the inner product~$\braket{\cdot,\cdot}$ on $\H$.
Just like in the Euclidean case, the Hessian is convenient to characterize strong convexity.
Indeed, $\braket{H, \nabla^2 f(\Theta) H} = \partial^2_{t=0} f(\exp_{\Theta}(tH))$ for all $H\in \H$.
Thus, $f$~is geodesically convex if and only if the Hessian is positive semidefinite, that is, $\nabla^2 f(\Theta) \succeq 0$.
Similarly, $f$ is $\lambda$-strongly geodesically convex if and only if~$\nabla^2 f(\Theta) \succeq \lambda I_{\H}$, i.e., the Hessian is positive definite with eigenvalues larger than or equal to~$\lambda$.

\subsection{Proof outline}\label{subsec:proof-sketch}
With the above definitions, we are able to state a proof plan for \cref{thm:tensor-frobenius}. Proofs of all claims not proved in this subsection can be found in the supplement\ifdefined\ARXIV\else~\cite{FORW25supp}\fi.
The proof is a Riemannian version of the standard approach using strong convexity, and it goes by the following steps:

\begin{enumerate}
\item\label{it:reduce}
\textbf{Reduce to identity:}
We can obtain $n$ independent samples from $\cN(0, \Theta^{-1})$ as $x'_i = \Theta^{-1/2} x_i$, where $x_1,\dots,x_n$ are distributed as $n$ independent samples from $\cN(0, I_D)$. By equivariance of the likelihood function, the MLE $\widehat{\Theta}(x')$ for the former is exactly $\Theta^{1/2} \widehat{\Theta}(x) \Theta^{1/2}$.
By invariance of the Fisher-Rao metric, $\dFR(\widehat\Theta(x'), \Theta) = \dFR(\widehat\Theta(x), I_D)$; the same is true for $\dop$.
This shows that to prove \cref{thm:tensor-frobenius} it is enough to consider the case that $\Theta = I_D$, i.e., the standard Gaussian.

\item\label{it:grad} \textbf{Bound the gradient:}
Show that the gradient $\nabla f_x(I_D)$ is small with high probability.

\item\label{it:convexity} \textbf{Strong convexity:}
with high probability, $f_x$ is $\Omega(1)$-strongly geodesically convex near $I$.
\end{enumerate}

These together imply the desired sample complexity bounds -- as in the Euclidean case, strong convexity in a suitably large ball about a point with small gradient implies the optimizer cannot be far.
Since in step~\ref{it:grad} we show that the gradient \emph{at the true covariance} is small, our approach will prove that the optimizer (i.e., the MLE) is not far from the true covariance.

We begin by formally stating the fact given in step~\ref{it:reduce}, as we will use it in later sections.

\begin{fact}\label{fact: reduce to identity}
    Let $x := (x_1, \ldots, x_n)$ be a tuple of $n$ independent samples of $\cN(0, I_D)$, and $x_i' := \Theta^{-1/2} x_i$ be the corresponding samples of $\cN(0, \Theta^{-1})$, with $x' := (x_1', \dots, x_n')$.
    If $\htheta(x), \htheta(x')$ are the MLE's for the samples $x, x'$, respectively, then $\htheta(x') = \Theta^{1/2} \htheta(x) \Theta^{1/2}$.

    Thus, $\dFR(\htheta(x'), \Theta) = \dFR(\htheta(x), I_D)$ and $\dop(\htheta(x'), \Theta) = \dop(\htheta(x), I_D)$.
\end{fact}

The following lemma shows that strong convexity in a ball about a point where the gradient is sufficiently small implies the optimizer cannot be far.
This lemma thus ensures that if we prove steps~\ref{it:grad} and~\ref{it:convexity}, then \cref{thm:tensor-frobenius} follows.

\begin{lemma}\label{lem:convex-ball}
Let $f\colon \P \to \R$ be geodesically convex and twice differentiable.
Let~$\Theta\in\P$ be such that $\norm{\nabla f(\Theta)}_F \leq \delta$, and $f$ is $\lambda$-strongly geodesically convex in a ball $B_r(\Theta)$ of radius~$r > \frac{2\delta}\lambda$.
Then the sublevel set $\{\Upsilon \in \P : f(\Upsilon) \leq f(\Theta)\}$ is contained in the ball~$B_{2\delta/\lambda}(\Theta)$, $f$ has a unique minimizer $\smash{\htheta}$, where~$\htheta \in B_{\delta/\lambda}(\Theta)$, and
$f(\htheta) \geq f(\Theta) - \frac{\delta^2}{2 \lambda}$.
\end{lemma}

Hence, we now need to carry out steps~\ref{it:grad} and~\ref{it:convexity} in the plan above.

\subsection{Bounding the gradient}
Proceeding as step~\ref{it:grad} of the plan from \cref{subsec:proof-sketch}, we now compute the gradient of the objective function and bound it using matrix concentration results.

To calculate the gradient, we need a definition from linear algebra.
Recall that our data comes as an $n$-tuple $x=(x_1,\dots,x_n)$ of $k$-tensors.
As in \cref{exa:usual-likelihood}, let $\rho := \frac1{nD}\sum_i x_i x_i^T$ denote the ``second sample moments'', and rewrite the objective function~\eqref{eq:neg log likelihood} as
\begin{align}\label{eq:obj via rho}
  f_x(\Theta) = \tr \rho \, \Theta - \frac1D \log \det \Theta.
\end{align}
We may also consider the ``second sample moments'' of a subset of the coordinates~$J \subseteq [k]$.
For this the following definition is useful.

\begin{definition}[Partial trace]\label{definition:partial-trace}
Let $\rho$ be an operator on $\R^{d_1} \ot \cdots \ot \R^{d_k}$, and~$J \subseteq [k]$ an ordered subset.
Define the \emph{partial trace} $\rho^{(J)}$ as the $d_J \times d_J$-matrix, where $d_J = \prod_{a\in J} d_a$, that satisfies the property that
\begin{align}\label{eq:partial trace duality}
  \tr \rho^{(J)} H
= \tr \rho \, H_{(J)}
\end{align}
for any $d_J\times d_J$ matrix~$H$, where $H_{(J)}$ denotes the operator on $\R^{d_1} \ot \cdots \ot \R^{d_k}$ that acts as~$H$ on the tensor factors labeled by~$J$ (in the order determined by~$J$) and as the identity on the rest.
This property uniquely determines $\rho^{(J)}$.
We write $\rho^{(a)}$ and $\rho^{(ab)}$ if~$J=\{a\}$ and~$J=\{a,b\}$, respectively.
\end{definition}

If $\rho$ is positive (semi)definite then so is $\rho^{(J)}$.
Moreover, $\tr \rho = \tr \rho^{(J)}$ and $(\rho^{(J)})^{(K)} = \rho^{(K)}$ for $K \subseteq J$.
Concretely, the partial trace $\rho^{(J)}$ can be computed analogously to the discussion in \cref{subsection: formal definitions}: ``flatten'' the data~$x$ by regarding it as a $d_J \times N_J$~matrix~$x^{(J)}$, where $N_J = \frac{nD}{d_J}$;
then $\rho^{(J)} = \frac1{nD} x^{(J)} (x^{(J)})^T$.
The gradient can be readily computed in terms of partial traces.

\begin{lemma}[Gradient]\label{lem:gradient}
Let $\rho = \frac{1}{nD} \sum_{i=1}^n \samp_i \samp_i^T $.
Then the components of the gradient~$\nabla f_x$ at the identity are given by
\begin{align}
 \nabla_a \ef_{\samp} &= \sqrt{d_a}\left( \rho^{(a)} - \frac{\tr\rho}{d_a} I_{d_a}\right)
  \qquad \text{ for } a \in [k], \label{eq:grad-a}\\
  \nabla_0 \ef_\samp &= \tr \rho - 1.\label{eq:grad-0}
\end{align}
\end{lemma}

\begin{remark}[Gradient at other points from equivariance]\label{remark:gradient-everywhere}
In the previous lemma we only computed the gradient at the identity.
However, this is without loss of generality, since from the calculations above one easily obtains $\nabla f_{x}(\Theta) = \nabla f_{\Theta^{1/2} x}(I)$.
That is, the gradient~$\nabla f_{x}(\Theta)$ is given by \cref{eq:grad-a,eq:grad-0} with $\rho$ replaced by $\Theta^{1/2}\rho\; \Theta^{1/2}.$
\end{remark}

Having calculated the gradient of the objective function, we are ready to state our bounds on the norm of the gradient, as outlined in step~\ref{it:grad} of \cref{subsec:proof-sketch}.

\begin{prop}[Gradient bound]\label{prop:gradient-bound}
Let $\rv = (\rv_1,\dots,\rv_n)$ consist of independent standard Gaussian random variables in~$\R^D$.
Suppose that $0<\eps<1$ and $n \geq \frac{\dmax^2}{D \eps^2}$.
Then, the following occurs with probability at least $1 - 2(k+1)e^{-\eps^2 {nD}/{(8\dmax)}}$:
\begin{align*}
  \norm{\nabla_a f_x}_{\op}
&\leq \frac{9\eps}{\sqrt{d_a}}
\qquad\text{ for all $a\in[k]$}, \\
  \abs{\nabla_0 f_x} &\leq \eps.
\end{align*}
As a consequence, 
$\norm{\nabla f_x}_F^2
\leq (1 + 81 k) \eps^2
\leq 82 k \eps^2$.
\end{prop}

\subsection{Strong convexity from expansion}\label{subsec:strong-convex}
In this section, we establish our strong convexity result, step~\ref{it:convexity} of the plan from \cref{subsec:proof-sketch}, in \cref{thm:ball-convexity}.
The proposition states that, with high probability, $f_x$ is strongly convex near the identity.
We will prove it by first establishing strong convexity \emph{at} the identity using quantum expansion techniques (\cref{thm:tensor-convexity}), and then (in the supplement) we bound how the Hessian changes away from the identity, see \CREFsupp{lem:convexRobustness}{Lemma~D.3}.
We then combine these results to prove \cref{thm:ball-convexity}.

Similar to our gradient calculations, we compute the components of the Hessian in terms of partial traces, but now we also need to consider two coordinates at a time.

\begin{lemma}[Hessian]\label{lem:hessian}
Let $\rho = \frac{1}{nD}\sum_{i=1}^n \samp_i \samp_i^T$.
Then the components of the Hessian~$\nabla^2 f_{\samp}$ at the identity are given by
\begin{align*}
  \braket{H, (\nabla^2_{aa} f_x) H} &= d_a \tr \rho^{(a)} H^2 \\
  \braket{H, (\nabla^2_{ab} f_x) K} &= \sqrt{d_a d_b} \tr \rho^{(ab)} \left( H \ot K \right)
\end{align*}
for all $a\neq b\in[k]$ and traceless symmetric $d_a\times d_a$ matrices $H$, $d_b\times d_b$ matrices~$K$, and
\begin{align*}
  \nabla^2_{0a} f_x&~\widehat= \sqrt{d_a} \left( \rho^{(a)} - \frac{\tr \rho}{d_a} I_{d_a} \right) \widehat=~ \nabla^2_{a0} f_x \quad (\forall a \in [k]), \\
  \nabla^2_{00} f_x&= \tr \rho.
\end{align*}
\end{lemma}

\noindent
Here we use the conventions from \cref{def:hess grad}.
In particular, we identify $\nabla^2_{a0} f_x$, which is a linear operator from $\R$ to the traceless symmetric matrices, with a traceless symmetric matrix, and similarly for its adjoint $\nabla^2_{0a} f_x$.
The notation $\widehat=$ reminds us of these identifications.

\begin{remark}[Hessian at other points from equivariance]\label{remark:hessian-everywhere}
Analogously to \cref{remark:gradient-everywhere}, we can compute the Hessian at other points using $\nabla^2 f_{x}(\Theta) = \nabla^2 f_{\Theta^{1/2} x}$.
That is, the Hessian~$\nabla^2 f_{x}(\Theta)$ is given by \cref{lem:hessian} with $\rho$ replaced by $\Theta^{1/2}\rho\; \Theta^{1/2}.$
\end{remark}

The most interesting part of the Hessian are the off-diagonal components for $a\neq b\in[k]$, which up to a multiplicative factor $\sqrt{d_a d_b}$ can be seen as the restrictions of the linear maps
\begin{align}\label{eq:hessian channel}
  \Phi^{(ab)} \colon \Mat(d_b) \to \Mat(d_a) \quad\text{ given by }\quad \braket{H, \Phi^{(ab)}(K)} &= \tr \rho^{(ab)} \left( H \ot K \right)
\end{align}
to the traceless symmetric matrices.
\Cref{eq:hessian channel} is a special case of a \emph{completely positive map}, which is a linear map of the form
\begin{align}\label{eq:def cp}
  \Phi_A \colon \Mat(d_b) \to \Mat(d_a), \quad \Phi_A(Z) = \sum_{i=1}^N A_i Z A_i^T
\end{align}
for $d_a\times d_b$ matrices $A_1,\dots,A_N$.
Completely positive maps are quantum analogues of nonnegative matrices.
To see that $\Phi^{(ab)}$ is completely positive, note that since $\rho^{(ab)}$ is positive semidefinite, it can be written in the form $\sum_{i=1}^N \vect(A_i) \vect(A_i)^T$; then $\Phi^{(ab)} = \Phi_A$ follows.
The matrices $A_1,\dots,A_N$ are known as \emph{Kraus operators}.
\Cref{eq:def cp} can also be written as
\begin{align}\label{eq:vec rep}
  \vect(\Phi_A(Z)) = \sum_{i=1}^N (A_i \ot A_i) \vect(Z).
\end{align}

Let $\Phi^*\colon\Mat(d_a)\to\Mat(d_b)$ be the adjoint of a completely positive map~$\Phi$ with respect to the Hilbert-Schmidt inner product; this is again a completely positive map, with Kraus operators $A_1^T,\dots,A_N^T$.
%
In our proof of strong convexity, we will show that strong convexity follows if the completely positive maps $\Phi^{(ab)}$ are good \emph{quantum expanders}. Quantum expansion is a quantum analogue of expansion of a nonnegative matrix viewed as a bipartite graph.

\begin{definition}[Quantum expansion]\label{def:expansion}
Let $\Phi\colon\Mat(d_b) \to \Mat(d_a)$ be a completely positive map.
Say $\Phi$ is \emph{$\eps$-doubly balanced} if
\begin{align}\label{eq:doubly balanced}
  \norm*{\frac{\Phi(I_{d_b})}{\tr \Phi(I_{d_b})} - \frac{I_{d_a}}{d_a}}_{\op} \leq \frac\eps{d_a}
\quad\text{and}\quad
  \norm*{\frac{\Phi^*(I_{d_a})}{\tr \Phi^*(I_{d_a})} - \frac{I_{d_b}}{d_b}}_{\op} \leq \frac\eps{d_b}.
\end{align}
The map $\Phi$ is an \emph{$(\eps, \eta)$-quantum expander} if $\Phi$ is $\eps$-doubly balanced and
\begin{align}\label{eq:expansion}
  \norm{\Phi}_0 :=
  \max_{\substack{H \in \Mat(d_{a}) \\ \text{traceless symmetric}}}
  \;
  \max_{\substack{K \in \Mat(d_{b}) \\ \text{traceless symmetric}}}
  \frac{\langle H, \Phi(K) \rangle}{\norm H_F \norm K_F}
\leq \eta \frac{\tr \Phi(I_{d_b})}{\sqrt{d_ad_b}}
\end{align}
A $(0, \eta)$-quantum expander is called a \emph{$\eta$-quantum expander}.
\end{definition}

\noindent
Quantum expanders originate in quantum information theory and quantum computation~\cite{H07}.
There one typically takes $d_a=d_b$ and $\eps=0$, so that \cref{eq:expansion} simplifies to~$\norm{\Phi}_0 \leq \eta$.
Here we follow the definitions of~\cite{KLR19,FM20}, who recognized the connection between quantum expansion and spectral gaps of the Hessian for operator scaling.\footnote{\Cref{def:expansion} is invariant under rescaling $\Phi \mapsto c\Phi$ for $c>0$. We note that some of the above can be slightly simplified if one opts for a non-scale invariant definition.}
The following lemma allows us to translate quantum expansion properties into strong convexity.

\begin{lemma}[Strong convexity from expansion]\label{lem:expansion-convexity}
If the completely positive maps $\Phi^{(ab)}$ defined in \cref{eq:hessian channel} are $(\eps,\eta)$-quantum expanders for every $a\neq b\in[k]$, then
\begin{align*}
  \norm*{\frac{\nabla^2 f_x}{\tr \rho} - I_\H}_{\op}
\leq (k-1)\eta + (\sqrt k + 1) \eps.
\end{align*}
Assuming $k\geq3$, the right-hand side is at most $k (\eta + \eps)$.
\end{lemma}

We are concerned with $\Phi^{(ab)}$ that arise from random Gaussians.
Just like random graphs give rise to good expanders, random completely positive maps (choosing Kraus operators at random from well-behaved distributions) yield good quantum expanders.
When the Kraus operators are standard Gaussians we have the following result by \cite{pisier2012grothendieck,P14}.%
\footnote{Pisier's technical result is slightly different. We state and prove our variant of Pisier's theorem in \CREFsupp{thm:Pisier-expansion}{Theorem~C.1} in the supplement.}

\begin{theorem}[Pisier]\label{thm:hess-pisier}
Let $A_1,\dots,A_N$ be independent $d_a\times d_b$ random matrices with independent standard Gaussian entries.
Then, for every $t\geq 2$, with probability at least~$1 - t^{-\Omega(d_a + d_b)}$, the completely positive map $\Phi_A$, defined as in \cref{eq:def cp}, satisfies
\begin{align*}
  \norm{\Phi_A}_0 \leq O\left(t^{2} \sqrt N \left( d_a + d_b \right)\right).
\end{align*}
\end{theorem}
\begin{proof}
Observe that
\begin{align*}
  \norm{\Phi_A}_0
= \max_{\substack{H \text{ traceless symmetric} \\ \norm H_F=1}} \norm{\Phi(H)}_F
\leq 
\max_{\substack{H \in \Mat(d_b) \\ \norm H_F = 1}} \norm{\Phi(\Pi(H))}_F
= \norm{\Phi \circ \Pi}_{\op}.
\end{align*}
Here we identify $\Mat(d_b) \cong \R^{d_b} \ot \R^{d_b}$, so $\Pi$ identifies with the orthogonal projection onto the traceless matrices, and we used that $\norm{\Pi(H)}_F \leq \norm{H}_F$, since $\Pi$ is an orthogonal projection.
Using \cref{eq:vec rep}, the result now follows from \CREFsupp{thm:Pisier-expansion}{Theorem~C.1} with $n=d_a$ and $m=d_b$.
\end{proof}

When the samples $x=(x_1,\dots,x_n)$ are independent standard Gaussians in $\R^D$, the random completely positive maps $\Phi^{(ab)}$ have the same distribution as~$\frac1{nD}\Phi_A$, where the Kraus operators~$A_1,\dots,A_N$ are $d_a \times d_b$ matrices with independent standard Gaussian entries and~$N=\frac{nD}{d_ad_b}$.
Accordingly, strong convexity at the identity follows quite easily from \cref{thm:hess-pisier} once doubly balancedness can be controlled.
For the latter, observe that
\begin{align*}
  \norm*{\frac{\Phi^{(ab)}(I_{d_b})}{\tr \Phi^{(ab)}(I_{d_b})} - \frac{I_{d_a}}{d_a}}_{\op}
= \frac1{\tr\rho} \norm*{\rho^{(a)} - \frac{\tr \rho}{d_a} I_{d_a}}_{\op}
= \frac1{1 + \nabla_0 f_x} \frac1{\sqrt{d_a}} \norm*{\nabla_a f_x}_{\op},
\end{align*}
by \cref{lem:gradient}, and similarly for the adjoint.
Therefore, the completely positive maps $\Phi^{(ab)}$ are $\eps$-doubly balanced if and only if, for all $a\in[k]$,
\begin{align}\label{eq:balanced via grad}
  \sqrt{d_a} \norm*{\nabla_a f_x}_{\op} \leq \eps \tr \rho = \left( 1 + \nabla_0 f_x \right) \eps,
\end{align}
hence double balancedness can be controlled using the gradient bounds in \cref{prop:gradient-bound}.

Using \cref{thm:hess-pisier} we can prove the following strong convexity result at the identity.

\begin{prop}[Strong convexity at identity]\label{thm:tensor-convexity}
There is a universal constant $C>0$ such that the following holds.
Let $x = (x_1,\dots,x_n)$ be independent standard Gaussian random variables in~$\R^D$, where $n \geq C k \frac{\dmax^2}D$.
Then, with probability at least~$1 - k^2 ( \frac{\sqrt{nD}}{k \dmax} )^{-\Omega(\dmin)}$,
\begin{align*}
  \norm{\nabla^2 f_x - I_\H}_{\op} \leq \frac14;
\end{align*}
in particular, $f_x$ is $\frac34$-strongly convex at the identity.
\end{prop}

We also prove a robustness result for the Hessian (\CREFsupp{lem:convexRobustness}{Lemma~D.3}), which implies that when our function is strongly convex at the identity then it is also strongly convex in an \emph{operator norm} (Thompson metric -- the $\dop$ defined in \cref{definition: fisher-rao and thompson}) ball about the identity.
Accordingly, we obtain the following proposition.

\begin{prop}[Strong convexity near identity]\label{thm:ball-convexity}
There are constants~$C,c>0$ such that the following holds.
Let $x = (x_1,\dots,x_n)$ be independent standard Gaussian random variables in~$\R^D$, where $n \geq C k \frac{\dmax^2}D$.
Then, with probability at least~$1 - k^2 ( \frac{\sqrt{nD}}{k \dmax} )^{-\Omega(\dmin)}$,
the function~$f_x$ is $\frac12$-strongly convex at any point $\Theta\in\P$ such that~$\dop(\Theta, I_D) \leq c$.
\end{prop}

While \cref{thm:ball-convexity} uses $\dop$ to quantify closeness, we can easily translate it into a statement in terms of the geodesic distance.
Namely, under the same hypotheses~$f_x$ is $\frac12$-strongly convex on the geodesic ball~$B_r(I_D)$ of radius
$r = c / \sqrt{(k+1)\dmax}$, where $c>0$ is the universal constant from \cref{thm:ball-convexity}.
This follows from the following lemma.

\begin{lemma}\label{lem:op ball vs frob ball}
For any $\Theta\in\SPD$, we have $\dop(\Theta, I_D) \leq \sqrt{(k+1)\dmax} \cdot d(\Theta, I_D)$.
\end{lemma}

\subsection{Tensor normal model: sample complexity \& error bounds}
\label{subsec:tnmSampleComplexity}
We have all ingredients to prove \cref{thm:tensor-frobenius} according to the plan in \cref{subsec:proof-sketch}.
Since the objective is strongly convex and its gradient is small with high probability, \cref{lem:convex-ball} implies the next result, which bounds the geodesic distance between the MLE and the true precision matrix.


\begin{proof}[Proof of \cref{thm:tensor-frobenius}]
By \cref{fact: reduce to identity}, we may prove the theorem assuming~$\Theta = I_D$.
Assuming this, we now show that the minimizer of $f_\rv$ is unique and is close to $\Theta = I_D$ with high probability.
Recall from \cref{eq:eps sqr assm} that $n \geq C k^2 \frac{\dmax^3}{D} t^2$.

Let~$c>0$ be the constant from \cref{thm:ball-convexity}.
Consider the two events:
\begin{enumerate}
\item\label{it:grad-bd} $\norm{\nabla f_x}_F \leq \delta := \sqrt{82 k} \dfrac{\dmax }{\sqrt{nD}} t$.
\item\label{it:sc-ball} $f_\rv$ is $\lambda$-strongly convex over $B_r(I_D)$, where $\lambda=\frac12$ and $r := \dfrac c{\sqrt{(k+1) \dmax}}$.
\end{enumerate}
By our choice of parameters, where $C$ is a large enough constant, we have
$$\frac{\delta}{\sqrt{82k}} < 1, \quad n \geq \smash{\frac{\dmax^2}{D (\frac\delta{\sqrt{82 k}})^2}}, \quad n \geq C k \frac{\dmax^2}{D}, \quad \text{ and } r > \frac{2\delta}{\lambda}.$$

Thus, \cref{prop:gradient-bound}, with $\eps=\smash{\frac\delta{\sqrt{82 k}}}$, applies and it shows that the first event holds up to a failure probability of at most
\begin{align*}
  2(k+1)e^{-(\frac\delta{\sqrt{82 k}})^2 \frac{nD}{8\dmax}}
= k e^{-\Omega(t^2 \dmax)}.
\end{align*}
Moreover, \cref{thm:ball-convexity,lem:op ball vs frob ball} also apply, showing that the second event holds up to a failure probability of at most
\begin{align*}
  k^2 \left( \frac{\sqrt{nD}}{k \dmax} \right)^{-\Omega(\dmin)}.
\end{align*}
By the above and the union bound, both events hold simultaneously with the claimed success probability.
Thus, as the above two events hold, \cref{lem:convex-ball} applies (with our choice of~$\delta$ and~$\lambda$) and shows that the MLE~$\htheta$ exists, is unique, and satisfies $d(\htheta, \Theta) \leq \frac\delta\lambda = 2 \delta$.

Since $d(\htheta_a, \Theta_a) \leq d(\htheta, \Theta)$, by the relationship between geodesic distance and Fisher-Rao distance, we get
$$\dFR(\htheta, \Theta) = \sqrt{\frac{D}{2}} \cdot d(\htheta, \Theta), \quad \text{ and } \quad \dFR(\htheta_a, \Theta_a) \leq \sqrt{\frac{d_a}{2}} \cdot d(\htheta_a, \Theta_a) \leq \sqrt{\frac{d_a}{2}} \cdot d(\htheta, \Theta)$$
which imply the desired distance bounds.
\end{proof}

\section{Matrix normal model: improved sample complexity \& error bounds}\label{sec:matrix-normal}

We can prove a stronger result for the matrix normal model ($k=2$).
\cref{thm:matrix-normal} improves over \cref{thm:tensor-frobenius} in the following aspects:
\begin{enumerate}
\item it works over a better (i.e. \emph{smaller}) sample threshold,
\item we obtain tight error bounds for the individual factors in \emph{spectral distance} $\dop$,
\item the failure probability is \emph{inverse exponential} in the number of samples.
\end{enumerate}

Recall that when $k=2$, the samples can be viewed as $d_1 \times d_2$-matrices, denoted by~$X_i$.
From the samples, we construct the completely positive map $\Phi_X : \Mat(d_2) \to \Mat(d_1)$ defined as $\Phi_X(Z) := \sum_{i=1}^n X_i Z X_i^T$.
The above improvements come from working directly with quantum expansion, via the spectral gap of the completely positive map $\Phi_X$, instead of translating it into strong convexity.

One of our main technical results is the following theorem, which shows that the expansion parameter of the map can be made \emph{constant} with \emph{exponentially small} failure probability.

\begin{theorem}[Improved expansion]\label{thm:operator-cheeger}
There are universal constants $C > 0$ and $\eta\in(0,1)$ such that the following holds.
For $d_1 \leq d_2$, $d_2>1$, let $X=(X_1,\dots,X_n)$ be random $d_1 \times d_2$ matrices with independent standard Gaussian entries, where $n \geq C \frac{d_2}{d_1} \max\{\log d_2, t^2\}$ and $t\geq1$.
Then, $\Phi_X$ is a $\left(t \sqrt{\frac{d_2}{n d_1}}, \eta\right)$-quantum expander with probability at least $1 - e^{ - \Omega( d_2 t^2)}$.
\end{theorem}

\noindent We prove \cref{thm:operator-cheeger} in \CREFsupp{app:cheeky}{Section~C.2} by the use of Cheeger's inequality.
Our techniques are similar to the ones used in \cite{FM20}.\footnote{\cref{thm:operator-cheeger} also improves our result on strong convexity (\cref{thm:tensor-convexity,thm:ball-convexity}) for $k=2$.
Indeed, for $k = 2$, using \cref{thm:operator-cheeger} in place of \cref{thm:hess-pisier} improves the failure probability to $1 - e^{ - \Omega( d_2 t^2)}$.
However, we cannot use this to improve our results for $k \geq 3$ because \cref{thm:operator-cheeger} is not capable of proving subconstant quantum expansion.
We need quantum expansion less than $1/(k-1)$ to obtain a nontrivial result from \cref{lem:expansion-convexity}.
}

To obtain our error bounds, we combine the above result on the quantum expansion with the work of \cite{KLR19}, which gives us bounds in operator norm on how far the MLE is from our true precision matrices as a function of the expansion.

The above takes care of aspect~1 (estimating in operator norm with a reduced sample threshold) and aspect~3 (inverse exponential failure probability), as well as tight error bounds on the larger Kronecker factor of the precision matrix.
Now, we need to work a bit more to get tight bounds on the smaller  factor of the precision matrix.
To get a better control on the smaller factor, the idea is to apply one step of the flip-flop algorithm to ``renormalize'' the samples such that the second (larger dimensional) partial trace is proportional to $I_{d_2}$.
This has the effect of making the second component of the gradient~$\nabla f_\samp$ equal to zero.
In \CREFsupp{lem:flipflop-concentration}{Proposition~E.6}, we show that, even after the first step of flip-flop, the first component still enjoys the same concentration exploited in \cref{prop:gradient-bound} -- thus the total gradient has become smaller, but only the second component of the MLE estimate has changed.
Thus, intuitively, the total change in the first component will be small.
Combining \CREFsupp{lem:flipflop-concentration}{Proposition~E.6} with \CREFsupp{lem:robust expansion}{Lemma~E.9}, which shows robustness of quantum expansion, we are able to control the quantum expansion of the new completely positive map.
Hence, we are again in position to employ \CREFsupp{cor:klr}{Corollary~E.4} to get tight error bounds for the smaller Kronecker factor.

The detailed proof of \cref{thm:matrix-normal} and the necessary claims are given in \CREFsupp{app:matrix}{Section~E}.

\section{Lower bounds}\label{sec:lower}

In this section we prove new lower bounds for estimating precision matrices in the matrix and tensor normal models.
Proofs of all claims not proved in this section can be found in the supplement \CREFsupp{app:lower}{Section~F}.
We begin by stating a well-known lower bound for estimating unstructured precision matrices (the case $k=1$).

\begin{prop}[Lower bound for unstructured Gaussians]\label{cor:relative-lower}
There is $c > 0$ such that the following holds.
Let $\htheta$ be any estimator for $\Theta \in \PD(d)$ from a tuple $X$ of $n$ samples from $\cN(0, \Theta^{-1})$.
Let $B\subset \PD(d)$ be the operator norm ball about $I_d$ of radius~$1/2$.
Then:
\begin{enumerate}
\item Let $\delta^2 = c \, \min \left\{1,\frac{d^2}{n}\right\}$. Then,
$\displaystyle \sup_{\Theta \in B} \Pr\left[ \dFR(\htheta,  \Theta)  \geq \delta\right] \geq \frac{1}{2}$.
\item Let $\delta^2 = c \, \min \left\{1,\frac{d}{n}\right\}$. Then,
$\displaystyle\sup_{\Theta \in B} \Pr\left[ \dop(\htheta,  \Theta) \geq \delta\right] \geq \frac{1}{2}$.
\end{enumerate}
As a consequence (see \cref{rem:Minimax}), we have
\begin{align*}
\sup_{\Theta \in B}\E[\dFR(\htheta,  \Theta)^2] =\Omega\left( \min \left\{\frac{d^2}{n},1\right\}\right)
\text{ and }
\sup_{\Theta \in B}\E[\dop(\htheta, \Theta)^2] = \Omega\left( \min \left\{\frac{d}{n},1\right\}\right).
\end{align*}
\end{prop}

Having the lower bound above in mind, we now discuss what is needed to prove a lower bound for the matrix normal model.
In this section we assume, without loss of generality, that $d_{2} \geq d_{1} \geq 1$ and we are given samples $X_{1}, ..., X_{n} \in \R^{d_{1} \times d_{2}}$ distributed according to $\text{vec}(X) \sim N(0, \Theta_{1}^{-1} \otimes \Theta_{2}^{-1})$.
If $\Theta_{1}$ was known, we could compute $Y := \Theta_{1}^{1/2} X$, which `de-correlates' the rows of $X$, and therefore we could treat the rows of $Y$ as $n d_{1}$ independent samples from $N(0,\Theta_{2}^{-1})$. If $nd_{1} \leq c d_{2}$ for some small enough $c > 0$ (i.e. $n < c d_{2}/d_{1}$), the above $k=1$ lower bound would imply that we cannot estimate $\Theta_{2}$ to constant accuracy in the operator norm even if we had complete knowledge of $\Theta_{1}$.

Since $d_{2} \geq d_{1}$, we could hope for better results for estimating $\Theta_{1}$, since we intuitively have more samples for this mode. Namely, assume we knew $\Theta_{2}$ and pre-process $Y := X \Theta_{2}^{1/2}$ to `de-correlate' the columns of $X$, which means we could treat the columns of $Y$ as $n d_{2}$ independent samples from $N(0,\Theta_{1}^{-1})$. In this case we could estimate $\Theta_{1}$ in operator norm with RMSE rate of $O(\sqrt{d_{1}/nd_{2}})$.
One could hope that this rate holds for $\Theta_1$ even when~$\Theta_2$ is not known. Here we show that, to the contrary, the rate for $\Theta_1$ cannot be better than $O(\sqrt{d_1/ n \min(n d_1, d_2)})$. Thus, for $n \ll d_{2}/d_{1}$, it is impossible to estimate $\Theta_1$ as well as one could if $\Theta_2$ were known.

\begin{theorem}[Lower bound for matrix normal models]\label{thm:matrix-lower}
There is $c > 0$ such that the following holds.
Let $d_1 \leq d_2$, $\Theta_1 \in \PD(D_1), \Theta_2 \in \PD(d_2)$ and $\htheta_1$ be any estimator for $\Theta_1$ from a tuple $X$ of $n$ samples of $\cN(0, \Theta_1^{-1} \ot \Theta_2^{-1})$. Let $B\subset \PD(d_1)$ denote the ball about $I_{d_1}$ of radius~$1/2$ in the operator norm.
Then:
\begin{enumerate}
\item\label{it:frob-lower} Let $\delta^2 = c \, \min \left\{1,\frac{d_1^2}{n \min \{n d_1, d_2\}}\right\}$.
Then,
$\displaystyle\sup_{\substack{\Theta_1 \in B \\ \Theta_2 \in \PD(d_2)}} \Pr\left[ \dFR(\htheta_1,  \Theta_1)  \geq \delta\right] \geq \frac{1}{2}$.
\item\label{it:op-lower} Let $\delta^2 = c \, \min \left\{1,\frac{d_1}{n \min \{n d_1, d_2\}}\right\}$.
Then,
$\displaystyle\sup_{\substack{\Theta_1 \in B \\ \Theta_2 \in \PD(d_2)}} \Pr\left[ \dop(\htheta_1,  \Theta_1) \geq \delta\right] \geq \frac{1}{2}$.
\end{enumerate}
As a consequence, we have
\begin{align*}\sup_{\Theta_1 \in B, \Theta_2 \in \PD(d_2)}\E[\dFR(\htheta_1, \Theta_1)^2] &=\Omega\left( \min \left\{\frac{d_1^2}{n \min \{n d_1, d_2\}},1\right\}\right)\\
\text{ and } \sup_{\Theta_1 \in B, \Theta_2 \in \PD(d_2)}\E[\dop(\htheta_1,  \Theta_1)^2] &= \Omega\left( \min \left\{\frac{d_1}{n \min \{n d_1, d_2\}},1\right\}\right).\end{align*}
\end{theorem}

Intuitively, the above theorem holds because we can choose $\Sigma_2$ to zero out all but $nd_1$ columns of each $X_i$, which allows access to at most $n \cdot n d_1$ samples from a Gaussian with precision $\Theta_1$.
However, this does not quite work because $\Sigma_2$ would not be invertible and hence the precision matrix $\Theta_2$ would not exist.
We must instead choose $\Sigma_2$ to be approximately equal to a random projection of rank $n d_1$.
This allows us to deduce the same lower bounds for estimating $\Theta_1$ as the Gaussian case with at most $n\min \{d_2, n d_1\}$ independent samples.

One might ask why the rank of the random projection cannot be taken to be even less than~$n d_1$, yielding an even stronger bound. If the rank is less than $n d_1$, then the support of~$\Sigma_{2}$ can be estimated.
This would allow one to approximately diagonalize $\Sigma_2$ so that the $n$ samples can be treated as $nd_2$ independent samples in $\R^{d_1}$, yielding the rate $\sqrt{d_1 / n d_2}$ for~$\Theta_{1}$ in the operator norm using, e.g., Tyler's M-estimator \cite{FM20}. We now state the main tool in establishing the lower bound.

\begin{lemma}\label{lem:reduce-lower}
Let $X$ denote a tuple of $n$ samples from $\cN(0, \Theta_1^{-1} \ot \Theta_2^{-1})$
and let $\widehat{\Theta}_1(X)$ be any estimator for $\Theta_1$.
Let $Y$ be a tuple of $n\min\{nd_1, d_2\}$ samples from $\cN(0, \Theta_1^{-1})$.
For every $\delta > 0$, there is a distribution on $\Theta_2$ and an estimator $\tilde{\Theta}(Y)$ such that the distribution of~$\widehat{\Theta}_1(X)$ and the distribution of~$\tilde{\Theta}(Y)$ differ by at most $\delta$ in total variation distance.
\end{lemma}

We will use this lemma to show \cref{thm:matrix-lower} in the contrapositive: if there was a good estimator for the matrix normal model, then we could use this to produce a good estimator for Gaussian estimation. Namely, given Gaussian samples $Y \sim N(0,\Theta_{1}^{-1})$, we could simulate samples $X \sim N(0, \Theta_{1}^{-1} \otimes \Theta_{2}^{-1})$ from the matrix normal model by considering $X_{i} := (Y_{i,1} \, \cdots \, Y_{i,d_{2}}) \sqrt{\Theta_{2}}$, i.e. grouping $d_{2}$ columns into a matrix and applying $\sqrt{\Theta_{2}}$ on the right.
Then by the above lemma, if $\hat{\Theta}_{1}(X)$ is a good estimator of $\Theta_{1}$, then $\Tilde{\Theta}(Y)$ is also a good estimator for $\Theta_{1}$.
We now give the formal proof of \cref{thm:matrix-lower}.

\begin{proof}[Proof of \cref{thm:matrix-lower}]
To show claim~\ref{it:frob-lower}, let $\delta^2 \leq c \, \min \left\{1,\frac{d_1^2}{n \min \{n d_1, d_2\}}\right\}$.
Let $\Theta_2$ be distributed as in \cref{lem:reduce-lower} so that, as guaranteed by \cref{lem:reduce-lower} for $n \min \{n d_1, d_2\}$ samples $Y \sim N(0, \Theta_{1}^{-1})$ there is an estimator $\tilde{\Theta}(Y)$ satisfying $\DTV (\htheta_1(X), \tilde{\Theta}(Y)) \leq \delta_0$.
Here $X$ is distributed according to the matrix normal model $X \sim N(0,\Theta_1^{-1} \otimes \Theta_2^{-1})$.
\cref{cor:relative-lower} implies \begin{align*}
\sup_{\Theta_1 \in B} \Pr_Y\left[ \dFR(\tilde{\Theta}(Y),  \Theta_1)  \geq \delta\right] \geq \frac{1}{2}.
\end{align*}
Clearly we have
$$\sup_{\substack{\Theta_1 \in B, \\ \Theta_2 \in \PD(d_2)}} \Pr_{X}
\left[ \dFR(\htheta_1(X),  \Theta_1) \geq \delta \right]
\geq \sup_{\Theta_1 \in B} \Pr_{\Theta_2, X}\left[ \dFR(\htheta_1(X), \Theta_1)\geq \delta\right].$$
On the other hand, since the distributions of $\hat{\Theta}_{1}(X)$ and $\Tilde{\Theta}(Y)$ differ by at most $\delta_{0}$ in total variation distance, this implies
\begin{align*}
\sup_{\Theta_1 \in B} \Pr_{\Theta_2, X}\left[ \dFR(\htheta_1(X),  \Theta_1)\geq \delta\right]
&\geq \sup_{\Theta_1 \in B} \Pr_{Y}\left[ \dFR(\tilde{\Theta}(Y),  \Theta_1)
 \geq \delta\right] - \delta_0 \\
&\geq \frac{1}{2}  -  \delta_0.
\end{align*}
Allowing $\delta_0 \to 0$ implies claim~\ref{it:frob-lower}.
To prove claim~\ref{it:op-lower}, replace $\dFR$ by $\dop$ in the above.
\end{proof}

We remark that the proof of \cref{thm:matrix-lower} uses no properties about $\dFR$ or $\dop$.
Therefore, the above proof implies that any lower bound for estimating a Gaussian with $n \min \{n d_1, d_2\}$ samples transfers similarly to the matrix normal model.
The above strategy can clearly be lifted to the tensor normal model by considering more components:

\begin{theorem}[Lower bound for tensor normal models]\label{thm:tensor-lower}
There is $c > 0$ such that the following holds. Let $\Theta_1 \in \PD(d_1), \Theta_a \in \PD(d_a)$ for $a \in [k]$ and $\htheta_1$ be any estimator for $\Theta_1$ from a tuple $X$ of $n$ samples of $\cN(0, \otimes_{a \in [k]} \Theta_a^{-1})$. Let $B\subset \PD(d_1)$ denote the ball about $I_{d_1}$ of radius~$1/2$ in the operator norm.
Then:
\begin{enumerate}
\item \label{it:tensor-frob-lower} Let $\delta^2 = c \, \min \left\{1,\frac{d_1^2}{n \min \{n d_1, D/d_{1}\}}\right\}$.
Then,
$\displaystyle\sup_{\substack{\Theta_1 \in B \\ \Theta_a \in \PD(d_a), 1 \neq a \in [k]}} \Pr\left[ \dFR(\htheta_1,  \Theta_1)  \geq \delta\right] \geq \frac{1}{2}$.
\item\label{it:tensor-op-lower} Let $\delta^2 = c \, \min \left\{1,\frac{d_1}{n \min \{n d_1, D/d_{1}\}}\right\}$.
Then,
$\displaystyle\sup_{\substack{\Theta_1 \in B \\ \Theta_a \in \PD(d_a), 1 \neq a \in [k]}} \Pr\left[ \dop(\htheta_1,  \Theta_1) \geq \delta\right] \geq \frac{1}{2}$.
\end{enumerate}
As a consequence, we have
\begin{align*}\sup_{\Theta_1 \in B, \Theta_a \in \PD(d_a), 1 \neq a \in [k]}\E[\dFR(\htheta_1, \Theta_1)^2] &=\Omega\left( \min \left\{\frac{d_1^2}{n \min \{n d_1, D/d_{1}\}},1\right\}\right)\\
\text{ and } \sup_{\Theta_1 \in B, \Theta_a \in \PD(d_a), 1 \neq a \in [k]}\E[\dop(\htheta_1,  \Theta_1)^2] &= \Omega\left( \min \left\{\frac{d_1}{n \min \{n d_1, D/d_{1}\}},1\right\}\right).\end{align*}
\end{theorem}

\section{Iteration complexity of the flip-flop algorithm}\label{sec:flip-flop}

We now prove \cref{thm:tensor-flipflop,thm:matrix-flipflop}, which state fast convergence of the flip-flop algorithm to the MLE with high probability.
Detailed proofs of our main technical result, \cref{prop:meta}, along with the claims needed to prove it, can be found in \CREFsupp{app:flip-flop}{Section~G} in the supplement.

We state the flip-flop algorithm for the general tensor normal model in \cref{alg:flip-flop}.
It generalizes \cref{alg:flip-flop matrix} presented earlier in \cref{subsec:flip-flop} for the matrix normal.

\begin{Algorithm}
\begin{description}
\item[\hspace{.2cm}\textbf{Input}:] Samples $\samp = (\samp_1, \ldots, \samp_n)$, where each $\samp_i \in \R^D = \R^{d_1} \ot \cdots \ot \R^{d_k}$.
Parameters $T\in\N$ and $\delta>0$.
Initial guess $\wttheta \in \SPD$ satisfying $\tr[\rho \wttheta] = 1$.

\item[\hspace{.2cm}\textbf{Output}:]
An estimate $\otheta = \otheta_1 \ot \cdots \ot \otheta_k \in \SPD$ of the MLE.

\item[\hspace{.2cm}\textbf{Algorithm}:]
\end{description}
\begin{enumerate}
\item\label{it:flip-flop step 1}
Set $\otheta_a = \wttheta_a$ for each $a \in [k]$.

\vspace{10pt}

\item\label{it:flip-flop step 2}
For $t=1,\dots,T$, repeat the following:

\vspace{5pt}

\begin{itemize}
\item Compute $\rho_t = \dfrac{1}{nD} \cdot  \otheta^{1/2} \left( \sum_{i=1}^n x_ix_i^T \right) \otheta^{1/2}$, where $\otheta = \otheta_1 \ot \cdots \ot \otheta_k$.

\item Compute each component of the gradient using the formula
$\nabla_a f_x(\otheta) = \sqrt{d_a} \left( \rho_t^{(a)} - \tr(\rho_t) \smash{\frac{I_{d_a}}{d_a}} \right)$,
where $\smash{\rho_t^{(a)}}$ denotes the partial trace (\cref{definition:partial-trace}),
and find the index $a \in [k]$ for which $\norm{\nabla_a f_x(\otheta)}_F$ is largest.

\vspace{5pt}

\item
If $\norm{\nabla_a f_x(\otheta)}_F \leq \delta$, return $\otheta$

\vspace{3pt}

\item Update $\otheta_a \leftarrow \frac{1}{d_a} \otheta_a^{1/2} \left(\rho_t^{(a)}\right)^{-1} \otheta_a^{1/2}$.
\end{itemize}
\end{enumerate}
\caption{Flip-flop algorithm for the tensor normal model ($k\geq2$).%
}\label{alg:flip-flop}
\end{Algorithm}

\begin{remark}[Matrix flip-flop from tensor flip-flop]\label{rem:flip-flop relation}
To see how \cref{alg:flip-flop matrix} arises from \cref{alg:flip-flop}, note that if we update $\otheta_a$ in the $t$-th iteration, then the corresponding gradient component vanishes in the subsequent iteration.
Since for the matrix normal model there are only two gradient components to consider, this means that the algorithm will necessarily alternate between updating $\otheta_1$ and $\otheta_2$.
In other words, for the matrix normal model the algorithm truly ``flip-flops'' between the two coordinates.
Moreover, \CREFsupp{lemma:flip-flop-update}{Lemma~G.1} shows that $\tr\rho_t=1$ from the second iteration of \cref{alg:flip-flop} onwards.
Therefore, \cref{alg:flip-flop matrix} agrees with \cref{alg:flip-flop} except that in the first iteration we skip the stopping condition and always update $\otheta_1$.
This will not impact the analysis, as one can see in \CREFsupp{lem:flip-flop-sinkhorn}{Lemma~G.5}.
\end{remark}

The key insight is that given appropriate initial conditions on the samples (which we later show to hold under the same sample requirements as for our results on the MLE), the flip-flop algorithm will converge quickly to the MLE.
Namely, we show that the MLE is in a constant size operator norm ball around the true precision matrix, where the negative log-likelihood function $f_x$ is strongly geodesically convex.
This implies that $f_x$ is strongly geodesically convex in a small geodesic ball around the MLE.
Hence, any point with sufficiently small gradient of $f_x$ is contained in a sublevel set on which $f_x$ is strongly geodesically convex (\CREFsupp{lem:gradient-strong-convexity-fm}{Lemma~G.2}).
Such a point is found in polynomially many iterations of the flip-flop algorithm (\CREFsupp{lem:flip-flop-sinkhorn}{Lemma~G.5}).
Then, strong convexity implies that a $\delta$-minimizer is found in $O(\log(1/\delta))$ further iterations (\CREFsupp{lem:descent-sublevel-set}{Lemma~G.3}).
Thus we obtain the main technical result of this section:

\begin{theorem}[Convergence from initial conditions]\label{prop:meta}
Let $\Theta \in \SPD$ be our true precision matrix, $x_1,\dots,x_n \in \R^D$ our samples, $\lambda > 0$ and $0 < \zeta \leq \min \{ 1, 16 \sqrt{(k+1)(k-1)} / \lambda \}$ s.t.
\begin{enumerate}
\item $f_x$ is $\lambda$-strongly geodesically convex at any $\Theta'\in\SPD$ such that $\dop(\Theta', \Theta) \leq \zeta$.
\item $\abs{\nabla_0 f_x(\Theta)} \leq 1/2$.
\item The MLE $\htheta$ exists and satisfies $\dop(\htheta,\Theta) \leq \zeta/2$.
\end{enumerate}
Then, for every $0<\delta<\lambda\zeta/16\sqrt{(k+1)\dmax}$, the number of iterations $T$ needed for \cref{alg:flip-flop} to output $\otheta$ with $\dFR(\otheta_a, \htheta_a) \leq \sqrt{\frac{d_a}{2}} \cdot \frac{\delta}{\lambda}$ for all $a\in[k]$ is:

\begin{enumerate}
    \item when the initial guess is $\wttheta$ with $\nabla_0 f_x(\wttheta) = 0$,
    \begin{align*}
  T = O\left(\frac{k^2\dmax}{\zeta^2 \lambda^2} \cdot \dop(\wttheta, \htheta)
  + \frac{k}{\lambda} \log \left(\frac{\lambda\zeta}{\delta \cdot \sqrt{k\dmax}}\right) \right)
\end{align*}
    \item if the initial guess $\widetilde{\Theta}$ satisfies $\nabla_0 f_x(\wttheta) = 0$ and $\dop(\widetilde{\Theta}, \htheta) \leq \frac{\lambda \zeta}{100 \dmax \sqrt{k(k+1)}}$, then
\begin{align*}
  T = O\left( \frac{k}{\lambda} \log \left( \frac{\sqrt{k \dmax} \cdot \dop(\wttheta, \htheta)}{\delta} \right) \right) = O\left(k \log \frac{1}{\delta}\right)
\end{align*}
    \item with initial guess $\frac{1}{f_x(I_D)} \cdot I_D$,
    \begin{align*}
  T = O\left(\frac{k^2\dmax}{\zeta^2 \lambda^2} \cdot \log \kappa(\Theta)
  + \frac{k}{\lambda} \log \left(\frac{\lambda\zeta}{\delta \cdot \sqrt{k\dmax}}\right) \right)
\end{align*}
\end{enumerate}
\end{theorem}

With \cref{prop:meta} at hand, fast convergence of the flip-flop algorithm for both the matrix and tensor normal models follow simply by proving that the initial conditions above will be satisfied with high probability, given a high enough number of samples.
More precisely, we show that the sample complexity results of \cref{sec:sample-complexity} already imply the conditions of \cref{prop:meta}, thereby proving \cref{thm:tensor-flipflop,thm:matrix-flipflop}.

\begin{proof}[Proof of \cref{thm:tensor-flipflop}]
For $\lambda=\frac{1}{2}$, $0 < \zeta < 1$ a sufficiently small universal constant, and $r=\frac{\zeta}{\sqrt{(k+1)\dmax}}$, consider the following events (i.e., the conditions of \cref{prop:meta}):
\begin{enumerate}
\item $f_x$ is $\lambda$-strongly geodesically convex at any $\Theta'\in\SPD$ such that $\dop(\Theta', \Theta) \leq \zeta$.
In particular, $f_x$ is $\lambda$-strongly geodesically convex on the geodesic ball~$B_r(\Theta)$.
\item $\norm{\nabla f_x(\Theta)}_F < \frac{r\lambda}2$.
In particular, $\abs{\nabla_0 f_x(\Theta)} < \frac12$.
\item The MLE $\htheta$ exists and satisfies $d(\htheta,\Theta) \leq r/2$.
In particular, $\dop(\htheta,\Theta) \leq \zeta/2$.
\end{enumerate}
We first bound the success probability of these events similarly to the proof of \cref{thm:tensor-frobenius}.
For this, we may assume without loss of generality that $\Theta=I_D$ by \cref{remark:gradient-everywhere,remark:hessian-everywhere}.
Then the first event holds with probability at least $1 - k^2 \smash{( \frac{\sqrt{nD}}{k \dmax} )^{-\Omega(\dmin)}}$ by \cref{thm:ball-convexity,lem:op ball vs frob ball}, provided we choose~$C$ large enough and $\zeta$ small enough universal constants.
For the second event, we apply \cref{prop:gradient-bound} with
\begin{align*}
  \eps
= \frac1{10\sqrt{k}} \frac{r\lambda}2
= \frac \zeta{40\sqrt{k(k+1)\dmax}},
\end{align*}
which satisfies $\eps<1$ and $n \geq \frac{\dmax^2}{D \eps^2}$ provided we choose~$\zeta$ sufficiently small and $C$ sufficiently large universal constants.
With these choices, the second event holds with probability at least
\begin{align*}
  1 - 2(k+1)e^{-\eps^2 \frac{nD}{8\dmax}}
= 1 - k \, e^{-\Omega(\frac{n D}{k^2\dmax^2})}.
\end{align*}
Thus, the two events hold simultaneously with the desired success probability by the union bound.
Moreover, by \cref{lem:convex-ball}, the events 1 and 2 together also imply event 3.
The above shows that the conditions of \cref{prop:meta} are satisfied.
Thus, the iteration complexity of \cref{alg:flip-flop} follows from \cref{prop:meta}.
\end{proof}


\begin{proof}[Proof of \cref{thm:matrix-flipflop}]
Consider the events below for constants~$\lambda,\zeta \in (0,1)$:
\begin{enumerate}
\item $f_x$ is $\lambda$-strongly geodesically convex at any $\Theta'\in\SPD$ such that $\dop(\Theta', \Theta) \leq \zeta$.
\item $\abs{\nabla_0 f_x(\Theta)} \leq 1/2$.
\item The MLE $\htheta$ exists and satisfies $\dop(\htheta,\Theta) \leq \zeta/2$.
\end{enumerate}
To bound the success probability of these events,
we may assume without loss of generality that $\Theta=I_D$ by \cref{remark:gradient-everywhere,remark:hessian-everywhere}.
We will also assume that $d_1 \leq d_2.$

If $\lambda\in(0,1)$ is a suitable universal constant, $C$ is a large enough universal constant, and $\zeta$ is a small enough universal constant, by \CREFsupp{cor:matrix-convexity}{Corollary~E.5} with $t^2 = n d_1/  d_2$, the first event holds with probability at least~$1 - e^{-\Omega(n d_1)} \geq 1 - e^{-\Omega( n d_1^2/ d_2 \log^2 d_1)}$ in view of our assumption on~$n$.

The second event holds with probability at least~$1-e^{-\Omega(nD)}$ by \CREFsupp{prp:xnorm}{Proposition~D.2}.
Finally, by \cref{thm:matrix-normal} with $t^2 = n d_1/  d_2 \log^2 d_1$ (which can be made larger than $1$ by our assumption on $n$ assuming $C$ is large enough), the third event holds with probability at least~$1 - e^{-\Omega(n d_1^2/ d_2 \log^2 d_1)}$.

Event 3 follows from \cref{thm:matrix-normal} via
the fact that $\dop(\smash{\htheta}_1 \ot \smash{\htheta}_2, \Theta_1 \ot \Theta_2) \leq \dop(\smash{\htheta}_1, \Theta_1) + \dop(\smash{\htheta}_2, \Theta_2)$.
Thus, with probability at least~$1 - e^{-\Omega(n d_1^2/ d_2 \log^2 d_1)}$ all three events hold simultaneously, by the union bound, meaning the conditions of \cref{prop:meta} are satisfied.
Thus, the iteration complexity of \cref{alg:flip-flop,alg:flip-flop matrix} follows from \cref{prop:meta}.
\end{proof}

\section{Conclusion and open problems}

In this work, we almost optimally address the fundamental question of parameter estimation for the matrix and tensor normal models, as well as the question of efficient computation of this estimator.
Contrary to the state of the art for unstructured covariance estimation (i.e., $k = 1$), all previous existing results (in their sample complexity bounds as well as the error rates and guarantees of their estimators) depended on the \emph{condition number} of the true covariance matrices and on a \emph{sufficiently accurate starting guess}, and therefore had suboptimal guarantees in the general case.
By proving strong convexity in the geometry induced by the Fisher information metric, we remedy these issues and obtain nearly optimal estimates (\emph{without} dependence on condition number) in the strongest possible metrics, namely the Fisher-Rao and Thompson distances.
As a consequence, we also control other equivariant statistical distances such as relative entropy and total variation distance.

In particular, we showed that the maximum likelihood estimator (MLE) for the covariance matrix in the matrix normal model has optimal sample complexity up to logarithmic factors in the dimensions.
We showed that the MLE for tensor normal models with a constant number of tensor factors has optimal sample complexity in the regime where it is information-theoretically possible to recover the covariance matrix to within a constant Frobenius error.
Whenever the number of samples is large enough for either of the aforementioned statistical results to hold, we show that the flip-flop algorithm converges to the MLE exponentially quickly.
Hence, the output of the flip-flop algorithm with $O\left(\dmax(1 + \log \kappa(\Theta)) +  \log n \right)$ iterations (see the discussion after \cref{thm:matrix-flipflop}) is an efficiently computable estimator with statistical guarantees comparable to those we show for the MLE.

Our main open question is whether the sample threshold requirement $n = \Omega(k^2 \dmax^3/ D)$ for \cref{thm:tensor-frobenius} can be weakened to $n = \Omega(k^2 \dmax^2/ D)$ for $k \geq 3$.
Equivalently, do the guarantees of \cref{thm:tensor-frobenius} hold even when one cannot hope to estimate the Kronecker factors to constant Frobenius error, but only to constant \emph{operator norm} error?
In the case~$k = 1$ (i.e., unstructured covariance estimation) the weaker assumption is well-known to suffice, and for~$k = 2$ the same follows (up to logarithmic factors) by our \cref{thm:matrix-normal}.
Filling in this gap will place the tensor normal model on the same sound theoretical footing as unstructured covariance estimation.

\begin{acks}[Acknowledgments]
CF acknowledges Ankur Moitra for interesting discussions and Shuheng Zhao for sharing code for her Gemini estimator.
All authors would like to thank the anonymous reviewers for their reviews and suggestions.

MW is also affiliated with the Faculty of Computer Science of Ruhr-Universit\"at Bochum and the Korteweg-de Vries Institute for Mathematics and QuSoft at the University of Amsterdam.
\end{acks}

\begin{funding}
AR and MW acknowledge support by the Dutch Research Council (NWO grant OCENW.KLEIN.267).
MW furthermore acknowledges supported by the European Union (ERC Grant Agreement No.~101040907) and the German Federal Ministry of Research, Technology and Space (QuBRA, 13N16135; QuSol, 13N17173).
\end{funding}

\ifdefined\ARXIV
\ifdefined\ARXIV
\newcommand{\CREFmain}[2]{\cref{#1}}
\else
\newcommand{\CREFmain}[2]{\cite[#2]{FORW25}}
\fi

\begin{appendix}
\ifdefined\ARXIV
\medskip
\begin{center}
\Large\textsc{Supplemental Material}
\end{center}
\fi
\numberwithin{table}{section}

\section{Error metrics and statistical distances}\label{app:rel-error}

As we discussed in \CREFmain{section: intro}{Section~1}, the choice of error metric depends on the downstream goal of the statistical problem, and so each goal may necessitate a different error measure.

We will now discuss some of the most common measures for our Gaussian estimation setting.
We also discuss the relation between these error measures, showing that the geodesic metric (\CREFmain{subsec:geom}{Section~2.1}) that we use in our results is locally equivalent to many natural notions of statistical and matrix error, and therefore is most well-suited to our problem setting.
On the other hand, the `absolute' metric used by previous works is well-suited to exploit combinatorial and structural properties, but can be quite different from the relevant statistical and geometric error metrics in the general case.

Throughout this section we will use $\alpha \lesssim \beta$ to indicate that there is a constant $C > 0$ such that $\alpha \leq C \beta$, and $\alpha \asymp \beta$ to denote that both $\alpha \lesssim \beta$ and $\beta \lesssim \alpha$ hold.

\subsection{Fisher-Rao and Thompson metrics}

When the parameter space is given by a smooth manifold (as in the case of matrix and tensor normal models), Chentsov's Theorem~\cite[Theorem 3]{cencov1978algebraic} states that the Fisher information metric is the \emph{unique Riemannian metric} which is \emph{invariant} under any information-theoretically relevant transformation on the input data.\footnote{This result is the purview of the field of information geometry, and we point the reader to the following text for more details \cite{AmariInfoGeo}.}
Due to the \emph{geometric} and \emph{statistical} properties of the Fisher information metric, in this work, we mainly focus on the Fisher-Rao and Thompson distances, which arise from the Riemannian structure induced by the Fisher information metric.
These are statistically motivated metrics that are also `linearly-invariant' in a technical sense that we will explain below.
This invariance, along with the geometry of the parameter space, are key to our analysis and allows to prove bounds independent of the condition number of the true parameter value.
We are able to give stronger bounds in the dense case in these tighter metrics that do not depend on the condition numbers of the true parameters.

The Fisher-Rao metric is a Riemannian structure that can be placed on any abstract sufficiently smooth parameter manifold.
Formally, the Fisher-Rao information matrix is the Hessian of the population log-likelihood function:
\[ H_{\theta} := \nabla_{\theta}^{2} \int_{x \in X} \log p_{\theta}(x) d p_{\theta}(x)
= \nabla_{\theta}^{2} \E_{x \sim p_{\theta}}[\log p_{\theta}(x)] ,    \]
where $p_{\theta} \in \mathcal{P}$ is an element of the statistical model over parameter space $\Theta$.
The metric induced by this Riemannian structure is known as the \emph{Fisher-Rao metric}.
As we will show in the rest of this section, the Fisher-Rao metric is intimately connected with a variety of statistical error measures.
More precisely, in \cref{lem:rel op distances}, we will see that the distances arising from the Fisher information metric are locally equivalent to the standard distributional error measures (total variation, relative entropy), which we define in \cref{definition:KLTV}.
Thus, by providing good distance bounds in terms of the Fisher information metric, we are also giving good bounds on \emph{all} such relative distances!

For readability, we now recall the Fisher-Rao and Thompson metrics in the Gaussian covariance estimation setting.\footnote{The Thompson metric is the `operator norm' version of the Fisher-Rao metric, and is not induced by an inner product.
The Thompson metric can be used for spectral applications such as PCA.
For more background, see \cite{Snyder16}.}

\begin{definition}[Fisher-Rao and Thomson distances]
The Fisher-Rao distance for centered Gaussians parameterized by their precision matrices is given by
\begin{align*}
  \dFR(\htheta, \Theta) = \frac{1}{\sqrt{2}} \norm{\log \Theta^{-1/2} \htheta \Theta^{-1/2}}_F.
\end{align*}
The Thompson distance is given by
\begin{align*}
  \dop(A, B) := \norm{\log B^{-1/2} A B^{-1/2}}_{\op}.
\end{align*}
\end{definition}

Before moving on to the other error metrics, we state some simple properties that are useful for our analysis.

\begin{fact} \label{f:dFRdopinvarianceAndTensor}
    For $A,B \in \PD(d)$, the Fisher-Rao and Thompson metrics satisfy
    \begin{enumerate}
        \item Inverse:
        \[ \dFR(A^{-1}, B^{-1}) = \dFR(A,B); \quad \text{and} \quad
        \dop(A,B) = \dop(A^{-1}, B^{-1}) ; \]
        \item Symmetry:
        \[ \dFR(A,B) = \dFR(B,A) ; \quad \text{and} \quad \dop(A,B) = \dop(B,A) ; \]
        \item Invariance: For invertible $X \in \R^{d \times d}$
        \[ \dFR(X A X^{T}, X B X^{T}) = \dFR(A,B); \quad \text{and} \quad \dop(X A X^{T}, X B X^{T}) = \dop(A,B) ; \]
        \item Tensor: For $A = \otimes_{a \in [k]} A_{a}, B = \otimes_{a \in [k]} B_{a}$ with $A_{a}, B_{a} \in \PD(d_{a})$,
    \[ \dop(A,B) \leq \sum_{a \in [k]} \dop(A_{a}, B_{a})
     , \]
    \[ \det(A) = \det(B) \implies \dFR(A,B)^{2} = \sum_{a \in [k]} \frac{D}{d_{a}} \dFR(A_{a},B_{a})^{2} .  \]
    \end{enumerate}
\end{fact}
\begin{proof}
    We first note that both $\dop, \dFR$ metrics between $A,B$ depend only on the spectrum of $B^{-1/2} A B^{-1/2}$, which is equivalent to the spectrum of $B^{-1} A$.
    \begin{enumerate}
        \item Inverse: note $(A^{-1} B)^{-1} = B^{-1} A$, so the spectrum of $\log A^{-1/2} B A^{-1/2}$ is the negative of the spectrum of $\log B^{-1/2} A B^{-1/2}$. The statement follows by definition of $\dop, \dFR$ depend only on the spectrum.
        \item Symmetry holds by the same argument.
        \item Invariance: we again calculate
        \[ (X B X^{T})^{-1} (X A X^{T}) = X^{-T} B^{-1} X^{-1} X A X^{T} = X^{-T} B^{-1} A X^{T} ,     \]
        which has the same spectrum as $B^{-1} A$.
        \item Tensor: Recall $\log (X \otimes Y) = \log X \otimes I + I \otimes \log Y$. The statement follows by definition of the operator norm:
        \begin{align*}
        \dop(A,B) & = \| \log B^{-1/2} A B^{-1/2} \|_{\op}
         = \bigg\| \sum_{a \in [k]} I_{\overline{a}} \otimes \log B_{a}^{-1/2} A_{a} B_{a}^{-1/2}  \bigg\|_{\op}
       \\ & \leq \sum_{a \in [k]} \| I_{\overline{a}} \otimes \log B_{a}^{-1/2} A_{a} B_{a}^{-1/2} \|_{\op} = \sum_{a \in [k]} \dop(A_{a}, B_{a}) ,
        \end{align*}
        where $I_{\overline{a}}$ is the identity on the remaining factors.
        Note that the inequality is not always tight as it depends on the sign of the largest eigenvalue of each tensor factor.

        For the Frobenius norm, the assumption implies that we can scale the factors so that $\det(B_{a}^{-1/2} A_{a} B_{a}^{-1/2}) = 1$ for all $a \in [k]$.
        A similar calculation gives
        \begin{align*}
        \dFR(A,B)^{2} & = \| \log B^{-1/2} A B^{-1/2} \|_{F}^{2}
         = \bigg\| \sum_{a \in [k]} I_{\overline{a}} \otimes \log B_{a}^{-1/2} A_{a} B_{a}^{-1/2}  \bigg\|_{F}^{2}
       \\ &  = \sum_{a \in [k]} \| I_{\overline{a}} \|_{F}^{2} \| \log B_{a}^{-1/2} A_{a} B_{a}^{-1/2} \|_{F}^{2}
        = \sum_{a \in [k]} \frac{D}{d_{a}} \dFR(A_{a}, B_{a}) ,
        \end{align*}
        where in the third step the $\det$ assumption implies $\langle I_{a},  \log B_{a}^{-1/2} A_{a} B_{a}^{-1/2} \rangle = 0$ so all terms in the sum are orthogonal. \qedhere
    \end{enumerate}
\end{proof}

\subsection{Relative error}

We now define the notion of relative error, which will be useful to locally approximate the Fisher-Rao (and Thompson) distance to distributional error measures.

\begin{definition}[Relative error]\label{dfn:relative-error}
For positive definite matrices $A, B$, define their \emph{relative Frobenius error} as
\begin{align}\label{eq:def D_F}
\DF(A \Vert B) = \norm{I - B^{-1/2} A B^{-1/2}}_F.
\end{align}
Similarly, define the \emph{relative spectral error} as
\begin{align}\label{eq:def D_op}
\Dop(A \Vert B) = \norm{I - B^{-1/2} A B^{-1/2}}_{\op}.
\end{align}
\end{definition}

\begin{remark}\label{remark: relative error and fisher-rao and thompson}
Note that by Taylor approximation ($|e^{x} - 1| \simeq x$ for $|x| \leq 1$), we can see that $\DF \approx \dFR$ and $\Dop \approx \dop$ when $\dop \leq 1$.
\end{remark}


An important property of the relative error notions from \cref{dfn:relative-error} are a \emph{local} triangle inequality, stated in \cite{yang1999information}, and approximate symmetry.

\begin{lemma}\label{lem:triangle-ineq}
Let $A, B, C \in \PD(d)$.
Let $D \in \{\Dop, \DF\}$. Provided $D(A\Vert B), D(B\Vert C)$ are at most an absolute constant $c>0$, we have
\begin{align}
D(A\Vert C) &= O\bigl( D(A\Vert B) + D(B\Vert C) \bigr),\label{eq:tri}\\
D(B\Vert A) &= O\bigl( D(A\Vert B) \bigr),\label{eq:sym} \text{ and }\\
D(A^{-1}\Vert B^{-1}) &= O\bigl( D(A\Vert B) \bigr).
\end{align}
\end{lemma}

\begin{proof} The approximate triangle inequality follows as
\begin{align*}
    D(A \Vert C) & = \|A^{-1/2}(A - C) A^{-1/2} \|
    \leq \|A^{-1/2}(A - B) A^{-1/2} \| + \|A^{-1/2} (B - C) A^{-1/2} \|
    \\ & \leq \|I - A^{-1/2} B A^{-1/2} \| + \|A^{-1/2} B^{1/2}\|_{\op}^{2} \| B^{-1/2} (B - C) B^{-1/2} \|
    \\ & \leq D(A \Vert B) + (1 + D_{\op}(A \Vert B)) D(B \Vert C) .
\end{align*}
For the symmetry and inverse properties, we can apply the same properties for $\dFR, \dop$ as shown in \cref{f:dFRdopinvarianceAndTensor}, and then combine with Taylor approximation for $D(A,B) \leq c$.
\end{proof}

\subsection{Distributional error}

We now consider error measures that directly compare the underlying distributions corresponding to given parameters.
In particular, the relative entropy and total variation distances are the most important and well-studied for statistical purposes.

\begin{definition}[Relative Entropy]\label{definition:KLTV}
    Given distributions $p,q$ on measure space $\mathcal{M}$, the Kullback-Leibler divergence and total variation distance are
    \[ \DKL(p \, \Vert \, q) := \int_{x \in \mathcal{M}} \log \frac{p(x)}{q(x)} d p(x) , \]
    \[ \DTV(p,q) := \frac{1}{2} \int_{x \in \mathcal{M}} |p(x) - q(x)| dx .  \]
\end{definition}

It turns out that the KL divergence is intimately related to the Fisher-Rao metric defined above.
In fact, we can re-derive the Fisher-Rao information matrix as the Hessian of the KL divergence.\footnote{Indeed this is the perspective taken as the starting point of information geometry \cite{AmariInfoGeo}.}
This implies that for sufficiently close distributions, the Fisher-Rao and relative entropy metrics are equivalent. As we show, in the Gaussian setting, this remains true for an explicitly bounded distance.


In the Gaussian setting, both of these statistical measures measures are linearly invariant, which can be proven by a simple change of variable.
In fact, the relative entropy between $N(0,\Theta_{1}^{-1})$ and $N(0,\Theta_{2}^{-1})$ can be given as an explicit expression depending only on the eigenvalues of the `relative' matrix $\lambda := \text{spec}(\Theta_{1}^{-1} \Theta_{2})$:
\begin{align} \label{eq:DKLCalculation}
\DKL\bigl(\mathcal{N}(0, \Theta_1^{-1}) \Vert \mathcal{N}(0, \Theta_2^{-1})\bigr)
&= \frac{1}{2} \tr \Theta_1^{-1} \Theta_2 - \frac{1}{2}\log\det(\Theta_1^{-1} \Theta_2) - \frac{d}{2}\\
& = \frac{1}{2} \sum_{i =1}^d (\lambda_i -1 -  \log \lambda_i) . \nonumber
\end{align}
A simple Taylor approximation relates this explicitly to the Fisher-Rao metric in \cref{lem:rel op distances}.

Unlike relative entropy, there is no explicit expression for the total variation between two Gaussian distributions just in terms of covariance matrices. But we can still use the linear invariance property to effectively relate it to the previous measures discussed.
Indeed, Pinsker's inequality gives the following bound for general distributions:
    \[ 2 \DTV(p,q)^{2} \leq \DKL(p \, \Vert \, q) .  \]
Therefore, any bound on relative entropy automatically implies a bound on total-variation. Of course, $\DTV$ is always between $0$ and $1$, so this bound is vacuous $\DKL$ is larger than $2$.
In the Gaussian setting, it turns out that when it is not vacuous, Pinsker's inequality gives a two-sided equivalence between relative entropy and total variation.

We now show that all of the relative error measures so far discussed are locally equivalent, i.e. if one of them is small then all of them are the same up to absolute constant factors.

\begin{lemma}[Relationships between dissimilarity measures]\label{lem:rel op distances}
There exists a constant $c > 0$ such that the following holds.
If $\dop(\Theta_{1} \Vert \Theta_{2}) \leq c$, then
\begin{align*}
    \dFR(\Theta_1\Vert  \Theta_2)^{2} \asymp \DKL\bigl(\mathcal{N}(0, \Theta_1^{-1}) \Vert  \mathcal{N}(0, \Theta_2^{-1})\bigr).
\end{align*}
Further, if any of
$\dFR(\Theta_1\Vert  \Theta_2)$,
$\DTV(\mathcal{N}(0, \Theta_1^{-1}), \mathcal{N}(0, \Theta_2^{-1}))$,
$\DKL(\mathcal{N}(0, \Theta_1^{-1}) \Vert  \mathcal{N}(0, \Theta_2^{-1}))$,
is at most~$c$, then
\begin{align*}
\dFR(\Theta_1\Vert  \Theta_2)
\asymp \DTV\bigl(\mathcal{N}(0, \Theta_1^{-1}), \mathcal{N}(0, \Theta_2^{-1})\bigr)
&\asymp \sqrt{\DKL\bigl(\mathcal{N}(0, \Theta_1^{-1}) \Vert  \mathcal{N}(0, \Theta_2^{-1})\bigr)}.
\end{align*}
\end{lemma}
\begin{proof}
By the Taylor approximation in \cref{remark: relative error and fisher-rao and thompson}, it is enough to relate the relative distance $\DF$ to $\DTV$ and $\DKL$.

To relate $\DF$ to relative entropy, we follow the calculation above in \cref{eq:DKLCalculation} and note that $\lambda - 1 -\log\lambda \asymp \frac12 |\log \lambda|^2$ on $[1/2, 3/2]$.
To complete the argument, choose $c$ small enough that $\frac{1}{2}(\lambda - 1 -\log\lambda) \leq c$ implies $\lambda \in [1/2, 3/2]$.

The relationship between $\DF$ and $\DTV$ comes from bounds in
\cite[Theorem 1.8]{arbas2023polynomial}:
\begin{equation*}
\frac{1}{200}  \leq \frac{\DTV\bigl(\mathcal{N}(0, \Theta_1^{-1}), \mathcal{N}(0, \Theta_2^{-1})\bigr)}{\DF(\Theta_1 \Vert  \Theta_2)} \leq \frac{1}{\sqrt{2}}.
\end{equation*}
\end{proof}

Note that the equivalence is `local' in the sense that we require one of the error measures to be small in order for them to be equivalent.
But this is the relevant case as $\DTV$ is always $\leq 1$ and the goal of statistical estimation is to compute an estimator with (vanishingly) small error given sufficiently many samples.

\subsection{Absolute error}

In this last subsection, we discuss `absolute' measures of error $\|\Theta - \Theta'\|_{F}$ and $\|\Theta - \Theta'\|_{\op}$. We can also consider normalized versions, where this is divided by $\|\Theta\|_{\op}$.

Prior works (\cite{RBLZ08,Cai2016,tsiligkaridis2013convergence,zhou2014gemini,Lyu2020Tlasso}) were motivated by graphical model estimation, i.e. understanding the support structure of the covariance and inverse covariance matrices.
While ostensibly the most natural, absolute error measures do not enjoy many of the geometric and statistical properties discussed above.
The following gives a simple relation to our relative notions of error:

\begin{prop} \label{prop:absRelation}
For $A, B \in \PD(d)$ with condition number $\kappa(B) := \frac{\lambda_{\max}(B)}{\lambda_{\min}(B)}$,
\[ \kappa(B)^{-1} \DF(A \Vert B) \leq \frac{\|A - B\|_{F}}{\|B\|_{\op}} \leq \DF(A \Vert B) \]
\[ \kappa(B)^{-1} \Dop(A \Vert B) \leq \frac{\|A - B\|_{\op}}{\|B\|_{\op}} \leq \Dop(A \Vert B)   . \]
\end{prop}

By \cref{remark: relative error and fisher-rao and thompson}, this also gives a similar relation to Fisher-Rao and Thompson metrics when the quantities are small enough.

Prior works gave improved results for sparse inputs in the Frobenius and operator norm, but their bounds also depend on the condition number of the true parameter.
This can lead to improved statistical guarantees, but only when the condition number is small.
Indeed, by the above proposition combined with \cref{lem:rel op distances},
absolute error and distributional error only match when the condition number of the parameter is small.
Thus such analyses come with an inherent trade-off between exploiting structural properties and allowing the most general parameter space (where the input could have arbitrarily large condition number). These previous results are discussed in more detail in \cref{subsec:matrix-tensor-previous-work}.

Note that it is difficult to prove bounds that depend on sparsity for any linearly-invariant measure (such as $\dFR, \dop$ as shown in \cref{f:dFRdopinvarianceAndTensor}), as sparsity is not preserved under linear transformations.
It is an intriguing open question to see if there is some estimator that achieves optimal error rates in the statistical or relative sense, that improves with sparsity, but does not depend on the condition number.

\section{Previous works}\label{appendix: previous works}

We begin with a summary on the contributions of previous works, and then provide a more detailed comparison for the interested reader.

\subsection{Summary of previous works}\label{subsection: previous works}

A great deal of research has been devoted to estimating the covariance matrix for the matrix and tensor normal models, but gaps in rigorous understanding remain.
Empirical works on the matrix and tensor normal models (\cite{mardia1993spatial,dutilleul1999mle,brown2001bayesian}) have proposed an alternating minimization algorithm, known as the \emph{flip-flop algorithm}, to compute the maximum likelihood estimator (MLE).
This can be justified by noting that, while the negative log-likelihood function for the tensor normal model is not convex as a function of the candidate precision matrices $\otheta_{1}, \dots, \otheta_{k}$, it becomes convex if we fix all but one of these matrices and optimize over the remaining matrix.
Therefore, the flip-flop algorithm can be seen as iteratively minimizing one parameter matrix at a time in order to approach the MLE.
\cite{werner2008estimation} was the first work to provide a rigorous guarantee on the MLE, with respect to the true covariance, along with a guarantee on the performance of the flip-flop algorithm.
In particular, they consider the \emph{asymptotic regime}, where the number of samples tends to infinity.
They show that both the MLE and the third iteration of the flip-flop algorithm are asymptotically consistent, meaning that as $n \to \infty$ both these estimators converge to the true covariances.
They further show both of these estimators are asymptotically normal and give explicit expressions for the expected deviation of these estimators from the true values as~$n \to \infty$.

The work~\cite{tsiligkaridis2013convergence} gives a quantitative analysis for the finite sample setting: for the matrix normal model, the three-step Flip-Flop estimator $\hat{\Theta}$ has sample threshold $n \geq \Tilde{O}(\max\{ d_{1}, d_{2}\} )$ and error rate $\Tilde{O} ( \sqrt{ \frac{d_{1}^{2} + d_{2}^{2} }{n}} )$.
Indeed, they claim this rate holds for all iterations of the Flip-Flop algorithm when the true precision has constant condition number.
Unfortunately, as we will see in \cref{subsec:matrix-tensor-previous-work}, the hidden constants (both in the sample complexity and in the error rate) depend polynomially on the condition number.

Apart from the above works on the MLE and on the analysis of the flip-flop algorithm, other works have proposed different estimators for the matrix and tensor normal models.
The main idea in this other line of works is to reduce the estimation problem of the matrix and tensor normal models (i.e. $k \geq 2$) to $k$ instances of the Gaussian estimation problem (i.e. $k = 1$).
To understand the approach of these works, we now give a high level overview of the Gaussian setting, and then discuss on a high level how the Gaussian approach is generalized to the matrix and tensor normal models.

In unstructured covariance matrix estimation, i.e. $k =1$, with covariance of dimension $d \times d$, it is well-known that the MLE exists almost surely iff $n \geq d$ and achieves minimax optimal error rates of $\sqrt{d^2/n}$ in relative Frobenius norm and $\sqrt{d/n}$ in relative operator norm, respectively.
On the other hand, there are many situations of interest where the dimension of the data $d$ is comparable to or larger than the number of samples $n$.
This fact is the starting point for a vibrant area of research attempting to estimate the covariance or precision matrix with fewer samples under structural assumptions.
Particularly important is the study of graphical models, which seeks to infer the support of the precision matrix under the assumptions that it is \emph{sparse} (has few nonzero entries) and that is has \emph{small condition number}.
In this setting, \cite{RBLZ08,CLIME, Cai2016} have obtained both sample complexity upper and lower bounds, respectively.
However, unlike the unstructured Gaussian estimation case, these works obtain estimates in absolute error, instead of relative error.
For more details on these works, see \cref{subsec:gaussian-previous-work}.

In the settings of matrix and tensor normal models, i.e. $k \geq 2$, it is much less clear what the estimator should be.
Indeed, the sample covariance and precision matrices, with high probability, will not be of the desired tensor form.
Thus, estimating the tensor factors from the samples is a much more difficult task.
Previous works have generalized the techniques used in the structured (i.e., sparse, constant condition number) Gaussian estimation to the settings of structured matrix and tensor normal models (see \cite{tsiligkaridis2013convergence,zhou2014gemini,Lyu2020Tlasso}).
However, the need to simultaneously estimate all the Kronecker factors imposes several new challenges, which were addressed in these works upon extra (strong) assumptions.
In addition to the sparse and \emph{constant condition number} assumptions, the aforementioned works have also assumed knowledge of an \emph{initial guess} which is \emph{sufficiently close to the true Kronecker factors}.


Under these assumptions, the above works have proposed to iteratively apply Gaussian estimators to each Kronecker factor, analyzing the convergence of their estimators to the true Kronecker factors in terms of \emph{absolute error}.
When the condition number is constant, these estimators achieve minimax optimal rates in terms of absolute error.
However, it is important to note that their sample complexity bounds all have multiplicative factors that \emph{depend polynomially on the condition number} of the \emph{true precision matrices}, as well as on the \emph{distance between their initial guess and the true precision matrices}.
As we explore in more detail in \cref{subsec:matrix-tensor-previous-work}, this dependence on the condition number and on the quality of the initial guess negates all benefits of exploiting sparsity as soon as the condition number or the distance of the initial guess to the true precision matrix is in the order of the square root of the dimension of the largest Kronecker factor.
Moreover, in the setting where condition number is large, absolute error no longer approximates statistical distance between two Gaussian distributions (see \cref{app:rel-error}).
For more details on the results on sample complexity and error bounds from the above works, see \cref{subsec:matrix-tensor-previous-work}.
For details on the complexity of their estimators, see \cref{subsec:complexity-previous-work}.

In a different direction, \cite{derksen2020matrix,derksen2022maximum} determined the precise number of samples for the MLE to (almost surely) exist and be unique, but their algebraic techniques do not give any guarantees on the goodness of this estimator.

\subsection{Detailed comparison with previous works}
We now summarize the main results of our article \ifdefined\ARXIV\else\cite{FORW25} \fi and then we proceed to have a more in-depth discussion of the results from previous works along with a more detailed comparison.

\begin{enumerate}
    \item \textit{Result: Nearly optimal sample complexity bounds for the matrix and tensor normal models.}
    \begin{enumerate}
        \item Our estimator works with \emph{provably minimal assumptions}, and our bounds are \emph{independent} of any properties of the distribution, such as condition number;
        \item We prove that the MLE, the most natural estimator, achieves the above bounds.
    \end{enumerate}

    \vspace{5pt}

    \noindent The MLE has been previously studied for the matrix and tensor normal model.
Namely, \cite{werner2008estimation} show asymptotic consistency and efficiency of the MLE, i.e. that the error goes to $0$ and the variance is optimal in the limit $n \to \infty$.
Also \cite{tsiligkaridis2013convergence} give error guarantees for a finite number samples.
However, their sample threshold and error becomes unbounded in the general probabilistic model where the condition number could be unbounded.
In this work, we give the first finite guarantees for the MLE in the most general model without any assumption on the condition number.

     Other works (\cite{tsiligkaridis2013convergence}, \cite{zhou2014gemini}, \cite{Lyu2020Tlasso}) have proposed different estimators with various guarantees; but these results crucially require the following assumptions: (1) the precision matrices are known to be sparse; (2) the \emph{condition number} of the factors are bounded by some fixed constant; and (3) there is a sufficiently accurate \emph{initial guesses} available for each of the factors $\bar{\Theta}_{a} \approx \Theta_{a}$.
        The quantitative guarantees of these estimators depend quite heavily on these assumptions, so while they can be in principle be relaxed, the bounds will degrade substantially. Indeed, for the most general model where the precision matrix is arbitrary, all previous works give \emph{no finite bounds} for the sample threshold or error.
    Stated another way, the estimator proposed by these previous works is not just a function of the data, but is also a function of these initial guesses, and furthermore the results as stated do not hold unless these guesses are sufficiently accurate.

Our work instead analyzes the MLE, which is solely a function of the data and does not require any initial guesses in its definition.
Unlike the aforementioned previous works, our error rates are independent of both \emph{condition number} and any \emph{inital guess}.
However, our error guarantees for the MLE do not improve with sparsity, unlike the estimators from prior works.
A detailed comparison can be found in \cref{subsection: previous works} and \cref{appendix: previous works}.

    \item \textit{Result: New lower bounds for tensors beyond the Gaussian setting.}

    The sample complexity bounds for the classical Gaussian estimation setting ($k=1$) are well-known:
    the sample threshold is $n \gtrsim d$, and the error rate is $\sqrt{d^{2}/n}$ and $\sqrt{d/n}$ with respect to $\dFR$ and $\dop$ respectively. Further, these bounds are known to be tight up to constant factors as these rates are achieved by the MLE.
    This immediately implies lower bounds of $\dFR \gtrsim \sqrt{d_{a}^{3}/nD}$ and $\dop \gtrsim \sqrt{d_{a}^{2}/nD}$ for each factor of the matrix and tensor normal by considering the special case where $\Theta_{2} = I_{d_{2}}$.
    Indeed, given $n$ samples from the matrix normal model $X_{1}, ..., X_{n} \sim \cN(0, \Theta_{1}^{-1} \otimes I_{d_{2}})$, the columns correspond exactly to $N := n d_{2}$ independent samples from $\cN(0,\Theta_{1}^{-1})$.

    Our results show that the MLE matches this error rate for the largest tensor factor. Similarly, previous works on the matrix and tensor normal model (\cite{tsiligkaridis2013convergence}, \cite{zhou2014gemini}, \cite{Lyu2020Tlasso}) analyze estimators for the sparse setting, showing that they can estimate each tensor factor with error rate matching the known lower bounds for sparse Gaussian estimation \cite{Cai2016} (albeit with the additional assumptions discussed above).

    Both of these previous lower bounds come from the simpler Gaussian estimation problem.
    Our new lower bound in \CREFmain{sec:lower}{Section 4} shows that estimating each precision factor of the matrix and tensor normal model is strictly harder than separate instances of Gaussian estimation.

    \item \emph{Result: We prove that in the above sample regimes, the \emph{flip-flop} algorithm quickly converges to the MLE, and thereby to the true covariance matrices.
    Our analysis also works for any geodesic descent method to compute the MLE.}

    \vspace{5pt}
    \noindent Prior estimators (\cite{tsiligkaridis2013convergence}, \cite{zhou2014gemini}, \cite{Lyu2020Tlasso}) use techniques from sparse precision estimation for the Gaussian setting.
    Concretely, each iteration requires a solution to a linear program \cite{zhou2014gemini} or a convex program \cite{tsiligkaridis2013convergence, Lyu2020Tlasso}.
    While these are somewhat structured programs, they are still quite computationally intensive to solve, either requiring high polynomial overhead for large inputs or very slow convergence.
    In fact, the output of these programs is used iteratively to compute subsequent tensor factors. This can be quite computationally intensive as they must be solved to high accuracy.
    The estimator of \cite{XZG17} uses a truncated gradient descent method, but due to the use of sample splitting it cannot run for many iterations, which hurts its convergence properties.

Our solution is to analyze the natural flip-flop algorithm and to prove that it efficiently approximates the MLE \emph{from any given starting guess}.
This is a significantly faster procedure, as
each iteration requires a single matrix inversion.
Convergence of the Flip-Flop procedure was studied in \cite{werner2008estimation} in the asymptotic setting, and in \cite{tsiligkaridis2013convergence} in the restricted condition number setting.
Our work is the first to give convergence guarantees with finite samples in the most general probabilistic model.
More precisely, we show that flip-flop has linear convergence when it is sufficiently close to the MLE, and has polynomial convergence outside of this region. This also explains the experimental results given in previous works (e.g. \cite{tsiligkaridis2013convergence, zhou2014gemini, Lyu2020Tlasso}) showing fast convergence of the flip-flop method for many datasets in practice.
For a detailed runtime analysis, see \CREFmain{thm:tensor-flipflop}{Theorem 1.13}.
For a detailed comparison with previous works, see \cref{subsec:complexity-previous-work}.
\end{enumerate}
For a detailed summary on the qualitative and quantitative improvements of our work over previous works, we refer the reader to \cref{table: main results} and \cref{table: guesses effect} for the sample complexity comparisons.
For comparisons on the computational complexity of the proposed estimators, we refer the reader to \cref{table: estimator performance}.
A simplified version of the above tables (for a natural setting of the parameters) are shown in \CREFmain{table: main results intro}{Table 1} and \CREFmain{table: estimator performance no assumptions intro}{Table 2}, after we formally state our main results.

\vspace{5pt}

We now give a more detailed description of previous works on the matrix and tensor normal models, both on sample complexity and error bounds, as well as on the complexity of previously proposed estimators.
To give some perspective on the settings studied and the assumptions made by previous works, we first describe the classical Gaussian estimation setting, i.e. $k = 1$.

As in the previous section, throughout this section we will use $\alpha \lesssim \beta$ to indicate that there is a constant $C > 0$ such that $\alpha \leq C \beta$, and $\alpha \asymp \beta$ to denote that both $\alpha \lesssim \beta$ and $\beta \lesssim \alpha$ hold.

\subsection{Gaussian estimation}\label{subsec:gaussian-previous-work}

In the Gaussian setting, given samples $X_{1}, ..., X_{n} \sim N(0,\Sigma)$, we would like to estimate the covariance $\Sigma$ or the precision matrix $\Theta := \Sigma^{-1}$.
This is a fundamental problem throughout science and engineering that has been extensively studied in statistics.
The sample covariance and the MLE are described, respectively, as:
\[ \hat{\Sigma} := \frac{1}{n} \sum_{i=1}^{n} X_{i} X_{i}^{T} \qquad \text{and} \quad \hat{\Theta} := \arg\min_{\otheta \succ 0} \langle \otheta, \hat{\Sigma} \rangle - \log\det(\otheta) .    \]

The above is a convex program whose solution is the inverse sample covariance $\hat{\Theta} = \hat{\Sigma}^{-1}$.
This estimation strategy is the gold standard for statistical purposes: it requires only $n \gtrsim d$ samples and gives $\dop(\Theta, \htheta) \lesssim \sqrt{d/n}$ with (very high) probability $\geq 1 - \exp(- \Omega(d))$ (see e.g. \cite[Corollary 5.50]{vershynin2010introduction}).
Intuitively, one has the requirement $n \geq d$, as otherwise we will not even see the whole vector space.
This can be made formal via the information-theoretic lower bound described in \cref{app:lower}.
Also note that the Thompson metric $\dop$ is the tightest error metric of all those considered, in particular $\dop \lesssim \sqrt{d/n}$ implies $\dFR \lesssim \sqrt{d^{2}/n}$.
Finally, from an algorithmic perspective, in this setting the estimator is just the inverse sample covariance, which is very simple to compute.

So what more could we hope for?
It turns out that a strong $\dop$ bound does not imply strong statistical guarantees in general.
In order for statistical measures of error such as $\DKL$ and $\DTV$ to be small, we require $\dFR$ to be a small constant, which requires $n \gtrsim d^{2}$ by the above analysis.
For more details on the relations between these measures, see \cref{app:rel-error}.

In settings such as neighborhood selection in graphical models, it may be the case that the underlying dimension $d$ is much larger than the number of available samples $n$.
In this case, the sample covariance $\hat{\Sigma}$ is not even invertible, so we need to find another way to analyze the estimator for the precision matrix.
What concentration bounds do we have in this setting?

When $d \gg n$, we no longer have the strong concentration of relative error between $\hat{\Sigma}, \Sigma$.
However, as soon as $n \gtrsim \log d$, we still have the following bounds on the entry-wise difference (with high probability):
\[ \|\hat{\Sigma} - \Sigma\|_{\max} \lesssim \|\Sigma\|_{\max} \sqrt{\frac{\log d}{n} } \leq \|\Sigma\|_{\op} \sqrt{\frac{\log d}{n} }. \]
In fact, for the diagonal entries we even have multiplicative error $\hat{\Sigma}_{ii} \in \Sigma_{ii} (1 \pm O(\sqrt{\log d/n}) )$.

A line of works, culminating with \cite{RBLZ08,Cai2016}, leverages this entry-wise bound to devise estimators with strong error guarantees in the low sample regime.
These works assume the the following structural assumptions: \emph{constant condition number} and \emph{sparsity of the precision matrix}.

The results from these works are described precisely in \cref{table: sample complexity gaussian}, where we denote by $s$ the sparsity of the precision matrix (setting of \cite{RBLZ08}), and by $r_s$ the row-sparsity of the precision matrix (setting of \cite{CLIME} and \cite{Cai2016}).
We denote by $\kappa := \kappa(\Theta)$ the condition number of the precision matrix, by $\Delta := \text{diag}(\Sigma)$ the diagonal matrix of variances, and $\Gamma := D^{-1/2} \Sigma D^{-1/2}$ is the `correlation matrix' which satisfies $\text{diag}(\Gamma) = 1_{d}$.

\begin{table}
\caption{Sample complexity for Gaussian Setting}
\label{table: sample complexity gaussian}
\resizebox{1\columnwidth}{!}{
\begin{tabular}{|c|c|c|c|}
\hline
Work & Sample complexity & Error rate & Algorithm \\
\hline
Standard (Folklore, \cite{vershynin2010introduction})    & $n \gtrsim d$ & $\displaystyle \dop(\hat{\Theta}, \Theta) \lesssim \sqrt{\frac{d}{n}}$ & MLE (matrix inversion) \\
\cite[Theorem 1]{RBLZ08}    & $n \gtrsim \kappa^{2} (s+d) \log d$ & $\frac{\|\hat{\Theta} - \Theta\|_{F}}{\|\Theta\|_{op}} \lesssim \kappa \sqrt{ \frac{(s+d) \log d}{n}}$ &  Convex Program
\\
\cite[Theorem 2]{RBLZ08}   & $n \gtrsim \|\Gamma^{-1}\|_{\op}^{2} (s+1) \log d$ & $\displaystyle \frac{\|\hat{\Theta} - \Theta\|_{\op}}{\|\Theta\|_{\op}} \lesssim \kappa(\Delta) \|\Gamma^{-1}\|_{\op} \sqrt{ \frac{(s+1) \log d}{n}}$ & Convex Program \\
\cite[Theorem 1]{CLIME}   & $n \gtrsim \log d$ & $\displaystyle \frac{\|\hat{\Theta} - \Theta\|_{1 \to 1}}{\|\Theta\|_{1 \to 1}} \lesssim \|\Sigma\|_{\max} \|\Omega\|_{1 \to 1} \sqrt{\frac{ r_s^{2} \log d}{n}}$ & Linear program \\
\cite[Theorem 4]{CLIME}   & $n \gtrsim \log d$ & $\displaystyle \frac{\|\hat{\Theta} - \Theta\|_{F}}{\|\Theta\|_{1 \to 1}} \lesssim \|\Sigma\|_{\max} \|\Omega\|_{1 \to 1} \sqrt{\frac{ r_s d \log d}{n}}$ & Linear program \\
\cite[Theorem 3.1]{Cai2016}   & $n \gtrsim \kappa^{2} r_s^{2} \log d$ & $\displaystyle \frac{\|\hat{\Theta} - \Theta\|_{1 \to 1}}{\|\Theta\|_{1 \to 1}} \lesssim \sqrt{ \|\Sigma\|_{\max} \|\Omega\|_{\max} \frac{ r_s^{2} \log d}{n}}$ & Linear program \\
\cite[Theorem 6.1]{Cai2016}   & $n \gtrsim \kappa^{2} r_s^{2} \log d$ & $\displaystyle \frac{\|\hat{\Theta} - \Theta\|_{F}}{\|\Theta\|_{1 \to 1}} \lesssim \sqrt{ \|\Sigma\|_{\max} \|\Omega\|_{\max} \frac{ r_s d \log d}{n}}$ & Linear program \\
\cite[Theorem 4.1]{Cai2016}   & $n \lesssim o(d), r_s \lesssim o(\sqrt{n})$ & $\displaystyle \frac{\|\hat{\Theta} - \Theta\|_{1 \to 1}}{\|\Theta\|_{1 \to 1}} \gtrsim \sqrt{ \frac{r_s^{2} \log d}{n}}$ & N/A (lower bound) \\
\cite[Theorem 6.1]{Cai2016}   & $n \lesssim o(d), r_s \lesssim o(\sqrt{n})$ & $\displaystyle \frac{\|\hat{\Theta} - \Theta\|_{F}}{\|\Theta\|_{1 \to 1}} \gtrsim \sqrt{ \frac{r_s d \log d}{n}}$ & N/A (lower bound) \\

\hline
\end{tabular}
}
\end{table}

By \cref{table: sample complexity gaussian} and \cref{app:rel-error}, when $\kappa = O(1)$, we have equivalence between relative error $\dFR$ and absolute error $\|\cdot\|_{F}$.
Hence, the above results provide an advantage over the MLE whenever $s \ll d^{2}$ or $r_{s} \ll d$.
However, as the condition number increases, both the error rate and the sample complexity deteriorate quite rapidly, as well as the difference between relative error and absolute error.

We now give a high-level overview of the approaches taken for these improved estimators.
Recall that the MLE is the solution of the following convex program (when $n \gtrsim d$):
\[ \hat{\Theta} := \arg\min_{\bar{\Theta} \succ 0} \langle \bar{\Theta}, \hat{\Sigma} \rangle - \log \det(\bar{\Theta}) ,    \]
where $\Sigma := \frac{1}{n} \sum_{i=1}^{n} X_{i} X_{i}^{T}$ is the sample covariance.

Note that the crucial relation here is $\hat{\Sigma} = \hat{\Theta}^{-1}$, so the closer $\hat{\Sigma} \approx \Sigma$, the closer this estimator will be to the true precision matrix $\Theta$.
But when $\hat{\Sigma}$ is not invertible, it is more difficult to analyze the above program.
Therefore we would like to exploit structural conditions of the true precision matrix and modify the above program to bias the optimum towards $\Theta$.
For example, \cite[Theorem 1]{RBLZ08} uses the following penalized likelihood program:
\[ \hat{\Theta} := \arg\min_{\bar{\Theta} \succ 0} \langle \bar{\Theta}, \hat{\Sigma} \rangle - \log\det(\bar{\Theta}) + \lambda \|\bar{\Theta}\|_{1,\text{off}} .      \]
Here $\|\bar{\Theta}\|_{1,\text{off}}$ measures the $\ell_{1}$ norm of the off-diagonal elements.
This is a ``lasso'' style penalty function which biases the optimum towards sparse solutions, and the parameter $\lambda$ is a tuning parameter which is chosen to balance sparsity and error so that the optimum solution is close to $\Theta$.
Then, \cite[Theorem 2]{RBLZ08} uses the observation that the diagonal entries of $\hat{\Sigma}_{ii}$ have much stronger concentration to the true values $\Sigma_{ii}$, so they use $\Gamma := \hat{\Delta}^{-1/2} \hat{\Sigma} \hat{\Delta}^{-1/2}$ with the above program to estimate the (inverse) correlation matrix $\Gamma^{-1}$, and then replace the estimated diagonals.

The CLIME estimator of \cite{CLIME} uses the following linear program:
\[ \hat{\Theta} := \arg\min_{\bar{\Theta} \succ 0} \|\bar{\Theta}\|_{1} \quad \text{s.t. } \quad \|\hat{\Sigma} \bar{\Theta} - I_{d} \|_{\max} \leq \lambda .      \]
Where $\lambda$ is a tuning parameter that makes sure the true precision matrix is a feasible solution, and depends on the entry-wise concentration of the sample covariance.
The intuition for this program is discussed in the introduction of \cite{CLIME} as a way to directly find an approximate solution to the optimality conditions of the lasso-type program used in \cite{RBLZ08}.
This is further refined in \cite{Cai2016} by exploiting stronger concentration bounds for the equation $\|\hat{\Sigma} \Theta - I_{d}\|_{\max}$.

We remark that the above two estimators \cite{CLIME} and \cite{Cai2016} have very low sample complexity, as can be seen in \cref{table: sample complexity gaussian}. But the estimators are not guaranteed to be positive semi-definite, and for this a larger number of samples is required. This is discussed more precisely in the supplement.

All of the above programs can be analyzed by replacing the sample covariance $\hat{\Sigma}$ with an arbitrary input $\bar{\Sigma}$ (which one should think of as a ``good initial guess'').
The important observation is that the error rate depends only on the entry-wise error of the initial guess to the true covariance.
Therefore in the following table, we precisely state these arguments in terms of the entry-wise accuracy of the initial guess, denoted by $\nu := \|\bar{\Sigma} - \Sigma\|_{\max}$.
This will be helpful in the following subsection in order to understand previous works on the matrix and tensor normal model.
For \cite[Theorem 2]{RBLZ08}, we are also given guess $\overline{\Delta} \approx \Delta$, and we use it to estimate the correlation matrix $\Gamma := \Delta^{-1/2} \Sigma \Delta^{-1/2}$. For this we use notation $\nu_{\Delta} := \|\Delta^{-1} \overline{\Delta} - I_{d} \|_{\max}$ and $\nu_{\Gamma} := \|\Delta^{-1/2} (\overline{\Sigma} - \Sigma ) \Delta^{-1/2} \|_{\max}$.

\begin{table}
\caption{Sample complexity for Gaussian Setting}
\label{table: entry wise gaussian}
\resizebox{1\columnwidth}{!}{
\begin{tabular}{|c|c|c|c|}
\hline
Work & \makecell{Accuracy of \\ initial guess} & Error rate & Tuning Parameter \\
\hline
\cite[Theorem 1]{RBLZ08}    & \makecell{$\displaystyle \nu \lesssim \dfrac{1}{\|\Theta\|_{\op}\sqrt{s + d}}$} & $\dfrac{\|\hat{\Theta} - \Theta\|_{F}}{\|\Theta\|_{\op}} \lesssim \|\Theta\|_{\op} \nu \sqrt{s + d}$ &  $\lambda \geq \nu$ \\
\cite[Theorem 2]{RBLZ08}
& $\nu_{\Delta} \lesssim 1, \nu_{\Gamma} \lesssim \frac{1}{\|\Gamma^{-1}\|_{\op} \sqrt{s}}$ & $\displaystyle \frac{\|\hat{\Theta} - \Theta\|_{\op}}{\|\Theta\|_{\op}} \lesssim \kappa(\Delta) \bigg( \|\Gamma^{-1}\|_{\op} \nu_{\Gamma} \sqrt{s} + \nu_{\Delta} \bigg)$ & $\lambda \geq \nu_{\Gamma}$ \\
\cite[Theorem 6]{CLIME}   & None & $\displaystyle \frac{\|\hat{\Theta} - \Theta\|_{\max}}{\|\Theta\|_{1 \to 1}} \lesssim \|\Theta\|_{1 \to 1} \nu $ & $\lambda \geq  \|\Theta\|_{1 \to 1} \nu$ \\
\hline
\end{tabular}
}
\end{table}

\medskip

\noindent\textbf{Computational complexity of estimators.} All of the above results that apply to structured inputs compute estimators that are solutions to convex programs or linear programs.
Moreover, they require some side information about the true solution in order to produce good tuning parameters.
In general, these programs can be solved to high accuracy (i.e. with $\log(1/\delta)$ convergence) using interior point methods or the ellipsoid method.
However, these methods incur a very high polynomial cost per iteration, and the convergence rate will depend on the condition number of the true solution.
Similarly, first-order methods have lower cost per iteration but will generally only provide $\poly(1/\delta)$ convergence to the optimum.
If we want to efficiently compute an estimator that matches the promised error rates, we need to solve the given program to high accuracy, so these rates will become prohibitive.
One very important advantage of the MLE is that it can be exactly computed by a single matrix inversion.

\vspace{10pt}

\subsection{Sample complexity and error rate of matrix and tensor normal models}\label{subsec:matrix-tensor-previous-work}

Previous works on non-asymptotic bounds for the matrix and tensor normal models essentially reduce these problems to separate Gaussian estimation problems for each Kronecker factor.
Hence, they rely on \emph{entry-wise accuracy bounds} and exploit structural assumptions such as sparsity, while focusing on the setting of \emph{constant condition number}.
Additionally, to overcome the difficulties arising from the multiple Kronecker factors, these works
required the following extra assumptions for their estimators:
\begin{equation}\label{equation:assumptions-previous work}
    \text{Initial guess } \otheta_a \text{ satsifying } \|\Sigma_{a}\|_{\op} \|\overline{\Theta}_a - \Theta_a\|_{\op} \lesssim \frac{1}{k}; \text{Knowledge of } \|\Sigma_a\|_{\op}, \|\Theta_{a}\|_{F}.
\end{equation}
\noindent As we will soon discuss (\cref{table: guesses effect}), the premises of a \emph{good initial guess} and \emph{constant condition number} of the true covariance matrices are strong assumptions, without which the quality of their estimators deteriorates quite rapidly.

We begin by presenting their results with the above assumptions in \cref{table: main results}, highlighting the dependence on the condition number.
In \cref{table: main results}, we denote by $s_a, r_{s,a}$ the sparsity and row-sparsity of the precision matrix $\Theta_a$, by $k$ the number of Kronecker factors, $\kappa_a$ denotes the condition number of $\Sigma_a$ and $\kappa_{max} = \max_a \kappa_a$.
We also recall that $\Delta_a := \text{diag}(\Sigma_a)$, $\Gamma_a := \Delta_a^{-1/2}\Sigma_a \Delta_a^{-1/2}$, $n$ is the number of samples and $\displaystyle D := \prod_{a=1}^k d_a$.

\begin{table}
\caption{Error rates and performance of estimators}
\label{table: main results}
\resizebox{1\columnwidth}{!}{
\begin{tabular}{|p{2.3cm}|p{1.4cm}|c|c|c|}
\hline
Work & Setting & Sample Threshold & Error Rate (above sample threshold) & Assumptions \\
 \hline
\cite[Theorem 3]{tsiligkaridis2013convergence}
    & \makecell[t]{general, \\ $k=2$}
    & $\sum_{a \in \{1,2\}} \max\{ 1, \frac{\kappa_{a}^{2}}{d_{a}} \} \kappa_{a}^{2} \frac{d_{a}^{3} \log D}{D}$
    & $ \frac{\|\hat{\Theta}^{(3)} - \Theta\|_{F}}{\|\Theta_{a}\|_{\op}} \lesssim \sum_{a \in \{1,2\}} \kappa_{a}^{2} \frac{d_{a}^{2} \log D}{n} $
    & \eqref{equation:assumptions-previous work} \\
    \cite[Theorem 4]{tsiligkaridis2013convergence}
    & \makecell[t]{$s_{a} \lesssim d_{a}$, \\ $k=2$}
    & $\sum_{a \in \{1,2\}} \max\{ 1, \frac{\kappa_{a}^{2}}{d_{a}} \} \kappa_{a}^{2} \frac{d_{a}^{2} \log D}{D}$
    & $\frac{\|\hat{\Theta} - \Theta\|_{F}}{\|\Theta\|_{\op}} \lesssim \sum_{a \in \{1,2\}} \kappa_{a}^{2} \frac{d_{a} \log D}{n} $
    & \eqref{equation:assumptions-previous work} \\
    \cite[Theorem 3.1]{zhou2014gemini}
    & \makecell[t]{$k=2$,\\ general $s_{a}$}
    & $\displaystyle  \sum_{a \in \{1,2\}} \max\{ 1, \frac{\kappa_{a}^{2}}{d_{a}} \} ( \kappa_{a} \kappa(\Delta_{a}) \|\Gamma_{a}^{-1}\|_{\op})^{2} \frac{d_{a}(s_{a} + 1)\log D}{D}$
    & $\displaystyle\frac{\|\hat{\Theta} - \Theta\|_{\op} }{\|\Theta\|_{\op}} \lesssim \sum_{a \in \{1,2\}} \kappa_{a} \kappa(\Delta_{a}) \|\Gamma_{a}^{-1}\|_{\op} \frac{(s_{a} + 1)\log D}{n}$
    & \eqref{equation:assumptions-previous work} \\
    \cite[Theorem 3.3]{zhou2014gemini}
    & $k=2$, general $r_{s,a}$
    & $\sum_{a \in \{1,2\}} (\|\Sigma_{a}\|_{\op} \|\Sigma_{a}\|_{\max} \|\Omega_{a}\|_{1 \to 1}^{2} )^{2} \frac{r_{s,a}^{2} d_{a} \log D}{D} $
    & $\displaystyle\frac{\|\hat{\Theta} - \Theta\|_{\op} }{\|\Theta\|_{\op}} \lesssim \sum_{a \in \{1,2\}} \frac{\|\Sigma_{a}\|_{\max} \|\Omega_{a}\|_{1 \to 1}^{2}}{\|\Omega_{a}\|_{\op}} \frac{r_{a}^{2} d_{a} \log D}{nD} $
    & \eqref{equation:assumptions-previous work}  \\
    \makecell{ \\ \cite{Lyu2020Tlasso}}
    & general $s_{a}$,  general $k$
    & \makecell{$\displaystyle k^{2} \sum_{a \in [k]} \max\{ 1, \frac{\kappa_{a}^{2}}{d_{a}} \} \kappa_{a}^{2} \frac{d_{a}(s_{a} + d_{a})\log D}{D} $}
    & $\frac{ \|\hat{\Theta}_{a} - \Theta_{a}\|_{F}}{\|\Theta_{a}\|_{\op}} \lesssim \kappa_{a} \sqrt{\frac{d_{a} (s_{a} + d_{a}) \log D}{nD}} $
    & \eqref{equation:assumptions-previous work}
    \\
    \CREFmain{thm:matrix-normal}{Theorem~1.11}
    & \makecell[t]{general, \\ $k=2$}
    & $\frac{d_{\max}^{2} \log D}{D}$
    & $ \dop(\hat{\Theta}_{a}, \Theta_{a}) \lesssim \sqrt{\frac{d_{a}^{2} \log^{2} d_{\min}}{n D}} $
    & None \\
    \CREFmain{thm:tensor-frobenius}{Theorem~1.10}
    & \makecell[t]{general,\\ $k \geq 2$}
    & $ \frac{k^{2}d_{\max}^{3}}{D}  $
    & $ \dFR(\hat{\Theta}_{a}, \Theta_{a}) \lesssim \sqrt{ \frac{k d_{\max}^{2} d_{a} }{n D}} $
    & None \\
    \hline
\end{tabular}
}
\end{table}

Intuitively, the reduction in previous works from the matrix normal model to Gaussian estimation works as follows: assume we knew exactly the value of $\Theta_2$.
Then we could `normalize' our matrix samples $X \in \R^{d_1 \times d_2}$, which we denote by $Y :=  X \Theta_2^{1/2}$,
and note that the columns of $Y$ are independent and distributed as $Ye_{j} \sim N(0,\Sigma_1)$.
In other words, we have decorrelated the columns of $Y$ and transformed them into samples from $\cN(0, \Sigma_1)$.
Thus, given samples $X_{1}, ..., X_{n}$, we estimate $\Theta_1$ by applying any technique for Gaussian estimation to the $n d_2 = n D/ d_1$ columns $\{Y_{i} e_{1}, \dots, Y_{i} e_{d_2} \}$ and get the error rates for the $k=1$ setting.

In that vein, the estimators of \cite[Theorem 4]{tsiligkaridis2013convergence} and \cite{Lyu2020Tlasso} reduce to the estimator of \cite[Theorem 1]{RBLZ08};
\cite[Theorem 3.1]{zhou2014gemini} reduces to that of \cite[Theorem 2]{RBLZ08};
and \cite[Theorem 3.3]{zhou2014gemini} reduces to that of \cite[Theorem 6]{CLIME}.

Of course, previous works do not know the true precision matrix $\Theta_2$.
The main contribution of these results on the matrix and tensor normal models is to show that the above analysis applies with essentially the same guarantees as long as we start with \emph{good enough guesses} for the precision matrices.
For simplicity, we focus on the implementation of the above strategy for the matrix normal model, and state the full results for the tensor normal model in \cref{table: guesses effect}.

Suppose one is given guesses $\overline{\Theta}_1, \overline{\Theta}_2$ for the precision matrices.
We would like to separate the matrix normal model problem into two Gaussian estimation problems, one for each of the Kronecker factors.
To achieve this, we need to produce an estimate for the true covariance $\Sigma_{1}$ (and analogously an estimate for $\Sigma_2$).

If we knew $\Theta_{2}$ exactly, then the above strategy implies that the sample covariance of the vectors $Y_{i}e_j$, given by $\widetilde{\Sigma}_{1} := \frac{1}{nd_2} \sum_{i=1}^{n} X_{i} \Theta_{2} X_{i}^{T}$ gives a reasonable guess for the true covariance $\Sigma_1$, and will have the same accuracy as in the $k=1$ setting, i.e.
\[ \|\widetilde{\Sigma}_{1} - \Sigma_{1} \|_{\max} \lesssim \|\Sigma_{1}\|_{\op} \sqrt{ \frac{ \log D}{n d_{2}}} = \|\Sigma_{1}\|_{\op} \cdot  \sqrt{\frac{d_{1} \log D}{n D}} .    \]
Because we only have a guess $\overline{\Theta}_2$, we can apply the above strategy with this guess, obtaining
\[ \overline{\Sigma}_{1} := \frac{1}{nd_2} \sum_{i=1}^{n}  X_{i} \overline{\Theta}_{2} X_{i}^{T},   \]
and we can hope that it will `approximately decorrelate' the columns of the matrices $X_i$ and produce a good approximation of the true covariance $\Sigma_1$.

To see where this approach fails, consider the example where $\Theta_{2} = I_{d_{2}}$ but our guess is the projector only the first column. In this extreme case, each matrix data is essentially reduced to just a single sample for each column.

Quantitatively, the accuracy of a given guess is bounded by
\[ \|\overline{\Sigma}_{1} - \Sigma_{1}\|_{\max} \lesssim \gamma_2 \cdot \|\widetilde{\Sigma}_{1} - \Sigma_{1} \|_{\max} \lesssim \gamma_2 \cdot \|\Sigma_{1}\|_{op} \sqrt{ \frac{ \log D}{n d_{2}}} = \gamma_2 \cdot \|\Sigma_{1}\|_{op} \cdot \sqrt{\frac{d_{1} \log D}{n D}} .  \]
Here the multiplicative factor $\gamma_2$
accounts for the accuracy of our guess for $\Theta_2$, and is given by
\begin{equation}\label{equation: multiplicative guess factor}
\gamma_b :=
\frac{\sqrt{d_{b}} \|\Theta_{b}^{-1/2} \overline{\Theta}_{b} \Theta_{b}^{-1/2}\|_{F}}{\tr[\Theta_{b}^{-1/2} \overline{\Theta}_{b} \Theta_{b}^{-1/2}]} .
\end{equation}
It can be seen (by Cauchy-Schwarz) that
this quantity always satisfies $1 \leq \gamma_{b} \leq d_{b}$.
In the supplement (Fact 6.5), we prove some bounds on this quantity in terms of the accuracy of the initial guess: if $\kappa(\Theta) \leq \kappa$ is known, and we choose guess $\overline{\Theta} = I_{d}$, then $\gamma \lesssim \sqrt{\kappa}$. Further, a very accurate guess gives
\[ \DF(\overline{\Omega}, \Omega) \leq \frac{\sqrt{d}}{2} \implies \gamma \leq 1 + O\bigg( \frac{\DF(\overline{\Omega}, \Omega)}{\sqrt{d}} \bigg) .    \]
We also show these bounds are tight: there are instances with $\kappa(\Theta) \leq \kappa$ and $\overline{\Theta} = I_{d}$ with $\gamma \gtrsim \sqrt{\kappa}$; and similarly there are instances with $\DF(\overline{\Omega}, \Omega) \approx \sqrt{d}$ and $\gamma \approx \sqrt{d}$.
This $\DF$ bound can be seen as the reason for \cref{equation:assumptions-previous work}.
With these accuracy bounds at hand, the results from \cref{table: entry wise gaussian} can be applied to obtain guarantees for the matrix and tensor normal model.

In general, for the tensor normal model case, we get the following accuracy bounds:
\begin{equation}\label{equation: accuracy bounds tensor}
    \nu_{a} := \|\overline{\Sigma}_{a} - \Sigma_{a}\|_{\max} \lesssim \left( \prod_{b \neq a} \gamma_{b} \right) \cdot \|\Sigma_{a}\|_{\op} \sqrt{ \frac{ d_{a} \log D}{n D}} .
\end{equation}


In order to use the estimators from the Gaussian estimation problem (\cref{table: entry wise gaussian}), we require the accuracy to be small enough, which in turn yield a requirement for the number of samples needed.
Moreover, as the error rate of the estimators from \cref{table: entry wise gaussian} depends on the accuracy, this approach will produce an estimator for the precision matrix which will have worse error.

Under assumption~\eqref{equation:assumptions-previous work}, the extra factors are constant, thus we get the results in \cref{table: main results}.
However, in the \emph{absence of good guesses} (which is the foundational estimation problem), the above bounds deteriorate quite rapidly, which leads to a larger sample threshold requirement.

In the absence of good initial guesses, it is still possible to achieve the same error rate as in \cref{table: main results}, albeit with a (much) higher number of samples.
To achieve the better error rates, previous works for the matrix and tensor normal model actually apply the above procedure iteratively, updating their guesses for the precision matrices with the estimators computed in the previous iteration.
If the iteration produces sufficiently accurate guesses, then the error rate will decrease down to the level of the `good guess' setting.
This additional accuracy requirement leads to the $\max\{1, \kappa^{2}/d\}$ factor shown in \cref{table: guesses effect} below, where we compare the sample threshold with our work.

\begin{table}
\caption{Sample requirements with initial guesses $\overline{\Theta}_a$ which are $\gamma_a$-accurate to achieve  error rates from \cref{table: main results}}
\label{table: guesses effect}
\resizebox{1\columnwidth}{!}{
\begin{tabular}{|p{3.5cm}|c|c|c|}
\hline
\makecell{Work} & Setting & \makecell{Sample threshold}  \\
 \hline
\cite[Theorem 3]{tsiligkaridis2013convergence} & general, $k=2$
    & $\displaystyle \sum_{a \in \{1,2\}} \max\{ 1, \frac{\kappa_{a}^{2}}{d_{a}} \} (\gamma_{b} \kappa_{a})^{2} \frac{d_{a}^{3} \log D}{D}$ \\
 \hline
\cite[Theorem 4]{tsiligkaridis2013convergence} & $s_{a} \lesssim d_{a}$, $k=2$
    & $\displaystyle \sum_{a \in \{1,2\}} \max\{ 1, \frac{\kappa_{a}^{2}}{d_{a}} \} (\gamma_{b} \kappa_{a})^{2} \frac{d_{a} \log D}{D}$ \\
    \hline
    \cite[Theorem 3.1]{zhou2014gemini} & $k=2$
    & $\displaystyle \sum_{a \in \{1,2\}} \max\{ 1, \frac{\kappa_{a}^{2}}{d_{a}} \} ( \gamma_{b} \kappa_{a} \kappa(\Delta_{a}) \|\Gamma_{a}^{-1}\|_{\op})^{2} \frac{d_{a}(s_{a} + 1)\log D}{D}$  \\
    \hline
    \cite[Theorem 3.3]{zhou2014gemini} & $k=2$, general $r_{s,a}$
    & \makecell{$\displaystyle \sum_{a \in \{1,2\}} (\gamma_{b} \|\Sigma_{a}\|_{\op} \|\Sigma_{a}\|_{\max} \|\Omega_{a}\|_{1 \to 1}^{2} )^{2} \frac{r_{s,a}^{2} d_{a} \log D}{D}$}  \\
    \hline
    \cite{Lyu2020Tlasso} & general $s_{a}$, general $k$
    & \makecell{$\displaystyle k^{2} \sum_{a \in [k]} \max\{ 1, \frac{\kappa_{a}^{2}}{d_{a}} \} \bigg( \prod_{b \neq a} \gamma_{b} \bigg)^{2} \kappa_{a}^{2} \frac{d_{a}(s_{a} + d_{a})\log D}{D}$}  \\
    \hline
    \makecell{\CREFmain{thm:matrix-normal}{Theorem~1.11}}
    & \makecell[t]{general, \\ $k=2$}
    & $\frac{d_{\max}^{2}}{D} \log D$
    \\
    \hline
    \makecell{\CREFmain{thm:tensor-frobenius}{Theorem~1.10}}
    & \makecell[t]{general,\\ $k \geq 2$}
    & $ \frac{k^{2}d_{\max}^{3}}{D}  $
    \\
    \hline
\end{tabular}
}
\end{table}

\begin{remark}
    Recall that the concentration bound in \cref{equation: accuracy bounds tensor} for entry-wise error applied only when the `guess' is a fixed deterministic input. As discussed, previous works apply this procedure iteratively so that the accuracy of the guesses improve and the error rate can be brought down. But this violates the independence assumption required for concentration, as the guesses in subsequent iterations are themselves random variables that depend on the random Gaussian input.
    Applying concentration for this new dependent random variable would require new non-trivial technical arguments.
    This error can be fixed by taking a fresh batch of independent samples for each iteration.
\end{remark}

It is important to note that, in the absence of any good guess, with only the guarantee that $\kappa(\Theta_a) \leq \kappa_{a}$, the best guess to choose is the identity, which case we have the bound $\gamma_a \lesssim \sqrt{\min\{\kappa_{a}, d_{a}\}}$.
Substituting this bound into the above table shows the sample threshold for these previous estimators in the setting of moderately large condition numbers $\kappa_{a} \geq d_{a}$ is just as bad as if we had to estimate a general precision matrix on tensor data, i.e. without the crucial structural assumption that the covariance has tensor structure $\Sigma = \otimes_{a \in [k]} \Sigma_{a}$.
A simplification of \cref{table: main results,table: guesses effect}, combined to reflect the worst case bounds on the sample threshold and error rate of all estimators is given in \CREFmain{table: main results intro}{Table 1}.

\smallskip

\textbf{Comparison with our work:} As the above discussion highlights, the MLE provides a high quality estimator in the general setting with no assumptions.
Further, if the precision matrices are known to satisfy structural assumptions such as sparsity, and the goal is to find estimators that are close in other error measures such as entry-wise error, then our result shows that the MLE can be plugged in as a high quality initial guess to these procedures, which gives an effective reduction to any estimator for the Gaussian setting.

\subsection{Computational complexity of previous estimators}\label{subsec:complexity-previous-work}

We now discuss the computational complexity guarantees of previously proposed estimators.

Earlier works on the matrix normal model (\cite{mardia1993spatial, dutilleul1999mle, brown2001bayesian} and references therein) proposed (seemingly independently) an iterative algorithm, known as the \emph{flip-flop algorithm}, to compute the MLE.
In the \textbf{asymptotic regime}, \cite{werner2008estimation} showed that the MLE is consistent and asymptotically normal, and showed the same for the estimator obtained by terminating the flip-flop after three steps, starting from an arbitrary initial guess.
For the tensor normal model, a natural generalization of flip-flop  was proposed \citep{mardia1993spatial,manceur2013maximum}, but its convergence was not proven.
The above works neither provide non-asymptotic guarantees, nor do they provide an estimator which computes the MLE.

In the \textbf{non-asymptotic regime}, prior to this work,  \cite[Theorem 3]{tsiligkaridis2013convergence} was the only work to study properties of the flip-flop algorithm, where they analyze the estimator given by applying 3 iterations of the flip-flop algorithm.
As we have seen in the previous section, other estimators have been proposed which generalize the $k=1$ case, and therefore these estimators need to (iteratively) solve certain convex programs (given by regularized variants of the MLE for the $k=1$ case) to estimate each of the Kronecker factors of the precision matrix.

The only algorithmic result for the tensor normal model that we have not yet discussed is the work of \cite{XZG17}, which analyses a constrained variant of the MLE, by imposing sparsity constraints on the precision matrices.
In this work, the authors propose a block coordinate gradient descent algorithm with truncation for sparsity and sample splitting, to solve the sparsity constrained MLE problem.
Their work follows the approaches of previous works, and shows that assuming that the \emph{initial guesses} are \textbf{close enough} to the \emph{true precision matrices}, then in a constant number of steps their algorithm obtains estimators which are close to the true precision matrices in Frobenius norm.

While the algorithm proposed in \cite{XZG17} is efficient per iteration, the assumptions needed to guarantee correctness have a heavy dependence on the condition number.
For their algorithm to obtain improvements on the distance to the true precision matrices \cite[Inequality (4.1)]{XZG17}, they need the number of samples to be $n \gtrsim \kappa^{2k} \cdot \max_a\left\{ \dfrac{(k \dmax)^2}{D^2} \cdot T d_a s_a \log s_a \right\}$.
In addition, if one sets the dimensions of each covariance factor to be the same, then their main theorem (Theorem 4.2) only works when $k \geq 4$ (in order to respect the condition number inequality $\kappa \geq 1$).

Another drawback of their algorithm is that it cannot run for more than constantly many iterations, due to sample splitting.
The use of sample splitting implies that their bound on the distance to the optimum worsens as the number of iterations increases, as each batch of samples may be a worse initial guess than if one considered all samples together.

From the discussion of the estimators in \cref{subsec:matrix-tensor-previous-work} and the above discussion, we note that all algorithms to compute the proposed estimators can be described by the iterative application of a main subroutine until convergence is achieved.
While in the works \cite{tsiligkaridis2013convergence,zhou2014gemini,Lyu2020Tlasso} the main subroutine is the solution of a linear program or convex program (which is costly in practice), the main subroutine in this work is simply the computation of a matrix inverse, and the main subroutine in \cite{XZG17} is one truncated gradient descent step, the last two being quite fast to compute.

A brief summary of the iteration complexity of all previous works can be found in \cref{table: estimator performance} assuming one has a good initial guess.
For previous works, the iteration complexity does not change when there is no good guess, but the sample complexity gets significantly larger.
On the other hand, the iteration complexity of the Flip-Flop depends on the accuracy of the guess,
but has a very cheap per-iteration cost.

\begin{table}
\caption{Performance of estimators under assumption~\eqref{equation:assumptions-previous work}}
\label{table: estimator performance}
\begin{tabular}{|p{3.5cm}|c|c|c|}
\hline
\makecell{Work} & Setting &  Main subroutine & \makecell{Iterations to \\ $\delta$-close to \\ estimator $\hat{\Theta}_a$} \\
\hline

\cite[Theorem 3]{tsiligkaridis2013convergence} & \makecell{$k=2$, \\ general} & \makecell{matrix inversion} & $3$ \\

\hline

\cite[Theorem 4]{tsiligkaridis2013convergence} & \makecell{$k=2$, \\ $s_a \lesssim d_a$} & convex program & $3$ \\

\hline

\cite[Theorem 3.1]{zhou2014gemini} & \makecell{$k=2$, \\ general $s_a$} & convex program & $3$ \\

\hline

\cite[Theorem 3.3]{zhou2014gemini} & \makecell{$k=2$, \\ $r_{s,a} \lesssim \sqrt{d_a}$} & linear program & $3$ \\

\hline

\cite{XZG17} & \makecell{$k \geq 4$, \\ general $r_{s,a}$} & \makecell{truncated \\ gradient descent} & N/A \\

\hline

\cite{Lyu2020Tlasso} & \makecell{$k \geq 2$,\\ general $s_a$} & convex program & $2k$ \\

\hline

\CREFmain{thm:tensor-frobenius,thm:matrix-normal}{Theorems~1.10 and~1.11} & $k \geq 2$ & matrix inversion  & $O(k \log(1/\delta))$ \\
\hline
\end{tabular}
\end{table}

\textbf{Comparison with our work:}
Prior to this work, none of the previous works provided an algorithm to correctly approximate the MLE.
Moreover, among all the proposed algorithms which compute an approximation to their corresponding estimator, the iteration complexity of the flip-flop algorithm is competitive when compared to the iteration complexity of other algorithms when given a good initial guess.
The iteration complexity of the Flip-Flop procedure has a logarithmic dependence on the accuracy of the initial guess, but the sample complexity remains the same for the MLE. This is in contrast to all previous results, where sample complexity grows \emph{polynomially} as the accuracy of the initial guess increases, while the iteration complexity stays the same.
Our analysis therefore shows the theoretical and practical advantages of flip-flop, as the iteration cost of the main subroutine of flip-flop is simply a matrix inversion computation, which has low cost per iteration, whereas the other algorithms require the solution of a convex program, which has a large cost per iteration, if one wants to compute a good enough solution (as discussed in the end of \cref{subsec:gaussian-previous-work}).

\section{Proofs of quantum expansion}\label{app:expansion}

In this appendix we give proofs of \cref{thm:Pisier-expansion} and \CREFmain{thm:operator-cheeger}{Theorem~3.1}, which establish quantum expansion for random completely positive maps.
These are used in \cref{app:tensor,app:matrix} but can be read independently of those sections.

\subsection{Pisier's argument}\label{app:pisier}
In this section we prove the main technical theorem used in the proof of \CREFmain{thm:hess-pisier}{Theorem~2.16}.
This follows from \cite{pisier2012grothendieck}, whose original theorem dealt with square matrices and gave slightly weaker probabilistic guarantees than \cref{thm:Pisier-expansion} stated below.
We adapt this result to give exponentially small error probability in the setting of rectangular matrices.
These are minor modifications, which follow readily from \cite{P14,pisier2012grothendieck}.
Therefore, we state the proof below for completeness and claim no originality.

\begin{theorem}[Pisier]\label{thm:Pisier-expansion}
Let $A_1,\dots,A_N,A$ be independent $n \times m$ random matrices with independent standard Gaussian entries.
For any~$t\geq 2$, with probability at least $1 - t^{-\Omega(m+n)}$,
\begin{align*}
  \norm*{\left(\sum_{i=1}^N A_i \otimes A_i \right) \circ \Pi}_{\op}
  \leq O\left( t^2 \sqrt{N} ( m+n ) \right),
\end{align*}
where $\Pi$ denotes the orthogonal projection onto the traceless subspace of $\R^m \ot \R^m$, that is, onto the orthogonal complement of $\vect(I_m)$.
\end{theorem}

In the remainder we discuss the proof of \cref{thm:Pisier-expansion}.
The proof proceeds by a symmetrization trick, followed by the trace method.
We first state some relevant bounds on Gaussian random variables and then give the proof of \cref{thm:Pisier-expansion}.

We will often use the following estimate of the operator norm of a standard Gaussian $n \times m$ random matrix~$A$ (see Theorem~5.32 in \cite{vershynin2010introduction}),
\begin{equation}\label{eq:op norm upper bound}
  \E \norm{A}_{\op} \leq \sqrt{n} + \sqrt{m}  .
\end{equation}

\begin{theorem}\label{thm:banach conc}
Let $A$ be a centered Gaussian random variable that takes values in a separable Banach space with norm $\|\cdot\|$.
Then $\|A\|$ satisfies the following concentration and moment inequalities with parameter $\sigma^2 := \sup \{ \E \langle \xi, A \rangle^{2} \mid \|\xi\|_{*} \leq 1 \}$, where $\|\cdot\|_{*}$ denotes the dual norm:
\[ \forall t > 0 \colon \quad \Pr\Bigl( \abs[\big]{\norm A - \E \norm A} \geq t\Bigr) \leq 2 \exp\Bigl( - \frac{\Omega(t^2)}{\sigma^{2}}\Bigr)  , \qquad \text{and}  \]
\begin{equation}\label{eq:conc via moments}
  \forall p \geq 1 \colon \quad \E \|A\|^{p} \leq (2 \E \|A\|)^{p}  + O ( \sigma \sqrt{p} )^{p}.
\end{equation}
\end{theorem}
\begin{proof}
The first statement on concentration is exactly Theorem~1.5 in \cite{P86}.
For the second, we consider the random variable $X := \frac{1}{\sigma} (\|A\| - \E \|A\| )$.
Then the equivalence in Lemma 5.5 of \cite{vershynin2010introduction} gives the moment bound
\[
  \Bigl( \E |X|^{p} \Bigr)^{1/p}
= \frac1\sigma \Big( \E \Big| \|A\| - \E \|A\| \Big|^{p} \Big)^{1/p}
 \leq O(\sqrt{p}).
\]
The moment bound in the theorem now follows by rearranging as
\[
  \E \|A\|^{p}
= \E \Bigl( \E \|A\| + \sigma X  \Bigr)^{p}
\leq 2^{p} \Bigl( (\E \|A\|)^{p} + O( \sigma \sqrt{p} )^{p} \Bigr),
\]
where the last step was by the simple inequality $(a+b)^{p} \leq 2^{p} (|a|^{p} + |b|^{p})$.
\end{proof}

Below, we calculate the $\sigma^{2}$ parameter in \cref{thm:banach conc} for our random matrix setting.

\begin{corollary}\label{lem:opNormSubG}
Let $A$ be an $n \times m$ matrix with independent standard Gaussian entries. Then the random variable $\norm{A}_{\op}$ satisfies the conclusions of \cref{thm:banach conc} with $\sigma^{2} = 1$.
\end{corollary}
\begin{proof}
Note that the dual norm is the trace norm~$\norm{\cdot}_{\operatorname{tr}}$, hence the concentration parameter can be estimated as
\begin{align*}
  \sigma^2
= \sup \left\{ \E \langle \xi, A \rangle^2 \;\mid\; \norm{\xi}_{\operatorname{tr}} \leq 1 \right\}
= \sup \left\{ \norm\xi_F^2 \;\mid\; \norm{\xi}_{\operatorname{tr}} \leq 1 \right\}
= 1,
\end{align*}
where we used that $\braket{\xi,A}$ has the same distribution as $\norm\xi_F A_{11}$ by orthogonal invariance, and that the trace norm dominates the Frobenius norm.
\end{proof}

We will also use the the \emph{Schatten $p$-norms} $\norm{A}_p = (\tr\left[(A^TA)^{\frac p2}\right])^{\frac1p}$, which generalize the trace, Frobenius, and operator norms.
They satisfy the following H\"older inequality for $p\geq1$:
\begin{align}\label{eq:holder}
  \abs*{\tr \prod_{i=1}^{p} A_i} \leq \prod_{i=1}^{p} \norm{A_i}_p,
\end{align}

\begin{proof} [Proof of \cref{thm:Pisier-expansion}]
The operator we want to control has entries which are dependent in complicated ways.
We first begin with a standard symmetrization trick to linearize (compare the proof of Lemma~4.1 in~\cite{P14}).
A single entry of $A_i \otimes A_i$ is either a product $g g'$ of two independent standard Gaussians, or the square $g^2$ of a single standard Gaussian.
In expectation, we have $\E g g' = 0, \E g^{2} = 1$, and so the expected matrix is
\[ \E \left( \sum_{i=1}^N A_i \otimes A_i \right) = N \vect(I_n) \vect(I_m)^T. \]
Accordingly, after projection we have
\[ \E \left( \sum_{i=1}^N A_i \otimes A_i \right) \circ \Pi = 0. \]
Therefore we may add an independent copy:
Let $B_1,\dots,B_N$ be independent $n\times m$ random matrices with standard Gaussian entries, that are also independent from~$A_1,\dots,A_N$.
Then,
\begin{align*}
  \left( \sum_{i=1}^N A_i \otimes A_i \right) \circ \Pi
= \E_B \left( \sum_{i=1}^N A_i \otimes A_i - \sum_{i=1}^N B_i \otimes B_i \right) \circ \Pi
\end{align*}
and hence, for any $p\geq1$,
\begin{align*}
  \E \norm*{\left( \sum_{i=1}^N A_i \otimes A_i \right) \circ \Pi}_{\op}^p
\leq \E \norm*{\left( \sum_{i=1}^N A_i \otimes A_i - \sum_{i=1}^N B_i \otimes B_i \right) \circ \Pi}_{\op}^p
\end{align*}
by Jensen's inequality, as $\norm{\cdot}_{\op}^p$ is convex as the composition of the norm $\norm{\cdot}_{\op}$ with the convex and nondecreasing function $x \to x^{p}$.
Now note $(A_i,B_i)$ has the same distribution as $(\frac{A_i+B_i}{\sqrt2},\frac{A_i-B_i}{\sqrt2})$, so the right-hand side is equal to
\begin{align*}
&\quad \E \norm*{ \frac{1}{2} \left( \sum_{i=1}^N (A_i + B_i) \otimes (A_i + B_i) - \sum_{i=1}^N (A_i - B_i) \otimes(A_i - B_i) \right) \circ \Pi}_{\op}^p \\
&= \E \norm*{\left( \sum_{i=1}^N A_i \otimes B_i + \sum_{i=1}^N B_i \otimes A_i \right) \circ \Pi}_{\op}^p
\leq 2^{p} \, \E \norm*{\sum_{i=1}^N A_i \otimes B_i}_{\op}^p
\end{align*}
Thus, we have proved that
\begin{align}\label{eq:symmetrization}
  \E \norm*{\left( \sum_{i=1}^N A_i \otimes A_i \right) \circ \Pi}_{\op}^p
\leq 2^{p} \, \E \norm*{\sum_{i=1}^N A_i \otimes B_i}_{\op}^p.
\end{align}
Note that we have lost the projection and removed the dependencies.
Next we use the trace method to bound the right-hand side of \cref{eq:symmetrization}.
That is, we approximate the operator norm by the Schatten $p$-norm for a large enough $p$ and control these Schatten norms using concentration of moments of Gaussians (compare the proof of Theorem~16.6 in~\cite{pisier2012grothendieck}).
For any $q\geq1$,
\begin{align*}
\E \norm*{\sum_{i=1}^N A_i \ot B_i}_{2q}^{2q}
&= \E \tr \left[ \left( \sum_{i,j\in[N]} A_i^T A_j \ot B_i^T B_j \right)^{\!\!q} \, \right] \\
&= \sum_{i, j \in [N]^q} \E \tr \left( A^T_{i_1} A_{j_1} \cdots A^T_{i_q} A_{j_q} \ot B^T_{i_1} B_{j_1} \cdots B^T_{i_q} B_{j_q} \right) \\
&= \sum_{i, j \in [N]^q} \E \tr \left( A^T_{i_1} A_{j_1} \cdots A^T_{i_q} A_{j_q} \right) \E \tr \left( B^T_{i_1} B_{j_1} \cdots B^T_{i_q} B_{j_q} \right)
\end{align*}
where we used the independence of $\{A_i\}$ and $\{B_i\}$ in the last step.
Now, the expectation of a monomial of independent standard Gaussian random variables is always nonnegative.
Thus the same is true for $\E \tr ( A^T_{i_1} A_{j_1} \cdots A^T_{i_q} A_{j_q} )$, so we can upper bound the sum term by term as
\begin{align*}
&\quad \sum_{i, j \in [N]^q} \E \tr \left( A^T_{i_1} A_{j_1} \cdots A^T_{i_q} A_{j_q} \right) \E \tr \left( B^T_{i_1} B_{j_1} \cdots B^T_{i_q} B_{j_q} \right) \\
&\leq \sum_{i, j \in [N]^q} \E \tr \left( A^T_{i_1} A_{j_1} \cdots A^T_{i_q} A_{j_q} \right) \E \left( \norm{ B_{i_1} }_{2q} \norm{B_{j_1}}_{2q} \cdots \norm{ B_{i_q} }_{2q} \norm{B_{j_q}}_{2q} \right) \\
&\leq \sum_{i, j \in [N]^q} \E \tr \left( A^T_{i_1} A_{j_1} \cdots A^T_{i_q} A_{j_q} \right) \E \left( \norm{ B_1 }_{2q}^{2q} \right) \\
&= \left( \E\norm{\sum_{i=1}^N A_i}_{2q}^{2q} \right) \left( \E \norm{ A }_{2q}^{2q} \right)
= N^q \left( \E\norm{A}_{2q}^{2q} \right)^2.
\end{align*}
In the first step we used H\"older's inequality~\eqref{eq:holder} for the Schatten norm.
The second step holds since $\E\norm{B_i}_{2q}^k \leq (\E \norm{B_i}_{2q}^{2q})^{\frac k {2q}}$ by Jensen's inequality, so we can collect like terms together.
Next, we used that the~$B_i$ have the same distribution as $A$.
In the last step, we used that $\sum_i A_i$ has the same distribution as $\sqrt N A$.
Accordingly, we obtain for~$q\geq1$,
\begin{align*}
  \E \norm*{\sum_{i=1}^N A_i \ot B_i}_{2q}^{2q}
\leq N^q \left( \E\norm{A}_{2q}^{2q} \right)^2,
\end{align*}
and hence
\begin{align*}
\E \norm*{\left( \sum_{i=1}^N A_i \otimes A_i \right) \circ \Pi}_{\op}^{2q}
&\leq 4^q \, \E \norm*{\sum_{i=1}^N A_i \ot B_i}_{\op}^{2q}
\leq 4^q \, \E \norm*{\sum_{i=1}^N A_i \ot B_i}_{2q}^{2q} \\
&\leq (4N)^q \left( \E\norm{A}_{2q}^{2q} \right)^2
\leq (4N)^q m^2 \Bigl( \E\norm{A}_{\op}^{2q} \Bigr)^2.
\end{align*}
The first inequality is \cref{eq:symmetrization}, and in the last inequality we used that $A \in \Mat(n,m)$ has rank~$\leq m$, and therefore $\norm{A}_{2q}^{2q} \leq m \norm A_{\op}^{2q}$.
To bound the right-hand side, we use \cref{thm:banach conc}, applied to the random variable~$A$ in the Banach space~$\Mat(n,m)$ with the operator norm~$\norm{\cdot}_{\op}$.
Then, $\sigma^2=1$, as computed in \cref{lem:opNormSubG}, and we find that
\begin{align*}
\E \norm*{\left( \sum_{i=1}^N A_i \otimes A_i \right) \circ \Pi}_{\op}^{2q}
\leq (4 N)^q m^2 \Bigl( (2 \E\norm A_{\op})^{2q} + (C \sqrt{q})^{2q} \Bigr)^{2}.
\end{align*}
where~$C>0$ is a universal constant implied by the big-$O$ notation in \cref{eq:conc via moments}.
We can bound the first term $\E \norm{A}_{\op} \leq \sqrt{m}+\sqrt{n}$ by \cref{eq:op norm upper bound}, so for $q = 2(m + n)$, we can upper bound the mean by
\begin{align*}
\E \norm*{\left( \sum_{i=1}^N A_i \otimes A_i \right) \circ \Pi}_{\op}^{2q}
&\leq 4 m^2 \Big( (\max\{2, C\})^{2} \cdot q \cdot \sqrt{ 4 N }\Big)^{2q} .
\end{align*}
Finally, we can use Markov's inequality to see that, for $C' = \sqrt{2} \max\{2,C\}$, the event
\begin{align}\label{eq:desired event}
  \norm*{\left(\sum_{i=1}^N A_i \otimes A_i \right) \circ \Pi}_{\op}
  \leq (C' t)^2 \cdot (m+n) \cdot \sqrt{4 N}
\end{align}
holds up to failure probability at most
\begin{align*}
  4 m^2 \left( \frac{(\max\{2, C\})^{2} \cdot q \cdot \sqrt{ 4 N } }{(C' t)^2 \cdot (m+n) \cdot \sqrt{4 N}} \right)^{2q}
  \leq 4 m^2 t^{-2q} \leq t^{-\Omega(m+n)} ,
\end{align*}
where the first step was by our choice of $q = 2(m+n)$ and of $C' = \sqrt{2} \max\{2,C\}$, and the final inequality was by the fact that $t\geq2$, so the prefactor $4 m^{2}$ can be absorbed at the cost of slightly changing the constant in the exponent.
\end{proof}

\subsection{Expansion from Cheeger}\label{app:cheeky}
We now prove \CREFmain{thm:operator-cheeger}{Theorem~3.1}, which asserts that a random completely positive map with sufficiently many Kraus operators is an almost quantum expander with exponentially small failure probability.
To prove the theorem, we first define the Cheeger constant of completely positive map.
This is similar to a concept defined in \cite{H07}.

\begin{definition}[Cheeger constant]\label{def:cheeger}
Let $\Phi \colon \Mat(d_2) \to \Mat(d_1)$ be a completely positive map.
The \emph{Cheeger constant} $\ch(\Phi)$ is given by
\begin{align*}
  \ch(\Phi) := \min_{\substack{\Pi_1, \Pi_2 \\ 0 < \vol(\Pi_1, \Pi_2) \leq \frac12 \vol(I_{d_1}, I_{d_2})}} \phi(\Pi_1,\Pi_2)
\end{align*}
where $\Pi_1\colon \C^{d_1} \to \C^{d_1}$ and $\Pi_2\colon \C^{d_2} \to \C^{d_2}$ are orthogonal projections, and the \emph{conductance}~$\phi(\Pi_1, \Pi_2)$ of the ``cut'' $\Pi_1, \Pi_2$ is defined to be
\begin{align*}
  \phi(\Pi_1,\Pi_2) := \frac{\cut(\Pi_1, \Pi_2)}{\vol(\Pi_1,\Pi_2)},
\end{align*}
where
\begin{align*}
  \vol(\Pi_1,\Pi_2) &:= \tr \Phi(\Pi_2) + \tr \Phi^*(\Pi_1), \\
  \cut(\Pi_1, \Pi_2) &:= \tr (I_{d_1} - \Pi_1) \Phi(\Pi_2) + \tr \Pi_1 \Phi(I_{d_2} - \Pi_2) \\
  &= \tr (I_{d_1} - \Pi_1) \Phi(\Pi_2) + \tr \Phi^*(\Pi_1) (I_{d_2} - \Pi_2).
\end{align*}
\end{definition}

The key connection that we will leverage to prove \CREFmain{thm:operator-cheeger}{Theorem~3.1} is that a large Cheeger constant implies quantum expansion, proved in \cite[Remark 5.5]{FM20}:

\begin{lemma} [Cheeger and expansion]\label{lem:op-cheeger}
There exist absolute constants~$c, C>0$ such that if $\Phi$ is a completely positive map that is $\eps$-doubly balanced for some~$\eps < c \ch(\Phi)^2$, then $\Phi$ is an $(\eps,\eta)$-quantum expander, where
\begin{align*}
  \eta = \max\left\{ \frac12, 1 -  \ch(\Phi)^2 + C \frac{\eps}{\ch(\Phi)^2} \right\}.
\end{align*}
\end{lemma}

For intuition, consider a weighed bipartite graph $G$ on $[d_1] \cup [d_2]$.
The projections~$\Pi_1$ and~$\Pi_2$ are analogous to subsets~$A \subset[d_1]$ and~$B \subset [d_2]$, respectively.
The quantity $\vol(\Pi_1, \Pi_2)$ is analogous to the total mass of the edges adjacent to $A$ plus that of the edges adjacent to $B$, which is the volume of $A \cup B$ considered as a cut of $G$.
The quantity $\cut(\Pi_1, \Pi_2)$ corresponds to the total mass of the edges between~$A \cup B$ and its complement, that is, to the weight of the cut defined by~$A \cup B$.
In fact, if the Cheeger constant were defined with~$\Pi_1$ and~$\Pi_2$ restricted to be coordinate projections, it would be exactly the Cheeger constant of the bipartite graph on $[d_1]$ and $[d_2]$ with edge $(i,j)$ weighted by $\tr e_i e_i^T \Phi(e_j e_j^T)$, and the volume and the cut would be the same as the volume and the cut on that bipartite graph.

For the remainder of this section let $\Phi = \Phi_X$ where $X_1, \dots, X_n$ are random $d_1 \times d_2$ matrices with independent standard Gaussian entries.
In this case, each edge-weight $\tr e_i e_i^T \Phi(e_j e_j^T)$ of the bipartite graph is an independent $\chi^2$ random variable with~$n$ degrees of freedom.
Accordingly:

\begin{lemma}\label{fact:chi}
Let $\Pi_1\colon\C^{d_1} \to \C^{d_1}$, $\Pi_2\colon \C^{d_2} \to \C^{d_2}$ be orthogonal projections of rank~$r_1$ and~$r_2$, respectively.
Then $\cut(\Pi_1, \Pi_2)$, $\vol(\Pi_1, \Pi_2)$, $\vol(I_{d_1}, I_{d_2})$ are jointly distributed as
\begin{align*}
  R_1, \, R_1 + 2R_2, \, 2R_1 + 2 R_2 + 2R_3,
\end{align*}
where $R_1$, $R_2$, $R_3$ are independent $\chi^2$ random variables with $F_1:=n r_1(d_2 - r_2) + n r_2(d_1-r_1)$, $F_2:= n r_1r_2$, and $F_3:= n(d_1 - r_1)(d_2 - r_2)$ degrees of freedom, respectively.
\end{lemma}
\begin{proof} As the distribution of $\Phi_X$ is invariant under the action $(U_1, U_2) \cdot \Phi_X(Y) =  U_1\Phi_X( U_2^T Y U_2) U_1^T$ of unitary matrices $U_1, U_2$, the joint distribution of $\cut(\Pi_1, \Pi_2)$, $\vol(\Pi_1, \Pi_2)$ depends only on the rank of~$\Pi_1$, $\Pi_2$. Thus we may compute in the case that $\Pi_1$ and $\Pi_2$ are coordinate projections, in which case one may directly verify the fact; see the discussion above.
\end{proof}

We next show a sufficient condition for the Cheeger constant being bounded away from $1$ that is amenable to the previous lemma.

\begin{lemma}\label{lem:suff}
Let $\Phi$ be a completely positive map and $\delta<0.005$ such that the following hold for all orthogonal projections $\Pi_1\colon \C^{d_1} \to \C^{d_1}$, $\Pi_2 \colon \C^{d_2}\to\C^{d_2}$, not both zero, where we denote by~$r_1$,~$r_2$ the ranks of $\Pi_1$ and $\Pi_2$, respectively, and abbreviate $F_1 := n r_1(d_2 - r_2) + n r_2(d_1-r_1)$, and $F_2 := n r_1 r_2$ as in \cref{fact:chi}:
\begin{enumerate}
\item If $F_2 \geq \frac49 n d_1 d_2$, then
\begin{align}\label{eq:vol}
  \vol(\Pi_1, \Pi_2)
\geq \left(\frac{101}{200} - \delta\right) \vol(I_{d_1}, I_{d_2})
= \left(1.01 - 2\delta\right) \tr \Phi(I_{d_2}).
\end{align}
\item If $F_2 < \frac49 n d_1 d_2$ and $\vol(\Pi_1, \Pi_2)>0$, then
\begin{align}\label{eq:cut}
  \vol(\Pi_1, \Pi_2) \leq \left(\frac43 + \delta\right)\left(F_1 + 2 F_2\right) \quad\text{and}\quad
  \cut(\Pi_1, \Pi_2) \geq \left(\frac23 - \delta\right) F_1.
\end{align}
\end{enumerate}
Then $\ch(\Phi) \geq \frac16 - O(\delta)$.
\end{lemma}
\begin{proof}
The first assumption implies we only need to reason about the case that $F_2 < \frac49 n d_1 d_2$.
This is because the minimization in the definition of the Cheeger constant is over $\Pi_1$, $\Pi_2$ such that $\vol(\Pi_1, \Pi_2) \leq \tr \Phi(I_{d_2})$.
Therefore, the second assumption implies that
\begin{align*}
  \ch(\Phi)
\geq \frac{\frac43 + \delta}{\frac23 - \delta} \min_{F_2 < \frac49 n d_1 d_2} \frac{F_1}{F_1 + 2 F_2}
= \left( \frac12 - O(\delta) \right) \min_{r_1 r_2 < \frac49 d_1 d_2} \frac{F_1}{F_1 + 2 F_2}.
\end{align*}
It suffices to show that $F_1/(F_1 + 2 F_2) \geq 1/3$ provided $r_1 r_2 < \frac49 d_1 d_2$.
Indeed, if either $r_1 = 0$ or $r_2 = 0$, then $F_2 = 0$ and $F_1>0$ and the claim holds.
Otherwise, if $r_1>0$ and $r_2>0$, then
\begin{align*}
  \frac{F_1}{F_1 + 2 F_2}
&= \frac{r_1 d_2 + r_2 d_1 - 2 r_1 r_2}{r_1 d_2 + r_2 d_1}
= 1 - \frac{2 r_1 r_2}{r_1 d_2 + r_2 d_1} \\
&= 1 - \sqrt{\frac{r_1r_2}{d_1d_2}} \frac2{\sqrt{\frac{r_1 d_2}{r_2 d_1}} + \sqrt{\frac{r_2 d_1}{r_1 d_2}}}
\geq 1 - \sqrt{\frac49}
= \frac13.
\end{align*}
In the last inequality we used that $a + a^{-1} \geq 2$ for all $a>0$ and that $r_1 r_2 < \frac49 d_1 d_2$.
\end{proof}

Next we use \cref{fact:chi} to show that for a random completely positive map, the events in \cref{lem:suff} hold with high probability for any \emph{fixed} $\Pi_1$ and $\Pi_2$. We also need a third bound which we will use to transfer properties of a $\delta$-net to the whole space of projections.

\begin{lemma}\label{lem:probabilities}
Suppose $d_1 \leq d_2$. Let $\Pi_1\colon \C^{d_1} \to \C^{d_1}$, $\Pi_2\colon \C^{d_2} \to \C^{d_2}$ be orthogonal projections of rank~$r_1$ and~$r_2$, respectively such that $r_1 + r_2 > 0$.
Let $F_1:= n r_1(d_2 - r_2) + n r_2(d_1-r_1)$ and $F_2 = n r_1 r_2$.
Then, the following holds for the random completely positive map $\Phi=\Phi_X$:
\begin{enumerate}
\item If $F_2 \geq \frac49 n d_1 d_2$, then \cref{eq:vol} holds with $\delta = 0$ with probability at least $1 - e^{-\Omega( n d_1 d_2)}$.
\item If $F_2 < \frac49 n d_1 d_2$, then \cref{eq:cut} holds with $\delta = 0$ with probability at least $1 - e^{-\Omega( F_1)}$.
\item Finally, $\vol(\Pi_1, \Pi_2) \geq \frac{1}{4d_2}  \vol(I_{d_1}, I_{d_2})$ with probability at least $1 - e^{- \Omega(F_1 + 2F_2)}$.
\end{enumerate}
\end{lemma}
\begin{proof}
Recall from \cref{fact:chi} that $\cut(\Pi_1, \Pi_2)$, $\vol(\Pi_1, \Pi_2)$, $\vol(I_{d_1}, I_{d_2})$ are jointly distributed as $R_1$, $R_1 + 2R_2$, $2R_1 + 2R_2 + 2R_3$ for $R_1, R_2, R_3$ independent $\chi^2$ random variables with $F_1$, $F_2$, and $F_3 = n(d_1-r_1)(d_2-r_2)$ degrees of freedom, respectively.
In view of \cref{eq:vol,eq:cut} with $\delta=0$, it is thus enough to show that
\begin{enumerate}
\item If $F_2 \geq \frac49 n d_1 d_2$, then $R_2 \geq \frac1{99} R_1 + \frac{101}{99} R_3$ with probability $1 - e^{-\Omega( n d_1 d_2)}$.
\item If $F_2 < \frac49 n d_1 d_2$, then $R_1 + 2R_2 \leq \frac43 (F_1 + 2F_2)$ and $R_1 \geq \frac23 F_1$ hold with probability $1 - e^{-\Omega(F_1)}$.
\item $R_1 + 2 R_2 \geq \frac23 (F_1 + 2F_2)$ and $R_1 + R_2 + R_3 \leq \frac43 (F_1 + F_2 + F_3)$ with probability $1 - e^{- \Omega(F_1 + 2F_2)}$.
\end{enumerate}
Indeed, the first (resp.\ second) claim above implies the first (resp.\ second) claim in the lemma by substituting the expressions for $\vol$ and $\cut $ of $(\pi_1, \pi_2)$ and $(I_{d_1}, I_{d_2})$ in terms of $R_1, R_2, R_3$. The last claim follows from the same reasoning combined with the inequality $F_1 + 2 F_2 \geq \frac{1}{d_2} (F_1 + F_2 + F_3)$ for $r_1, r_2$ not both zero and the fact that $d_1 \leq d_2$.

All three follow from standard results for concentration of $\chi^2$ random variables; see e.g.~\cite{W19}.
We only prove the first claim; the second and third claims are straightforward.
To prove the first claim, we first reason about the case when one of $r_1  = 0$.

note that $F_1 + 2 F_2 \geq \frac43(F_1 + F_2 + F_3)$, because
\begin{align*}
  \frac{F_1 + 2 F_2}{F_1 + F_2 + F_3}
= \frac{r_1}{d_1} + \frac{r_2}{d_2}
= \sqrt{\frac{r_1r_2}{d_1d_2}} \left( \sqrt{\frac{r_1 d_2}{d_1 r_2}} + \sqrt{\frac{d_1 r_2}{r_1 d_2}} \right)
\geq \sqrt{\frac49} \cdot 2 = \frac43.
\end{align*}
In particular, $F_2 \geq \frac23(F_2 + F_3)$ and $F_2 \geq \frac16(F_1 + F_2)$. 

We first reason about the ratio between $R_2$ and $R_3$ using the first inequality.
With probability $1 - e^{- c F_2} \geq 1 - e^{- \Omega(n d_1 d_2)}$, it holds that $R_2 \geq \frac89 F_2$ and $R_2 + R_3 \leq \frac{10}9 (F _2 + F_3)$. The latter holds because $R_2 + R_3$ is a $\chi^2$ random variable with $F_2 + F_3 \geq F_2$ degrees of freedom. so $R_2 \geq \frac89 \frac23 \frac9{10} (R_2 + R_3) = \frac8{15} (R_2 + R_3)$, or $R_2 \geq \frac87 R_3$.
We next apply the same reasoning with the inequality $F_2 \geq (F_1 + F_2)/6$ to estimate the ratio between $R_1$ and $R_2$.
With probability $1 - e^{- c n d_1 d_2}$, we have $R_2 \geq \frac89 F_2$ and $R_1 + R_2 \leq \frac{10}9 (F_1 + F_2)$.
Thus $R_1 \geq \frac89 \frac16 \frac9{10} (R_1 + R_2) = \frac4{30} (R_1 + R_2)$, or $R_2 \geq \frac2{13} R_1$.
Together, we obtain that $R_2 \geq \frac1{99} R_1 + \frac{101}99 R_3$ with probability~$1 - e^{-\Omega( n d_1 d_2)}$.
\end{proof}

Finally, a net argument shows that the Cheeger constant is large for \emph{all} projections.

\begin{lemma}[\cite{FM20}, Lemma~5.18]\label{lem:net}
For any $\eps>0$, there is an operator norm $\eps$-net of the rank-$r$ orthogonal projections $\Pi\colon \C^d \to \C^d$ with cardinality $e^{O(d r \abs{\log\eps})}$.
\end{lemma}

As a corollary, the set of pairs of projections $\Pi_1$, $\Pi_2$ of rank $r_1$ and $r_2$, respectively, has an (elementwise) operator norm $\eps$-net of cardinality $e^{O((r_1 d_1 + r_2 d_2) \abs{\log \eps})}$.

\begin{lemma}[Continuity of cut and volume]\label{lem:net-suffices}
Let $\Pi_1,\Pi'_1 \colon \C^{d_1}\to\C^{d_1}$ and $\Pi_2,\Pi'_2\colon\C^{d_2}\to\C^{d_2}$ be orthogonal projections such that $\norm{\Pi_1 - \Pi'_1}_{\op}\leq\eps$ and $\norm{\Pi_2 - \Pi'_2}_{\op}\leq\eps$.
Then:
\begin{align*}
  \abs{\cut(\Pi'_1,\Pi'_2) - \cut(\Pi_1,\Pi_2)}
\leq 2 \eps \vol(I_{d_1}, I_{d_2}) \\
  \quad\text{and}\quad
  \abs{\vol(\Pi'_1,\Pi'_2) - \vol(\Pi_1,\Pi_2)}
\leq 2 \eps \vol(I_{d_1}, I_{d_2}).
\end{align*}
\end{lemma}
\begin{proof}
We begin with the first inequality:
\begin{align*}
  \abs{\cut(\Pi'_1,\Pi'_2) - \cut(\Pi_1,\Pi_2)}
&\leq \abs{\tr \Pi'_1 \Phi(I_{d_2} - \Pi'_2) - \tr \Pi_1 \Phi(I_{d_2} - \Pi_2)} \\
&+ \abs{\tr (I_{d_1} - \Pi'_1) \Phi(\Pi'_2) - \tr (I_{d_1} - \Pi_1) \Phi(\Pi_2)}.
\end{align*}
Consider the first term:
\begin{align*}
&\quad \abs{\tr \Pi'_1 \Phi(I_{d_2} - \Pi'_2) - \tr \Pi_1 \Phi(I_{d_2} - \Pi_2)} \\
&= \abs{\tr (\Pi'_1 - \Pi_1) \Phi(I_{d_2} - \Pi'_2) + \tr \Pi_1 \Phi(\Pi_2 - \Pi'_2)} \\
&= \abs{\tr (\Pi'_1 - \Pi_1) \Phi(I_{d_2} - \Pi'_2) + \tr \Phi^*(\Pi_1) (\Pi_2 - \Pi'_2)} \\
&\leq \norm{\Pi'_1 - \Pi_1}_{\op} \norm{\Phi(I_{d_2} - \Pi'_2)}_{\operatorname{tr}} + \norm{\Pi_2 - \Pi'_2}_{\op} \norm{\Phi^*(\Pi_1)}_{\operatorname{tr}} \\
&\leq \eps \tr \Phi(I_{d_2}) + \eps \tr \Phi^*(I_{d_1})
= \eps \vol(I_{d_1}, I_{d_2}).
\end{align*}
The same inequality for the second term follows by symmetry.
The proof of the second inequality is similar.
\end{proof}

\begin{prop}[Cheeger constant lower bound]\label{lem:union}
There is a universal constant $C>0$ such that the following holds:
If $d_1 \leq d_2$, $d_2>1$, and $n \geq C \frac{d_2}{d_1} \log (d_2)$, then $\ch(\Phi) = \Omega(1)$ with probability $1 - e^{- \Omega(n d_1)}$.
\end{prop}
\begin{proof}
Let $\eps = \frac c {d_2}$ for some sufficiently small constant~$c>0$.
For $r_1 \leq d_1$ and $r_2 \leq d_2$ not both zero, let~$\cN(r_1,r_2)$ denote an (elementwise) operator norm $\eps$-net for the set of pairs of projections of rank~$r_1$ and~$r_2$, respectively.
As discussed below \cref{lem:net}, we may assume that $\abs{\cN(r_1,r_2)} \leq e^{O((d_1r_1 + d_2r_2) \abs{\log\eps})}$.
Let $\mathcal N = \smash{\bigcup_{r_1,r_2}} \mathcal N(r_1,r_2)$.
We claim that to establish the lemma it suffices to show that with probability $1 - e^{- \Omega(n d_1)}$, the following is simultaneously true for all $r_1,r_2$ and for all $(\Pi_1,\Pi_2) \in \mathcal N(r_1,r_2)$:
\begin{enumerate}
\item If $F_2 := n r_1 r_2 \geq \frac49 n d_1 d_2$, then \cref{eq:vol} holds with $\delta=0$.
\item If $F_2 < \frac49 n d_1 d_2$, then \cref{eq:cut} holds with $\delta=0$.
\item $\vol(\Pi_1,\Pi_2) \geq \frac1{4d_2} \vol(I_{d_1}, I_{d_2})$.
\end{enumerate}
To see that this suffices, we only need to show that it implies the hypotheses of \cref{lem:suff} for~$\delta = 32 c$.
Let $(\Pi_1',\Pi_2')$ be an arbitrary pair of projections, not both zero.
Let $r_1$ and $r_2$ denote their ranks.
Then there exists a pair $(\Pi_1,\Pi_2) \in \cN(r_1,r_2)$ such that $\norm{\Pi'_1 - \Pi_1}_{\op} \leq \eps$ and~$\norm{\Pi'_2 - \Pi_2}_{\op} \leq \eps$.
If $F_2 \geq \frac49 n d_1 d_2$,
\begin{align*}
  \vol(\Pi'_1,\Pi'_2)
&\geq \vol(\Pi_1,\Pi_2) - 2 \eps \vol(I_{d_1}, I_{d_2}) \\
&\geq \left( \frac{101}{200} - 2 \eps \right) \vol(I_{d_1}, I_{d_2})
\geq \left( \frac{101}{200} - 2 c \right) \vol(I_{d_1}, I_{d_2}),
\end{align*}
where the first inequality is \cref{lem:net-suffices}, the second uses the first claim above, and finally we estimate $\eps \leq c$.
Thus we have verified that \cref{eq:vol} holds for $(\Pi'_1,\Pi'_2)$, that is, the first hypothesis of \cref{lem:suff}.
If $F_2 < \frac49 n d_1 d_2$, then
\begin{align*}
  \vol(\Pi'_1,\Pi'_2)
&\leq \vol(\Pi_1,\Pi_2) + 2 \eps \vol(I_{d_1}, I_{d_2})
\leq (1 + 8 \eps d_2) \vol(\Pi_1,\Pi_2) \\
&\leq (1 + 8 \eps d_2) \frac43 \left( F_1 + 2 F_2 \right)
= \left( \frac43 + \frac{32}3 c \right) \left( F_1 + 2 F_2 \right),
\end{align*}
where $F_1 := n r_1(d_2 - r_2) + n r_2(d_1-r_1)$, the first inequality is \cref{lem:net-suffices}, the second inequality uses the third claim above, and the third inequality uses the second claim above.
On the other hand,
\begin{align*}
  \cut(\Pi'_1,\Pi'_2)
&\geq \cut(\Pi'_1,\Pi'_2) - 2 \eps \vol(I_{d_1}, I_{d_2})
\geq \cut(\Pi'_1,\Pi'_2) - 8 \eps d_2 \vol(\Pi_1,\Pi_2) \\
&\geq \frac23 F_1 - 8 \eps d_2 \frac43 (F_1 + 2 F_2)
= \frac23 F_1 - 8 c \frac43 (F_1 + 2 F_2) \\
&\geq \frac23 F_1 - 32 c F_1
= \left( \frac23 - 32 c \right) F_1,
\end{align*}
again using \cref{lem:net-suffices}, the second and third claim above.
In the last step we used the fact that $F_1 \geq \frac{1}{3} (F_1 + 2 F_2)$ provided $F_2 < \frac49 n d_1 d_2$, which we established in the proof of \cref{lem:suff}.
Thus we have verified that \cref{eq:cut} holds for $(\Pi'_1,\Pi'_2)$, which is the remaining hypotheses of \cref{lem:suff}.

To prove the lemma we still need to show that the three conditions above hold with the desired probability.
We show that for fixed $r_1$ and $r_2$, each condition holds with probability at least $1 - e^{\Omega(n (r_1 d_2 + d_1 r_2))}$.
By the union bound, this implies that the conditions hold simultaneously for all $r_1 \leq d_1$ and $r_2 \leq d_2$, not both zero, with the desired probability, because the sum of $e^{-\Omega(n (r_1 d_2 + d_1 r_2))}$ over all such~$r_1$ and~$r_2$ is $e^{- \Omega( n d_1)}$, using that $d_1 \leq d_2$.
Thus fix~$r_1$ and~$r_2$ as above.
We first bound the probability for the first claim.
By \cref{lem:probabilities} and the union bound, if $F_2 \geq \frac49 n d_1 d_2$ then \cref{eq:vol} holds for every $(\Pi_1,\Pi_2) \in \cN(r_1, r_2)$ with probability
\begin{align*}
1 - \abs{\cN(r_1, r_2)} e^{- \Omega( n d_1 d_2) } &\leq 1 - e^{(d_1r_1 + d_2r_2) \abs{\log\eps}} e^{ - \Omega(n (r_1 d_2 + d_1 r_2))}
\leq 1 - e^{ - \Omega(n (r_1 d_2 + d_1 r_2))}.
\end{align*}
The last step follows by our assumption on~$n$ (for a large enough universal constant $C>0$), since
\begin{align*}
  (d_1r_1 + d_2r_2) \abs{\log\eps}
\leq \frac{d_2}{d_1}(r_1 d_2 + d_1 r_2) \left( \log(d_2) + \abs{\log c} \right)
= O\left(\frac{d_2}{d_1} \log(d_2) \right) \cdot (r_1 d_2 + d_1 r_2).
\end{align*}
Next we bound the probability for the second claim.
By \cref{lem:probabilities} and the union bound, if $F_2 < \frac49 n d_1 d_2$, \cref{eq:cut} holds for every $(\Pi_1,\Pi_2) \in \cN(r_1, r_2)$ with probability
\begin{align*}
  1 - \abs{\cN(r_1, r_2)} e^{-\Omega(F_1)}
\leq 1 - \abs{\cN(r_1, r_2)} e^{-\Omega(n (r_1 d_2 + r_2 d_1))}
\leq 1 - e^{ - \Omega(n (r_1 d_2 + d_1 r_2))},
\end{align*}
where the first step holds since $F_1 \geq \frac{1}{3} (F_1 + 2 F_2) = \frac13 n (r_1 d_2 + r_2 d_1)$ whenever $F_2 < \frac49 n d_1 d_2$, as already used earlier in the proof, and the second step follows as above by our assumption on~$n$ (for large enough $C>0$).
The probability for the third claim can be bounded completely analogously.
\end{proof}

\begin{proof}[Proof of \CREFmain{thm:operator-cheeger}{Theorem~3.1}]
Let $\Phi:=\Phi_X$.
Since~$n \geq C \frac{d_2}{d_1} \log d_2$, \cref{lem:union} shows that $\ch(\Phi) = \Omega(1)$ with failure probability $e^{- \Omega(n d_1)}$.
The latter is $e^{- \Omega(d_2 t^2)}$ using our assumption that $n \geq C \frac{d_2}{d_1} t^2$.

Now let $\eps := \smash{t \sqrt{\frac{d_2}{n d_1}}}$, which by the same assumption satisfies $\eps \leq \smash{\frac1{\sqrt C}}$.
Moreover, $n \geq \frac {\dmax^2} {D \eps^2}$, since this is equivalent to our assumption that $t\geq1$.
Therefore, if we choose~$C$ sufficiently large then, similarly to the proof of \CREFmain{thm:tensor-convexity}{Proposition~2.17}, we find using \CREFmain{prop:gradient-bound}{Proposition~2.11} and \cref{prp:xnorm} that $\Phi$ is $\eps$-doubly balanced with failure probability $e^{-\Omega(nD)} + O(e^{-\Omega(n d_1 \eps^2)}) \leq e^{-\Omega(d_2 t^2)}$.

By making~$C$ larger, we can ensure that~$\eps$ is less than any absolute constant.
Then \cref{lem:op-cheeger} applies (with balancedness~$\eps$) and shows that $\Phi$ is an $(\eps,\eta)$-quantum expander for some absolute constant~$\eta<1$.
\end{proof}

\ifdefined\ARXIV
\section{Proofs of results in Section~\ref{sec:sample-complexity} and Theorem~\ref{thm:tensor-frobenius}}\label{app:tensor}
\else
\section{Proofs of results in \texorpdfstring{\CREFmain{sec:sample-complexity}{Section~2} and \CREFmain{thm:tensor-frobenius}{Theorem~1.10}}{Section 2 and Theorem 1.10}}\label{app:tensor}
\fi

We first prove \CREFmain{lem:convex-ball}{Lemma~2.7}, which states strong convexity in a ball about a point where the gradient is sufficiently small implies the optimizer cannot be far.

\begin{proof}[Proof of \CREFmain{lem:convex-ball}{Lemma~2.7}]
We first show that the sublevel set of~$f(\Theta)$ is contained in the ball of radius~$\frac{2\delta}\lambda$.
Consider $g(t) := f(\exp_\Theta(tH))$, where~$H\in\H$ is an arbitrary vector of unit norm~$\norm H_F = 1$.
Then, using the assumption on the gradient,
\begin{align}\label{eq:grad bound}
  g'(0)
= \partial_{t=0} f(\exp_\Theta(tH))
= \braket{\nabla f(\Theta), H}
\geq -\norm{\nabla f(\Theta)}_F \norm H_F
\geq -\delta.
\end{align}
Since $f$ is $\lambda$-strongly geodesically convex on $B_r(\Theta)$, we have $g''(t) \geq \lambda$ for all $\abs t\leq r$.
It follows that for all $0 \leq t \leq  r$ we have
\begin{align}\label{eq:g convex lower}
  g(t) \geq g(0) - \delta t + \frac12 \lambda t^2.
\end{align}
Plugging in $t = r$ yields
$g(r) \geq  
g(0) + \left( \frac{\lambda r}2 - \delta \right)  r
> g(0)$.
Since $g$ is convex due to the geodesic convexity of $f$, it follows that, for any~$t\geq  r$,
\begin{align*}
  g(0) < g( r) \leq \left( 1-\frac{ r}t \right) g(0) + \frac{ r}t g(t),
\end{align*}
hence
\begin{align*}
  f(\Theta) = g(0) < g(t) = f(\exp_\Theta(tH)).
\end{align*}
Thus, since $H$ was an arbitrary unit norm tangent vector, the sublevel set of~$f(\Theta)$ is contained in the ball of radius~$r$ about~$\Theta$.
By replacing~$r$ with any smaller~$r'>\frac{2\delta}\lambda$, we see that the sublevel set is in fact contained in the closed ball of radius~$\frac{2\delta}\lambda$.
In particular, the minimum of $f$ is attained and any minimizer~$\smash{\htheta}$ is contained in this ball.
Moreover, as the right-hand side of \cref{eq:g convex lower} takes a minimum at $t=\frac\delta\lambda$, we have $g(t) \geq g(0) - \frac{\delta^2}{2\lambda}$ for all~$0\leq t\leq r$.
By definition of $g$, this implies that $f(\htheta) \geq f(\Theta) - \frac{\delta^2}{2 \lambda}$.

Next, we prove that any minimizer of~$f$ is necessarily contained in the ball of radius~$\frac\delta\lambda$.
To see this, take an arbitrary minimizer~$\htheta$ and write it in the form $\htheta = \exp_\Theta(TH)$, where~$H\in \H$ is a unit vector and~$T>0$.

As before, we consider the function $g(t) = f(\exp_\Theta(tH))$.
Then, using \cref{eq:grad bound}, the convexity of~$g(t)$ for all $t\in\R$ and the $\lambda$-strong convexity of~$g(t)$ for~$\abs t \leq  r$, we have
\begin{align*}
  0 = g'(T) = g'(0) + \int_0^T g''(t) \, dt \geq \lambda \min(T,  r) - \delta.
\end{align*}
If $T> r$ then we have a contradiction as $\lambda r - \delta > \lambda r/2 - \delta > 0$.
Therefore we must have~$T\leq r$ and hence $\lambda T - \delta \leq 0$, so $T \leq \frac\delta\lambda$.
Thus we have proved that any minimizer of $f$ is contained in the ball of radius~$\frac\delta\lambda$.

We still need to show that the minimizer is unique; that this follows from strong convexity is convex optimization ``folklore,'' but we include a proof nonetheless.
Indeed, suppose that $\htheta$ is a minimizer and let $H\in \H$ be arbitrary.
Consider $h(t) := f(\exp_{\htheta}(tH))$.
Then the function $h(t)$ is convex, has a minimum at $t=0$, and satisfies $h''(0) > 0$, since $f$ is $\lambda$-strongly geodesically convex near~$\htheta$, as $\htheta \in B_r(\Theta)$ by what we showed above.
It follows that $h(t) > h(0)$ for any~$t\neq0$.
Since $H$ was arbitrary, this shows that $f(\Upsilon) > f(\htheta)$ for any $\Upsilon\neq \htheta$.
\end{proof}

Next we prove \CREFmain{lem:gradient}{Lemma 2.9}, which computes the gradient in terms of partial traces.

\begin{proof}[Proof of \CREFmain{lem:gradient}{Lemma 2.9}]
For all $a\in[k]$ and any traceless symmetric $d_a\times d_a$ matrix~$H$,
\begin{align*}
\braket{\nabla_a f_x(I_D), H}
&= \partial_{t=0} f_x(e^{t\sqrt{d_a} H_{(a)}})
= \partial_{t=0} \left( \tr \rho \, e^{t\sqrt{d_a} H_{(a)}} - \frac1D\log\det(e^{t\sqrt{d_a} H_{(a)}}) \right) \\
&= \sqrt{d_a} \tr \rho \, H_{(a)}
= \sqrt{d_a} \tr \rho^{(a)} \, H
\end{align*}
using \CREFmain{eq:obj via rho}{Eq.~(2.2)} and \CREFmain{eq:partial trace duality}{Eq.~(2.3)} and that $\tr H_{(a)} = 0$ since $\tr H = 0$.
Since $\nabla_a f_{\samp}$ is traceless and symmetric by definition, while $\rho^{(a)}$ is symmetric, this implies that $\nabla_f f_{\samp}$ is the orthogonal projection of $\rho^{(a)}$ onto the traceless matrices, i.e.,
\begin{align*}
  \nabla_a f_{\samp}
= \sqrt{d_a} \left( \rho^{(a)} - \frac{\tr \rho^{(a)}}{d_a} I_{d_a} \right)
= \sqrt{d_a} \left( \rho^{(a)} - \frac{\tr \rho}{d_a} I_{d_a} \right).
\end{align*}
Finally,
\[
  \nabla_0 f_x
= \partial_{t=0} \left( \tr \rho e^t - \frac1D \log \det(e^t I_D) \right)
= \partial_{t=0} \left( \tr \rho e^t - t \right)
= \tr \rho - 1.
\]
\end{proof}

To prove \CREFmain{prop:gradient-bound}{Proposition~2.11} we will need a standard result in matrix concentration.
By the discussion below \CREFmain{definition:partial-trace}{Definition~2.8}, when the samples $x=(x_1,\dots,x_n)$ are independent standard Gaussians in $\R^D$, then $\rho^{(a)}$ is distributed as $\frac1{nD} Y Y^T$, where~$Y$ is a random $d_a \times N_a$ matrix with independent standard Gaussian entries, where~$N_a = \frac{nD}{d_a}$.
The following result bounds the singular values of such random matrices.

\begin{theorem}[Corollary 5.35 of \cite{vershynin2010introduction}]\label{cor:vershynin}
Let $Y \in \R^{d \times N}$ have independent standard Gaussian entries where $N\geq d$.
Then, for every $t > 0$, the following occurs with probability at least $1 - 2 e^{-t^2/2}$:
\begin{align*}
  \sqrt{N} - \sqrt{d} - t \leq \sigma_d(Y) \leq \sigma_1(Y) \leq \sqrt{N} + \sqrt{d} + t,
\end{align*}
where $\sigma_j$ denotes the $j$-th largest singular value.
\end{theorem}

We will also need to bound $\tr\rho = \frac1{nD} \norm x_2^2$.
Because $\norm x_2^2$ is simply a sum of $nD$ many $\chi^2$ random variables, the next proposition follows from standard concentration bounds.

\begin{prop}[Example~2.11 of \cite{W19}]\label{prp:xnorm}
Let $\rv = (\rv_1,\dots,\rv_n)$ consist of independent standard Gaussian random variables in~$\R^D$.
Then, for $0 < t < 1$, the following occurs with probability at least $1 - 2e^{-t^2 nD/8}$:
\begin{align*}
  (1 - t) nD \leq \norm{x}_2^2 \leq (1 + t) nD.
\end{align*}
\end{prop}

Equipped with the above we now prove our gradient bounds in \CREFmain{prop:gradient-bound}{Proposition~2.11}.

\begin{proof}[Proof of \CREFmain{prop:gradient-bound}{Proposition~2.11}]
For any fixed $a\in[k]$, recall that $\rho^{(a)}$ has the same distribution as~$\frac1{nD} YY^T$, where $Y$ is a $d_a\times N_a$-matrix with standard Gaussian entries where $N_{a} = \frac{nD}{d_{a}}$.
By \cref{cor:vershynin}, we have the following bound with failure probability at most~$2 e^{-t^2/2}$:
\begin{align*}
  \sqrt{N_a} - \sqrt{d_a} - t \leq \sigma_d(Y) \leq \sigma_1(Y) \leq \sqrt{N_a} + \sqrt{d_a} + t.
\end{align*}
This event tells us that the eigenvalues of $d_a \rho^{(a)}$ are in the range $( (1 - \frac{\sqrt{d_a} + t}{\sqrt{N_a}})^2, (1 + \frac{\sqrt{d_a} + t}{\sqrt{N_a}})^2)$.
Let $t = \eps \sqrt{n D / d_a} = \eps \sqrt{N_a}$.
Because $n \geq \dmax^{2}/D\eps^{2}$ and $0 < \eps \leq 1$, we have $\sqrt{d_{a}} \leq t \leq \sqrt{N_a}$.
Hence, the eigenvalues of $d_a \rho^{(a)}$ are contained in $( 1 - 4\frac{t}{\sqrt{N_a}}, 1 + 8 \frac{ t}{\sqrt{N_a}})$, and so the eigenvalues of $d_{a} \rho_{a} - I_{d_{a}}$ are bounded in absolute value by~$8\eps$ with failure probability at most~$2e^{-\eps^{2} n D / 2 d_{a}}$.
Moreover, by \cref{prp:xnorm}, we have that $\abs{\tr \rho - 1} \leq \eps$ with failure probability at most $2e^{-\eps^{2} n D / 8}$.
The formulae in \CREFmain{lem:gradient}{Lemma 2.9} and the union bound imply
\begin{align*}
  \norm{\nabla_a f_x}_{\op}
&\leq\frac{1}{\sqrt{d_a}} \norm*{d_{a} \rho^{(a)} - I_{d_a}}_{\op} + \frac{\abs{\tr\rho - 1}}{\sqrt{d_a}}
\leq \frac{8\eps}{\sqrt{d_a}} + \frac{\eps}{\sqrt{d_a}}
\leq \frac{9\eps}{\sqrt{d_a}}
\end{align*}
for all $a\in[k]$, as well as
\begin{align*}
  \abs{\nabla_0 f_x}
= \abs{\tr \rho - 1}
\leq \eps,
\end{align*}
with failure probability at most $2 (k+1) e^{-\eps^{2} n D / 8 \dmax}$.
\end{proof}

Next we prove \CREFmain{lem:hessian}{Lemma 2.12}, which computes the Hessian in terms of partial traces.

\begin{proof}[Proof of \CREFmain{lem:hessian}{Lemma 2.12}]
  Note that the Hessian of~$f_x$ coincides with the one of $\tr\rho\,\Theta$.
  This follows from \CREFmain{eq:obj via rho}{Eq.~(2.2)}, since the Hessian of $\log\det\Theta$ vanishes identically.
  Accordingly, we will compute the Hessian of~$\tr\rho\,\Theta$.
  For $a\in[k]$ and any traceless symmetric $d_a\times d_a$ matrix $H$, we have
  \begin{align*}
    \braket{H, (\nabla^2_{aa} f_x) H}
  = \partial_{s=0} \partial_{t=0} \tr \rho \, e^{(s+t) \sqrt{d_a} H_{(a)}}
  = d_a \tr \rho H_{(a)}^2
  = d_a \tr \rho^{(a)} H^2
  \end{align*}
  using \CREFmain{eq:partial trace duality}{Eq.~(2.3)}.
  Similarly, for $a\neq b\in[k]$, any traceless symmetric $d_a\times d_a$ matrix $H$, and any traceless symmetric $d_b\times d_b$ matrix $K$, we find that
  \begin{align*}
    \braket{H, (\nabla^2_{ab} f_x) K}
  &= \partial_{s=0} \partial_{t=0} \tr \rho \, e^{s \sqrt{d_a} H_{(a)} + t \sqrt{d_b} K_{(b)}} \\
  &= \sqrt{d_a d_b} \tr \rho \, H_{(a)} K_{(b)}
  = \sqrt{d_a d_b} \tr \rho^{(ab)} \left( H \ot K \right)
  \end{align*}
  using \CREFmain{eq:partial trace duality}{Eq.~(2.3)}.
  Next, for $a\in[k]$ and any traceless symmetric $d_a\times d_a$ matrix $H$, we have
  \begin{align*}
    \braket{H, \nabla^2_{a0} f_x}
  = \partial_{s=0} \partial_{t=0} \tr \rho \, e^{s\sqrt{d_a} H_{(a)} + t}
  = \sqrt{d_a} \tr \rho \, H_{(a)}
  = \sqrt{d_a} \tr \rho^{(a)} H.
  \end{align*}
  As we identify $\nabla^2_{a0} f_x$ with a traceless symmetric $d_a\times d_a$ matrix;
  this shows that
  \begin{align*}
    \nabla^2_{a0} f_x = \sqrt{d_a} \left( \rho^{(a)} - \frac{\tr \rho}{d_a} I_{d_a} \right),
  \end{align*}
  and similarly for the transpose.
  Finally,
  \[
    \nabla^2_{00} f_x
  = \partial_{s=0} \partial_{t=0} \tr \rho \, e^{s+t}
  = \tr \rho.
  \]
\end{proof}

We now prove \CREFmain{lem:expansion-convexity}{Lemma~2.15}, which translates quantum expansion into strong convexity.

\begin{proof}[Proof of \CREFmain{lem:expansion-convexity}{Lemma~2.15}]
\noindent
It suffices to verify the hypothesis for $a<b$.
Indeed, since $\tr \Phi^*(I_{d_a}) = \tr \Phi(I_{d_b})$, any $\Phi$ is an $(\eps,\eta)$-quantum expander if and only if this is the case for the adjoint $\Phi^*$, but note that the adjoint of~$\Phi^{(ab)}$ is~$\Phi^{(ba)}$.
To prepare the proof, we also note that
\begin{align}\label{eq:channel to single marginals}
   \Phi^{(ab)}(I_{d_b}) = \rho^{(a)}
\quad\text{and}\quad
   (\Phi^{(ab)})^*(I_{d_a}) = \Phi^{(ba)}(I_{d_a}) = \rho^{(b)},
\end{align}
hence in particular $\tr \Phi^{(ab)}(I_{d_b}) = \tr \rho$.

We wish to bound the operator norm of $M = \frac{\nabla^2 f_\samp}{\tr \rho} - I_\H$, which we consider as a block matrix as in \CREFmain{def:hess grad}{Definition~2.5}.
For this, we use the following basic estimate of the norm of a block matrix in terms of the norm of the matrix of block norms:
\begin{align}\label{eq:baby norm bounds}
  \norm{M}_{\op} \leq \norm{m}_{\op},
\quad \text{ where } m=(\norm{M_{ab}}_{\op})_{a,b\in\{0,1,\dots,k\}}.
\end{align}
We first bound the individual block norms, using that the blocks can be computed using \CREFmain{lem:hessian}{Lemma 2.12}.
Recall that the off-diagonal blocks of the Hessian, $\nabla^2_{ab} f_x$ for $a \neq b\in[k]$, are given by the restriction of $\sqrt{d_a d_b} \Phi^{(ab)}$ to the traceless symmetric matrices.
Since $\Phi^{(ab)}$ is an $(\eps,\eta)$-quantum expander, we have
\begin{align*}
  \norm{M_{ab}}_{\op}
= \frac{\norm{\nabla^2_{ab} f_x}_{\op}}{\tr\rho}
= \frac{\sqrt{d_a d_b}}{\tr \Phi^{(ab)}(I_{d_b})} \norm{\Phi^{(ab)}}_0
\leq \eta,
\end{align*}
using that $\tr \Phi^{(ab)}(I_{d_b}) = \tr \rho$.
The remaining off-diagonal blocks can be bounded as
\begin{align*}
\norm{M_{a0}}
= \frac{\norm{\nabla^2_{a0} f_x}_{\op}}{\tr \rho}
&= \norm*{\sqrt{d_a} \left( \frac{\rho^{(a)}}{\tr \rho} - \frac{I_{d_a}}{d_a} \right)}_F
= \sqrt{d_a} \norm*{\frac{\Phi^{(ab)}(I_{d_b})}{\tr \Phi^{(ab)}(I_{d_b})} - \frac{I_{d_a}}{d_a}}_F \\
&\leq d_a \norm*{\frac{\Phi^{(ab)}(I_{d_b})}{\tr \Phi^{(ab)}(I_{d_b})} - \frac{I_{d_a}}{d_a}}_{\op}
\leq \eps,
\end{align*}
using the fact that the operator norm of a linear functional $\braket{K, -}$ is the same as the Frobenius norm of~$K$, and \cref{eq:channel to single marginals}.
On the other hand, the diagonal blocks for $a\in[k]$ can be bounded by observing that, for any traceless Hermitian $H$,
\begin{align*}
  \abs{\braket{H, M_{aa} H}}
&= \abs*{\braket{H, \left( \frac{\nabla^2_{aa} f_x}{\tr \rho} - I \right) H}}
= d_a \abs*{\tr \left( \frac{\rho^{(a)}}{\tr \rho} - \frac{I_{d_a}}{d_a} \right) H^2} \\
&\leq d_a \norm*{\frac{\rho^{(a)}}{\tr \rho} - \frac{I_{d_a}}{d_a}}_{\op} \norm{H}_F^2
\leq \eps \norm H_F^2,
\end{align*}
hence $\norm{M_{aa}}_{\op} \leq \eps$, while $\abs{M_{00}} = \abs{\frac{\nabla^2_{00} f_x}{\tr \rho} - 1} = 0$.
To conclude the proof, decompose
\begin{align*}
  m
= \left[\begin{array}{c|cccc}
  0 & 0 & 0 & \cdots & 0 \\
  \hline
  0 & 0 & m_{12} & \cdots & m_{1k} \\
  0 & m_{21} & 0 & & m_{2k} \\
  \vdots & \vdots & & \ddots & \vdots \\
  0 & m_{k1} & m_{k2} & \cdots & 0
  \end{array}\right]
+ \left[\begin{array}{c|cccc}
  0 & 0 & 0 & \cdots & 0 \\
  \hline
  0 & m_{11} & 0 & \cdots & 0 \\
  0 & 0 & m_{22} & & 0 \\
  \vdots & \vdots & & \ddots & \vdots \\
  0 & 0 & 0 & \cdots & m_{kk}
  \end{array}\right]
+ \left[\begin{array}{c|cccc}
  0 & m_{01} & m_{02} & \cdots & m_{0k} \\
  \hline
  m_{10} & 0 & 0 & \cdots & 0 \\
  m_{20} & 0 & 0 & & 0 \\
  \vdots & \vdots & & \ddots & \vdots \\
  m_{k0} & 0 & 0 & \cdots & 0
  \end{array}\right].
\end{align*}
The nonzero entries of the first matrix are bounded by $\eta$, hence its operator norm is at most $(k-1)\eta$.
The second matrix is diagonal with diagonal entries bounded by $\eps$, hence its operator norm is at most~$\eps$.
The third matrix has nonzero entries bounded by $\eps$, hence its operator norm is bounded by~$\sqrt k \eps$.
Using \cref{eq:baby norm bounds} we obtain the desired bound.
\end{proof}

We will now use \CREFmain{thm:hess-pisier}{Theorem~2.16}, which shows that random completely positive maps are good expanders, to establish strong convexity at the identity.

\begin{proof}[Proof of \CREFmain{thm:tensor-convexity}{Proposition~2.17}]
By \CREFmain{lem:expansion-convexity}{Lemma~2.15}, it is enough to prove that with the desired probability all $\Phi^{(ab)}$ are $(\eps,\eta):=(\frac1{40 k^{1/2}},\frac1{20k})$-quantum expanders for $a\neq b\in[k]$ and $\tr \rho \in (\frac78,\frac98)$.
If that is the case, then
\begin{align*}
  \norm*{\nabla^2 f_x - I_\H}_{\op}
&\leq \tr \rho \cdot  \norm*{\frac{\nabla^2 f_x}{\tr \rho} - I_\H}_{\op}  + \abs{1 - \tr\rho} \\
&\leq \left( (k-1)\eta + (\sqrt k + 1) \eps \right) \tr\rho + \abs{1 - \tr\rho}
\leq \frac14.
\end{align*}
Firstly, $\tr \rho = \frac{1}{nD} \|X\|^2$ is in $(\frac78, \frac98)$ with failure probability $e^{-\Omega(nD)}$ by \cref{prp:xnorm}.

Next, we describe an event that implies the $\Phi^{(ab)}$ are all $\eps$-doubly balanced for $\eps=\frac1{40k^{1/2}}$.
By \CREFmain{eq:balanced via grad}{Eq.~(2.11)}, this is equivalent to the condition $\sqrt{d_a} \norm{\nabla_a f_{\rv}}_{\op} \leq \eps \tr \rho$ for all $a \in [k]$.
By \CREFmain{prop:gradient-bound}{Proposition~2.11}, and assuming the bound $\tr \rho \geq \frac78$ from above, the latter occurs with failure probability~$k \smash{e^{-\Omega(\frac{nD}{k \dmax})}}$ provided $n \geq C k \smash{\frac{\dmax^2}D}$ for a universal constant~$C>0$.

Finally, we describe an event that ensures that $\norm{\Phi^{(ab)}}_0 \leq \eta \smash{\frac{\tr\rho}{\sqrt{d_a d_b}}}$ for $\eta=\frac1{20k}$ for any fixed~$a \neq b$, which is the other condition needed for quantum expansion.
Recall that each~$\Phi^{(ab)}$ is distributed as $\frac1{nD} \Phi_A$, where $A$ is a tuple of $\frac{nD}{d_ad_b}$ many $d_a \times d_b$ matrices with independent standard Gaussian entries.
Thus, taking $t^2 = O(\smash{\frac{\eta \sqrt{nD}}{d_a + d_b}})$ and again assuming that $\tr\rho \geq \frac78$, we have $\norm{\Phi^{(ab)}}_0 \leq \eta \frac{\tr\rho}{\sqrt{d_a d_b}}$ by \CREFmain{thm:hess-pisier}{Theorem~2.16},
with failure probability at most~$\smash{( \frac{\sqrt{nD}}{k \dmax} )^{-\Omega(\dmin)}}$.

By the union bound, we conclude that all $\Phi^{(ab)}$ for $a\neq b$ are $(\eps,\eta)$-quantum expanders and that $\tr\rho \in (\frac78,\frac98)$, up to a failure probability of at most
\begin{align*}
  e^{-\Omega(nD)}
+ k \smash{e^{-\Omega\bigl(\frac{nD}{k \dmax}\bigr)}}
+ k^2 \left( \frac{\sqrt{nD}}{k \dmax} \right)^{-\Omega(\dmin)}.
\end{align*}
The final term dominates, which implies the desired failure probability.
To see that the final term dominates compare exponents: it suffices to show that $nD/ k \dmax \geq \dmin \log (\frac{\sqrt{nD}}{k \dmax})$ by our assumption on $n$, which states that $\alpha:= nD/k \dmax^2 \geq C$.
Writing the desired inequality in terms of $\alpha$, we need $\dmax \alpha \geq \dmin \log(\sqrt{\alpha/k})$.
This holds for $C$ large enough.
\end{proof}

Next we wish to show that strong convexity at the identity implies strong convexity nearby, as formulated in the following lemma:

\begin{lemma}[Robustness of strong convexity]\label{lem:convexRobustness}
There is a universal constant $0 < \eps_0 < 1$ such that if $\norm{\nabla_a f_x(I_D)}_{\op} \leq \eps_0/\sqrt{d_a}$ for all $a\in[k]$ and $\abs{\nabla_{0} f_{\samp}(I_{D})} \leq \eps_0$, then
$$\|\nabla^{2} f_{\samp}(\Theta) - \nabla^{2} f_{\samp}(I_D)\|_{\op} = O(\delta)$$
for any $\Theta\in\P$ such that $\delta := \dop(\Theta, I_D) \leq \eps_0$.
In particular, for any $\lambda > 0$, if~$f_x$ is $\lambda$-strongly convex at $I_D$ then $f_x$ is $(\lambda-O(\delta))$-strongly convex at $\Theta$.
\end{lemma}

The proof of this result requires some preparation.
First note that by \CREFmain{remark:hessian-everywhere}{Remark~2.13}, we have $\nabla^2 f_\samp(\Theta) = \nabla^2 f_{\samp'}$ where $\samp' = \Theta^{1/2} \samp$.
Thus we need only bound the difference between $f_\samp$ and $f_{\samp'}$ for $\|\log \Theta\|_{\op}$ small, $\Theta \in \P$.
For a matrix $\delta_{a}$ in $\Mat(d_a)$, we use
$e^{(\delta_{a})_{(a)}}$ to denote
$$ e^{(\delta_{a})_{(a)}} = I_{d_1} \otimes \cdots \otimes I_{d_{a-1}} \otimes e^{\delta} \otimes I_{d_{a+1}} \otimes \cdots \otimes I_{d_k},$$
as in \CREFmain{definition:partial-trace}{Definition~2.8}.
We will write $\Theta^{1/2}$ as $e^{\delta}$, where $\delta = \sum_{a=1}^k (\delta_{a})_{(a)}$.
We now have
$\Theta^{1/2} = e^{\delta} = \otimes_{a = 1}^k e^{\delta_a}$, and $\frac{1}{2}\|\log \Theta\|_{\op} = \|\delta\|_{\op} = \sum_{a =1}^k \|\delta_{a}\|_{\op}$.
To bound the difference between $\nabla^2 f_{x'}$ and $\nabla^2 f_x$, we will show each component of the Hessian $\nabla^2 f_{x'}$ (as presented in \CREFmain{lem:hessian}{Lemma 2.12}) only changes (from $\nabla^2 f_x$) by a small amount under the perturbation $x \to x' := e^\delta x$.
In particular we will give bounds on each block under each component-wise perturbation $x \to e^{(\delta_{a})_{(a)}} x$, and write the overall perturbation as a sequence of such component-wise perturbations. For convenience, we adopt the short-hand
$$ \rho_x:= \frac{1}{nD} x x^T.$$

We begin with an easy fact relating the exponential map and matrix norms.

\begin{fact} \label{f:expTaylor} For all symmetric $d\times d$ matrices $A$ such that $ \|A\|_{\op} \leq 1$, we have
$$ \|e^{A} - I\|_{\op} \leq 2 \|A\|_{\op}
\ \ \ \text{ and } \ \ \
\|e^{A} - I\|_{F} \leq 2 \|A\|_{F}. $$
\end{fact}

\noindent The $00$ component of the Hessian is a scalar $\nabla^{2}_{00} f = \tr[\rho]$, and for $a \geq 1$ we think of each $0a$ component as a vector:
\[ \sum_{a} \langle z_{0}, (\nabla^{2}_{0a} f) Z_{a} \rangle = z_{0} \langle \rho, \sum_{a} \sqrt{d_{a}} Z_{(a)} \rangle       \]
The diagonal components involve only one-body marginals of $\rho$:
\[ \langle Z_{a}, (\nabla^{2}_{aa} f) Z_{a} \rangle = \langle d_{a} \rho^{(a)}, Z_{a}^{2} \rangle       \]
And the off-diagonal components involve two-body marginals:
\[ \langle Z_{a}, (\nabla^{2}_{ba} f) Z_{b} \rangle =  \langle \sqrt{d_{a} d_{b}} \rho^{(ab)}, Z_{a} \otimes Z_{b} \rangle.   \]

In \cref{atoaaRobustness} and \cref{btoaaRobustness}, we will prove perturbation bounds on one-body marginals, and in \cref{btoabRobustness} we will prove perturbation bounds on two-body marginals.
This will allow us to bound the change in the $0a$ components, diagonal components, and the off-diagonal components, respectively.
Following the structure of the proof of \CREFmain{thm:tensor-convexity}{Proposition~2.17}, we will collect all the term-wise bounds to prove an overall bound at the end of the section.

\begin{lemma} \label{atoaaRobustness}
For $\samp \in \R^{D \times n}$ and a symmetric matrix $\delta \in \Mat(d_{a})$ such that $\|\delta\|_{\op} \leq 1$, if we denote $\samp' := e^{\delta_{(a)}} \samp$ then
\[ \|\rho_{\samp'}^{(a)} - \rho_{\samp}^{(a)}\|_{\op} \leq 8 \|\delta\|_{\op} \|\rho_{\samp}^{(a)}\|_{\op}   . \]
\end{lemma}
\begin{proof}By definition, $\|\rho_{\samp'}^{(a)} - \rho_{\samp}^{(a)}\|_{\op} = \sup_{\|Z\|_{1} \leq 1} \langle Z_{(a)}, \rho_{\samp'} - \rho_{\samp} \rangle $.
Let $\delta' := e^{\delta} - I_{a}$. Note that $\|\delta'\|_{\op} \leq 2 \|\delta\|_{\op}$ by \cref{f:expTaylor} and our assumption $\|\delta\|_{\op} \leq 1$. Now
\begin{align*} \langle Z_{(a)}, \rho_{\samp'} - \rho_{\samp} \rangle &= \langle Z_{(a)}, (I+\delta')_a \rho_{\samp} (I+\delta')_a - \rho_{\samp} \rangle \\
& = \langle Z, \delta' \rho_\samp^{(a)} \rangle + \langle Z, \rho_\samp^{(a)} \delta' \rangle + \langle Z, \delta' \rho_\samp^{(a)} \delta' \rangle \\
& \leq (2\|\delta'\|_{\op} + \|\delta'\|_{\op}^{2}) \|\rho^{(a)}\|_{\op}\|Z\|_1  \leq 8 \|\delta\|_{\op} \|\rho^{(a)}\|_{\op}.
\end{align*}
\end{proof}

\begin{lemma} \label{btoaaRobustness}
For $\samp \in \R^{D \times n}$ and symmetric matrix $\delta \in \Mat(d_{b})$ such that $\|\delta\|_{\op} \leq 1$, if we denote $\samp' := e^{\delta_{(b)}} \samp$ then for $b \neq a$:
\[ \|\rho_{\samp'}^{(a)} - \rho_{\samp}^{(a)}\|_{\op} \leq 2 \|\delta\|_{\op} \|\rho_{\samp}^{(a)}\|_{\op} .    \]
\end{lemma}
\begin{proof}
By definition, $\displaystyle \|\rho_{\samp'}^{(a)} - \rho_{\samp}^{(a)}\|_{\op} = \sup_{\|Z\|_{1} \leq 1, Z \succeq 0} |\langle Z_{(a)}, \rho_{\samp'} - \rho_{\samp} \rangle|$.
Let $\delta' := e^{\delta} - I_b$.
Note that $\|\delta'\|_{\op} \leq 2 \|\delta\|_{\op}$ by \cref{f:expTaylor} and our assumption $\|\delta\|_{\op} \leq 1$.
Now
\begin{align*}
|\langle Z_{(a)}, \rho_{\samp'} - \rho_{\samp} \rangle|
& = | \langle Z_{(a)}, e^{\delta_{(b)}}\rho_{\samp} e^{\delta_{(b)}} - \rho_{\samp} \rangle|\\
& = | \langle Z_{(a)} \delta'_{(b)}, \rho_{\samp} \rangle   |
= | \langle Z \otimes \delta', \rho_{\samp}^{(ab)} \rangle   | \\
&\leq \langle \|\delta'\|_{\op} Z \otimes  I_b , \rho_{\samp}^{(ab)} \rangle\\
&= \|\delta'\|_{\op} \langle Z, \rho_{\samp}^{(a)} \rangle \leq 2\|\delta\|_{\op} \|Z\|_1 \|\rho_{\samp}^{(a)}\|_{\op}.
\end{align*}
\end{proof}

This is already enough to prove a bound on $0a$ and $aa$ terms:

\begin{corollary} \label{diagRobustness}
Let $\samp \in \R^{D \times n}$ be such that $\|d_{a} \rho_{\samp}^{(a)}\|_{\op} \leq 1 + \frac{1}{20}$, and for $b \in [k]$ let $\delta_b \in \Mat(d_b)$ be a  symmetric matrix such that $\sum_{b} \|\delta_{b}\|_{\op} \leq \frac{1}{8}$. Denoting $\delta_{(b)} := (\delta_b)_{(b)}$,  $\delta = \sum_b \delta_{(b)}$ and $x' = e^{\delta} \samp$, for $a \geq 1$ we have
\[ \|\nabla^{2}_{aa} f(e^{2\delta}) - \nabla^{2}_{aa} f(I)\|_{\op} \leq 25 \|\delta\|_{\op} .  \]
\end{corollary}
\begin{proof}
Recall from \CREFmain{lem:hessian}{Lemma 2.12} that $\langle Y, (\nabla^{2}_{aa} f_{\samp}) Y \rangle = \langle d_{a} \rho_{\samp}^{(a)}, Y^{2} \rangle$; thus it is enough to show that $\|\rho_{\samp'}^{(a)} - \rho_{\samp}^{(a)}\|_{\op} \leq 25 \|\delta\|_{\op} /d_a$. We treat the perturbation $e^\delta$ as the composition of $k$ perturbations:
\[ \samp_{(0)}:=\samp \to \samp_{(1)}:= e^{\delta_{(1)}} \samp_{(0)} \to \dots \to \samp_{(k)}:=e^{\delta_{(k)}} \samp_{(k-1)} = \samp'. \]
We can use \cref{atoaaRobustness} to handle $e^{\delta_{(a)}}$ and \cref{btoaaRobustness} for the rest. Let $Z$ be a symmetric matrix.
\begin{align*}
 |\langle \rho_{\samp'}^{(a)} - \rho_{\samp}^{(a)}, Z \rangle|
 &\leq \sum_{j=1}^{k} |\langle \rho_{\samp_{(j)}}^{(a)} - \rho_{\samp_{(j-1)}}^{(a)}, Z \rangle| \\
 &\leq \sum_{j=1}^{k}  8 \|\delta_{j}\|_{\op} \|\rho_{\samp_{(j-1)}}^{(a)}\|_{\op} \|Z\|_{1}.
\end{align*}
Where the last inequality is due to \cref{atoaaRobustness,btoaaRobustness}.
To bound each term in the right-hand side, note that by \cref{atoaaRobustness,btoaaRobustness} we have
$$\|\rho_{\samp_{(j)}}^{(a)}\|_{\op} \leq \|\rho_{\samp_{(j)}}^{(a)} - \rho_{\samp_{(j-1)}}^{(a)}\|_{\op} + \|\rho_{\samp_{(j-1)}}^{(a)}\|_{\op} \leq   (1+8 \|\delta_{j}\|_{\op})\|\rho_{\samp_{(j-1)}}^{(a)}\|_{\op}$$
and hence by induction the $j^{th}$ term in the sum is at most $$8 \|\delta_j\|_{\op} \left( \prod_{l=1}^k (1+8 \|\delta_{l}\|_{\op}) \right) \|\rho_{\samp}^{(a)}\|_{\op} \|Z\|_{1}.$$ By our assumption that $\sum_l \|\delta_l\|_{\op} \leq 1/8$, this is at most $8 \|\delta_j\|_{\op} e^{8 \sum \|\delta_l\|_{\op}} \|\rho_{\samp}^{(a)}\|_{\op} \|Z\|_{1} \leq 8e \|\delta_j\|_{\op} \|\rho_{\samp}^{(a)}\|_{\op} \|Z\|_{1}. $ Adding up the terms and using that $\|\delta\|_{\op} = \sum \|\delta_{(c)}\|_{\op}$, the overall sum is then at most $8 e \|\delta\|_{\op} \|\rho_{\samp}^{(a)}\|_{\op} \|Z\|_{1}$. Using our assumption on $\|d_{a} \rho_{\samp}^{(a)}\|_{\op}$ completes the proof.
\end{proof}

\begin{corollary} \label{constantRobustness}
Let $\samp \in \R^{D \times n}$ be such that $\|d_{a} \rho_{\samp}^{(a)}\|_{\op} \leq 1 + \frac{1}{20}$, and for $b \in [k]$ let $\delta_b$ be symmetric matrices such that $\|\sum_{b} \delta_{(b)}\|_{\op} = \sum_{b} \|\delta_{b}\|_{\op} \leq \frac{1}{8}$, where once again we denote $\delta_{(b)} := (\delta_b)_{(b)}$ and $\delta := \sum_b \delta_{(b)}$.
Denoting $\samp' := e^{\delta} \samp$, for $a \geq 1$ we have
\begin{align*} |\nabla^{2}_{00} f_{\samp'} - \nabla^{2}_{00} f_{\samp}| &\leq 5 \|\delta\|_{\op}  \\
\text{ and } \|\nabla^{2}_{0a} f_{\samp'} - \nabla^{2}_{0a} f_{\samp}\|_{\op} &\leq 25 \|\delta\|_{\op} .  \end{align*}
\end{corollary}
\begin{proof}
Recall from \CREFmain{lem:hessian}{Lemma 2.12} that the $00$ component of the Hessian is just the scalar $\tr \rho$. The assumption that $\|d_{a} \rho_{\samp}^{(a)}\|_{\op} \leq 1 + \frac{1}{20}$ implies $\tr[\rho_{\samp}] = \tr \rho_{\samp}^{(a)} \leq 1 + 1/20$. Now we can use the approximation for $e^{\delta}$ in \cref{f:expTaylor}:
\[ |\tr[\rho_{\samp'} - \rho_{\samp}]| = |\langle \rho_{\samp}, e^{2\delta} - I \rangle| \leq \tr[\rho_{\samp}] \|e^{2 \delta} - I\|_{\op} \leq 5 \|\delta\|_{\op}     \]
In the last step we used our bound on $\tr[\rho_{\samp}].$
The $0a$ component is a vector, so it is enough to bound the inner product with any traceless matrix $Z$ of unit Frobenius norm:
\[ |\langle \rho_{\samp'}^{(a)} - \rho_{\samp}^{(a)}, \sqrt{d_{a}} Z \rangle| \leq \|\rho_{\samp'}^{(a)} - \rho_{\samp}^{(a)}\|_{\op} \sqrt{d_{a}} \|Z\|_{1}. \]
In the proof of \cref{diagRobustness} we showed under the same assumptions we have $\|\rho_{\samp'}^{(a)} - \rho_{\samp}^{(a)}\|_{\op} \leq 25 \|\delta\|_{\op}/d_a$, from which it follows that the above is at most $25 \|\delta\|_{\op} \|Z\|_{F}.$ \end{proof}

The off-diagonal components require the following two lemmata on bipartite marginals:

\begin{lemma} \label{btoabRobustness}
For $\samp \in \R^{D \times n}$ and a symmetric matrix $\delta \in \Mat(d_{c})$ such that $\|\delta\|_{\op} \leq \frac{1}{8}$; if we denote $\samp' := e^{\delta_{(c)}} \samp$, then for $c \in \{a,b\}$ we have
\[ \sup_{Y \in \smallSym_{d_{a}}^{0}, Z \in \smallSym_{d_{b}}^{0}} \frac{|\langle \rho_{\samp'}^{(ab)} - \rho_{\samp}^{(ab)}, Y \otimes Z \rangle|}{\|Y\|_{F} \|Z\|_{F}} \leq 3 \|\delta\|_{\op} \sup_{Y \in \smallSym_{d_{a}}, Z \in \smallSym_{d_{b}}} \frac{\langle \rho_{\samp}^{(ab)}, Y \otimes Z \rangle}{\|Y\|_{F} \|Z\|_{F}}.        \]
Note that $\smallSym_{d}^{0}$ are traceless symmetric matrices, whereas $\smallSym_{d}$ are symmetric matrices.
\end{lemma}
\begin{proof}
By taking adjoints, we can assume w.l.o.g. that $c = b$. Let $R : \Mat(d_{b}) \to \Mat(d_{b})$ be defined as $R(Z) :=  e^{\delta} Z e^{\delta}$. Then
\[ |\langle \rho_{\samp'}^{(ab)} - \rho_{\samp}^{(ab)}, Y \otimes Z \rangle| = |\langle \rho_{\samp}^{(ab)}, Y \otimes (R(Z) - Z) \rangle|  \]
The subspace $\smallSym_{d_{b}}^{0}$ is not invariant under $R$, but we show $R \approx I$. Let $\delta' :=  e^{\delta} - I$; by \cref{f:expTaylor}, $\|\delta'\|_{\op} \leq \frac{1}{4}$. Now
\[ \|R(Z) - Z\|_{F} \leq 2 \|\delta' Z\|_{F} + \|\delta' Z \delta'\|_{F} \leq (2 \|\delta'\|_{\op} + \|\delta'\|_{\op}^{2}) \|Z\|_{F} \leq 3\|\delta\|_{\op}\|Z\|_{F}.    \]
We combine these inequalities and apply a change of variables $R(Z) - Z \leftarrow Z'$ to finish the proof.
\begin{align*} \sup_{Y \in \smallSym_{d_{a}}^{0}, Z \in \smallSym_{d_{b}}^{0}} \frac{|\langle \rho_{\samp'}^{(ab)} - \rho_{\samp}^{(ab)}, Y \otimes Z \rangle|}{\|Y\|_{F} \|Z\|_{F}}
& = \sup_{Y \in \smallSym_{d_{a}}^{0}, Z \in \smallSym_{d_{b}}^{0}}\frac{|\langle \rho_{\samp}^{(ab)}, Y \otimes (R(Z) - Z) \rangle|}{\|Y\|_F\|Z\|_F} \\
&\leq \sup_{Y \in \smallSym_{d_{a}}^{0}, Z' \in \smallSym_{d_{b}}} \frac{|\langle \rho_{\samp}^{(ab)}, Y \otimes Z' \rangle| \cdot 3 \|\delta\|_{\op}}{\|Y\|_F\|Z'\|_F}.
\end{align*}
\end{proof}

\begin{lemma} \label{ctoabRobustness}
For $\samp \in \R^{D \times n}$ and a symmetric matrix $\delta \in \Mat(d_{c})$ such that $\|\delta\|_{\op} \leq \frac{1}{8}$; if we denote $\samp' := e^{\delta_{(c)}}  \samp$, then for $c \not\in \{a,b\}$ we have
\begin{align*}
\sup_{Y \in \smallSym_{d_{a}}^{0}, Z \in \smallSym_{d_{b}}^{0}} \frac{|\langle \rho_{\samp'}^{(ab)} - \rho_{\samp}^{(ab)}, Y \otimes Z \rangle|}{\|Y\|_{F} \|Z\|_{F}} \leq 4 \|\delta\|_{\op} \sup_{Y \in \smallSym_{d_{a}}, Z \in \smallSym_{d_{b}}} \frac{\langle \rho_{\samp}^{(ab)}, Y \otimes Z \rangle}{\|Y\|_{F} \|Z\|_{F}}.
\end{align*}
\end{lemma}
\begin{proof}
Let $\delta' := e^{2 \delta} - I_{c}$ so that $|\langle \rho_{\samp'}^{(ab)} - \rho_{\samp}^{(ab)}, Y \otimes Z \rangle| =  |\langle \rho_{\samp}^{(abc)}, Y \otimes Z \otimes \delta' \rangle|$.
We first assume $Y,Z \succeq 0, $ and without loss of generality we assume that $\|Y\|_{F} = \|Z\|_{F} = 1$. Because $\rho_{\samp}^{(abc)}, Y, Z \succeq 0,$ and $\delta' \preceq \norm{\delta'}_{\op} \cdot I_c$, we have
\begin{align*}
|\langle \rho_{\samp}^{(abc)}, Y \otimes Z \otimes \delta' \rangle|
& \leq \langle \rho_{\samp}^{(abc)}, Y \otimes Z \otimes \norm{\delta'}_{\op} \cdot I_c \rangle\\
& \leq \|\delta'\|_{\op} \langle \rho_{\samp}^{(ab)}, Y \otimes Z \rangle \leq 2 \|\delta\|_{\op} \langle \rho_{\samp}^{(ab)}, Y \otimes Z \rangle, \end{align*}
where the last inequality is by \cref{f:expTaylor}.
To finish the proof we decompose $Y = Y_{+} - Y_{-}, Z = Z_{+} - Z_{-}$, where $Y_+, Y_-, Z_+, Z_-$ are all positive semidefinite, and bound
\begin{align*} |\langle \rho_{\samp'}^{(ab)} - \rho_{\samp}^{(ab)}, Y \otimes Z \rangle|
& \leq \sum_{s,t \in \{+,-\}} |\langle \rho_{\samp'}^{(ab)} - \rho_{\samp}^{(ab)}, Y \otimes Z \rangle| \\
& \leq \sum_{s,t \in \{+,-\}} 2\|\delta\|_{\op} \langle \rho_{\samp}^{(ab)}, Y_s \otimes Z_t \rangle\\
&\leq 2\left( \sup_{Y \in \smallSym_{d_{a}}, Z \in \smallSym_{d_{b}}} \frac{\langle \rho_{\samp}^{(ab)}, Y \otimes Z \rangle}{\|Y\|_{F} \|Z\|_{F}} \right) \|\delta\|_{\op} \sum_{s,t \in \{+,-\}} \|Y_{s}\|_{F} \|Z_{t}\|_{F}
\end{align*}
The Cauchy Schwarz inequality allows us to bound the summation:
\[\sum_{s,t \in \{+,-\}} \|Y_{s}\|_{F} \|Z_{t}\|_{F}  \leq (2\|Y_{+}\|_{F}^{2} + 2\|Y_{-}\|_{F}^{2})^{1/2} (2\|Z_{+}\|_{F}^{2} + 2\|Z_{-}\|_{F}^{2})^{1/2} = 2 \|Y\|_{F} \|Z\|_{F} .   \]
Plugging this bound in to the supremum on the left-hand side in the statement of the lemma completes the proof.
\end{proof}

The following lemma, from \cite{KLR19}, will be helpful.

\begin{lemma} \label{inftyto2} For $x \in \R^{D \times n}$,
$$\|\nabla^{2}_{ab} f_{\samp}\|_{F \to F}^{2} \leq \|d_{a} \rho_{\samp}^{(a)}\|_{\op} \|d_{b} \rho_{\samp}^{(b)}\|_{\op}.$$
\end{lemma}

Analogously to the proof of \cref{diagRobustness}, we can now combine \cref{btoabRobustness} and \cref{ctoabRobustness} to bound the effect of a perturbation with more than one nontrivial tensor factor.
To state the result, we recall the definition of the seminorm $\norm{\cdot}_0$ of a linear map $M \colon \Mat(d_b) \to \Mat(d_a)$ from \CREFmain{eq:expansion}{Eq.~(2.10)} in \CREFmain{def:expansion}{Definition~2.14},
\begin{align*}
  \norm M_0 :=
  \max_{\substack{K \in \Mat(d_{a}) \\ \text{traceless symmetric}}}
  \;
  \max_{\substack{H \in \Mat(d_{b}) \\ \text{traceless symmetric}}}
  \frac{\langle K, M(H) \rangle}{\norm K_F \norm H_F},
\end{align*}
which will be helpful for translating the above lemmas into statements about the Hessian.

\begin{corollary} \label{offdiagRobustness}
Let $\samp \in \R^{D \times n}$ be such that $\|d_{a} \rho_{\samp}^{(a)}\|_{\op}, \|d_{b} \rho_{\samp}^{(b)}\|_{\op} \leq 1+\frac{1}{20}$, and for $c \in [k]$ let $\delta_c$ be a symmetric matrix such that $\|\sum_{c} \delta_{(c)}\|_{\op} = \sum_{c} \|\delta_{c}\|_{\op} \leq \frac{1}{8}$. Denoting $\samp' := e^{\delta} \samp$, we have
\[ \|\nabla^{2}_{ab} f_{\samp'} - \nabla^{2}_{ab} f_{\samp}\|_{0} \leq 21 \|\delta\|_{\op}  \]
\end{corollary}
\begin{proof}
First, using \CREFmain{lem:hessian}{Lemma 2.12}, we write the left-hand and right-hand sides of the inequalities in \cref{btoabRobustness} and \cref{ctoabRobustness} in terms of the Hessian:
\begin{align*}
\sup_{Y \in \smallSym_{d_{a}}^{0}, Z \in \smallSym_{d_{b}}^{0}} \frac{\langle \rho_{\samp'}^{(ab)} - \rho_{\samp}^{(ab)}, Y \otimes Z \rangle}{\|Y\|_{F} \|Z\|_{F}} &= \frac{\|\nabla^{2}_{ab} f_{\samp'} - \nabla^{2}_{ab} f_{\samp}\|_{0}}{\sqrt{d_{a} d_{b}} }, \\
\text{ and }\sup_{Y \in \smallSym_{d_{a}}, Z \in \smallSym_{d_{b}}} \frac{\langle \rho_{\samp}^{(ab)}, Y \otimes Z \rangle}{\|Y\|_{F} \|Z\|_{F}} &= \frac{\|\nabla^{2}_{ab} f_{\samp}\|_{F \to F}}{\sqrt{d_{a} d_{b}} }   .
\end{align*}
Using the same iterative strategy as in the proof of \cref{diagRobustness} for the left-hand sides of the above identities, we have
\[ |\langle Y, (\nabla^{2}_{ab} f_{\samp'} - \nabla^{2}_{ab} f_{\samp}) Z \rangle| \leq 20 \|\delta\|_{\op} \|\nabla^{2}_{ab} f_{\samp}\|_{F \to F} \|Y\|_{F} \|Z\|_{F} ,   \]
using \cref{btoabRobustness} for $a$ and $b$ and \cref{ctoabRobustness} for the rest. Finally, we may rewrite \cref{inftyto2} using \CREFmain{lem:hessian}{Lemma 2.12} to find $\|\nabla^{2}_{ab} f_{\samp}\|_{F \to F}^{2} \leq \|d_{a} \rho_{\samp}^{(a)}\|_{\op} \|d_{a} \rho_{\samp}^{(a)}\|_{\op}$. Using our assumption that $\|d_{a} \rho_{a}\|_{\op}, \|d_{b} \rho_{b}\|_{\op} \leq 1 + \frac{1}{20}$ completes the proof.
\end{proof}

We can finally prove \cref{lem:convexRobustness} by combining the above term-by-term bounds.

\begin{proof}[Proof of \cref{lem:convexRobustness}]
The above \cref{constantRobustness,diagRobustness,offdiagRobustness} require $\|d_{a} \rho^{(a)}\|_{\op} \leq 1 + \frac{1}{20}$ , which are implied by our assumption on the gradient:
\[ \|d_{a} \rho^{(a)}\|_{\op} \leq 1 + |\tr \rho - 1| + \|d_{a} \rho^{(a)} - (\tr \rho) I_{d_{a}} \|_{\op}    \]
\[ = 1 + |\nabla_{0} f| + \|\sqrt{d_{a}} \nabla_{a} f\|_{\op} \leq 1 + 2 \eps_0 ,    \]
 so choosing $\eps_0 \leq \frac{1}{40}$ suffices.
Recall the expression of the Hessian as a quadratic form evaluated on $Z = (z_0, Z_1, \dots, Z_k):$
\begin{align*} \langle Z, (\nabla^{2} f) Z \rangle&=\\
  z_{0} &(\nabla^{2}_{00} f) z_{0} + 2 \sum_{a} \langle z_{0}, (\nabla^{2}_{0a} f) Z_{a} \rangle + \sum_{a} \langle Z_{a}, (\nabla^{2}_{aa} f) Z_{a} \rangle + \sum_{a \neq b} \langle Z_{a}, (\nabla^{2}_{ab} f) Z_{b} \rangle .  \end{align*}
Let $\samp' := e^{\delta} \samp$. Then by \cref{constantRobustness} we have a bound on the $0a$ terms:
\[ | z_{0}^{2} (\nabla^{2}_{00} f_{\samp'} - \nabla^{2}_{00} f_{\samp} ) + 2 \sum_{a} \langle z_{0}, (\nabla^{2}_{0a} f_{\samp'} - \nabla^{2}_{0a} f_{\samp}) Z_{a} \rangle |      \]
\[ \leq 5 \|\delta\|_{\op} z_{0}^{2} + (2 |z_{0}|) 25 \|\delta\|_{\op} \sum_{a} \|Z_{a}\|_{F}
\leq \|\delta\|_{\op} (17 k z_{0}^{2} + 25 \sum_{a} \|Z_{a}\|_{F}^{2})   \]
In the last step we used Young's inequality ($2pq \leq p^{2} + q^{2}$) for each term with $p = z_0$, $q = \|Z_{a}\|_F$.

By \cref{diagRobustness} we have a bound on the diagonal terms, and by \cref{offdiagRobustness} we have a bound on the off-diagonal terms:
\[ |\sum_{ab} \langle Z_{a}, (\nabla^{2}_{ab} f_{\samp'} - \nabla^{2}_{ab} f_{\samp} ) Z_{b} \rangle | \leq \|\delta\|_{\op} \left( 25 \sum_{a} \|Z_{a}\|_{F}^{2} + 21 \sum_{a \neq b} \|Z_{a}\|_{F} \|Z_{b}\|_{F} \right)   \]
\[ \leq (25 + 21(k-1)) \|\delta\|_{\op} \left( \sum_{a} \|Z_{a}\|_{F}^{2} \right)   \]
So combining all three terms we see:
\[ |\langle Z, (\nabla^{2} f_{\samp'} - \nabla^{2} f_{\samp} ) Z \rangle | \leq \|\delta\|_{\op} \left( 17 k z_{0}^{2} + (25 + 25 + 21 (k-1)) \sum_{a} \|Z_{a}\|_{F}^{2} \right)    \]
\[ \leq 50 k \|\delta\|_{\op} \left( z_{0}^{2} + \sum_{a} \|Z_{a}\|_{F}^{2} \right)  = 50 k \|\delta\|_{\op} \|Z\|^2.  \]
Note that this also gives an upper bound for $\|\nabla^{2} f_{\samp'}\|_{\op}$.
\end{proof}

With \cref{lem:convexRobustness} in hand, we can establish strong convexity near the identity.

\begin{proof}[Proof of \CREFmain{thm:ball-convexity}{Proposition~2.18}]
We can choose $C>0$ such that both \CREFmain{prop:gradient-bound,thm:tensor-convexity}{Propositions~2.11 and~2.17} apply (the former with $\eps\leq\eps_0/9\; $, where $\eps_0$ is the universal constant from \cref{lem:convexRobustness}).
Then the assumptions of \cref{lem:convexRobustness} are satisfied for $\lambda=\frac34$ with failure probability at most
\begin{align*}
  2(k+1)e^{-\eps^2 \frac{nD}{8\dmax}} + k^2 \left( \frac{\sqrt{nD}}{k \dmax} \right)^{-\Omega(\dmin)},
\end{align*}
where the latter term dominates, and there exists a constant $0<c\leq\eps_0$ such that $f$ is $\frac12$-strongly convex at any point $\Theta$ such that $\dop(\Theta, I_D) \leq c$.
\end{proof}

The final lemma we need to prove is \CREFmain{lem:op ball vs frob ball}{Lemma~2.19} which shows that any operator norm ball contains a geodesic ball.

\begin{proof}[Proof of \CREFmain{lem:op ball vs frob ball}{Lemma~2.19}]
If $\Theta = \exp_{I_D}(H)$, then
\begin{align*}
  \norm{\log\Theta}_{\op}
\leq \abs{H_0} + \sum_{a=1}^k \sqrt{d_a} \norm{H_a}_{\op}
\leq \sqrt{\dmax} \left( \abs{H_0} + \sum_{a=1}^k \norm{H_a}_{\op} \right) \\
\leq \sqrt{\dmax} \left( \abs{H_0} + \sum_{a=1}^k \norm{H_a}_F \right)
\leq \sqrt{\dmax} \sqrt{k+1} \norm{H}_F,
\end{align*}
so if $d(\Theta, I_D) = \norm{H}_F \leq r$, then $\dop(\Theta, I_D) = \norm{\log\Theta}_{\op} \leq r \sqrt{(k+1) d_{\max}}$.
\end{proof}

\ifdefined\ARXIV
\section{Proofs of results in Section~\ref{sec:matrix-normal} and Theorem~\ref{thm:matrix-normal}}\label{app:matrix}
\else
\section{Proofs of results in \texorpdfstring{\CREFmain{sec:matrix-normal}{Section~3} and \CREFmain{thm:matrix-normal}{Theorem~1.11}}{Section 3 and Theorem 1.11}}\label{app:matrix}
\fi

Throughout this appendix we assume without loss of generality that $d_1 \leq d_2$.
The proof plan is similar to that in \CREFmain{subsec:proof-sketch}{Section~2.2}, the main difference being that we now work directly with quantum expansion instead of translating into strong convexity.
The key technical result that we will use is \CREFmain{thm:operator-cheeger}{Theorem~3.1}, which states the expansion constant of a random completely positive map can be made constant with \emph{exponentially small} failure probability.
\CREFmain{thm:operator-cheeger}{Theorem~3.1} is proved in \cref{app:cheeky}.

To exploit this result we also use a bound by~\cite{KLR19} which directly controls the operator norm error.
It relies on the notion of a \emph{spectral gap}, which is closely related to quantum expansion and defined as follows.

\begin{definition}[Spectral gap]\label{def:gap}
Let $\Phi\colon\Mat(d_b) \to \Mat(d_a)$ be a completely positive map.
Say $\Phi$ has \emph{spectral gap} $\gamma>0$ if
\begin{align}\label{eq:spectral-gap}
  \sigma_2(\Phi) \leq (1 - \gamma) \frac{\tr \Phi(I_{d_b})}{\sqrt{d_a d_b}}
\end{align}
where $\sigma_2$ denotes the second largest singular value of~$\Phi$.
Note that $\gamma\leq1$.
Moreover, the definition is invariant under rescaling $\Phi \mapsto c\Phi$ for $c>0$.
\end{definition}

\noindent

Recall that by the variational formula for singular values, if we let $K \in \Mat(d_{b})$ be the first (right) singular vector of $\Phi$, we can rewrite the above condition as
\begin{align*}
  \sigma_2(\Phi) = \max_{\langle H, K \rangle = 0} \frac{\norm{\Phi(H)}_{F}}{\norm{H}_{F}} \leq (1 - \gamma) \frac{\tr \Phi(I_{d_b})}{\sqrt{d_a d_b}} .
\end{align*}
On the other hand, the definition of an $(\eps,\eta)$-quantum expander is given in \CREFmain{eq:expansion}{Eq.~(2.10)} as
\begin{align*}
      \norm{\Phi}_0 := \max_{\langle X, I_{d_{a}} \rangle = 0} \max_{\langle H, I_{d_{b}} \rangle = 0} \frac{\langle X, \Phi(H) \rangle}{\norm{X}_{F} \norm{H}_{F}}
\leq \eta \frac{\tr \Phi(I_{d_b})}{\sqrt{d_a d_b}}  .
\end{align*}
Due to the $\eps$-doubly balanced condition in \CREFmain{eq:doubly balanced}{Eq.~(2.9)}, these two notions are closely related, as the following lemma proved in \cite[Lemma~A.3]{FM20} shows.

\begin{lemma}\label{lem:fm20}
There exists a universal constant~$c>0$ with the following property.
If $\Phi$ is an $(\eps,\eta)$-quantum expander and $\eps \leq c(1-\eta)$, then~$\Phi$ has spectral gap~$1-\eta-O(\eps)$.
\end{lemma}

In the next theorem, we state the bound of~\cite[Theorems 1.8 and 3.22]{KLR19} in our language.
Because~$k = 2$, the gradient and Hessian are completely described by the single completely positive map~$\Phi^{(12)}$ (compare the formulas in \CREFmain{lem:gradient,lem:hessian}{Lemmas~2.9 and~2.12} with \CREFmain{eq:hessian channel}{Eq.~(2.6)} and \cref{eq:channel to single marginals}).
Suppose we are given samples $y_1,\dots,y_n$, which we can identify with $d_1\times d_2$ matrices $Y_1,\dots,Y_n$.
Then~$\Phi^{(12)} = \frac1{nD}\Phi_Y$, as discussed below \CREFmain{thm:hess-pisier}{Theorem~2.16}.
Moreover, the double balancedness and spectral gap are invariant under rescaling.
This explains why the following bound can be purely stated in terms of~$\Phi_Y$.
In the following we denote by $\SSPD(d)$  the $d\times d$ positive definite symmetric matrices of unit determinant.

\begin{theorem}\label{thm:klr}
There is a universal constant $C>0$ such that the following holds.
If $d_1 \leq d_2$ and the completely positive map $\Phi_Y$ is $\eps$-doubly balanced and has spectral gap~$\gamma$, where~$\gamma^2 \geq C \eps \log d_1$, then, restricted to $\SSPD(d_1) \ot \SSPD(d_2)$, the function~$f_y$ has a unique minimizer~$P = P_1 \ot P_2$ such that $f_y(P) \geq (1 - \frac{4 \eps^2}{\gamma}) \tr\rho$ and
\begin{align*}
  \max\ \Bigl\{ \norm{P_1 - I_{d_1}}_{\op}, \norm{P_2 - I_{d_2}}_{\op} \Bigr\}
= O\left(\frac{\eps \log d_1}\gamma\right).
\end{align*}
\end{theorem}

We can immediately translate this into a statement about the MLE.

\begin{corollary}[Spectral gap implies MLE nearby]\label{cor:klr}
There is a universal constant $C>0$ such that the following holds.
Let~$\eps,\gamma\in(0,1)$, $1 < d_1 \leq d_2$, and suppose the completely positive map $\Phi_Y$ is $\eps$-doubly balanced and has spectral gap~$\gamma$, where $\gamma^2 \geq C \eps \log d_1$.
Further assume that $\norm{y}_2^2 = nD$.
Then the MLE $\htheta = \htheta_1 \ot \htheta_2$ exists, is unique, and satisfies (using our conventions)
\begin{align*}
  \max \, \Bigl\{ \norm{\htheta_1 - I_{d_1}}_{\op}, \norm{\htheta_2 - I_{d_2}}_{\op} \Bigr\}
= O\Bigl(\frac{\eps \log d_1}\gamma\Bigr).
\end{align*}
\end{corollary}

\begin{proof}
To compute the MLE, we reparameterize by $\htheta_1 = \lambda P_1$ and $\htheta_2 = \lambda P_2$ where $P_1 \in \SSPD(d_1)$, $P_2 \in \SSPD(d_2)$, and $\lambda \in \R_{> 0}$.
Plugging this reparametrization into \CREFmain{eq:neg log likelihood}{Eq.~(1.3)} for~$f_y$ shows that $(\lambda, P_1, P_2)$ solve
\begin{align*}
  \argmin_{\lambda, P_1, P_2} \lambda^2 f_x(P_1 \ot P_2) - \log(\lambda^2).
\end{align*}
In particular, the MLE $\htheta_1, \htheta_2$ exists uniquely if $f_y$ has a unique minimizer~$P = P_1 \ot P_2$ when restricted to $\SSPD(d_1) \ot \SSPD(d_2)$.
Such unique minimizers exist by \cref{thm:klr}.
Given $P_1$, $P_2$, solving the simple one-dimensional optimization problem for $\lambda$ yields
\begin{align*}
  \lambda = \frac1{\sqrt{f_y(P_1)}}.
\end{align*}
By \cref{thm:klr} and using the assumption that $\tr\rho=\frac{\norm{y}_2^2}{nD}=1$, $f_y(P) \geq 1 - \frac{4 \eps^2}{\gamma}$, and we also have $f_y(P) \leq f_y(I_D) = \tr\rho = 1$ since~$P$ is the minimizer in $\SSPD(d_1) \ot \SSPD(d_2)$.
Therefore,
\begin{align*}
  1
\leq \lambda
\leq \left( 1 - \frac{4 \eps^2}{\gamma}\right)^{-1/2}.
\end{align*}
By our assumptions on $\gamma$ and $\eps$, we have $\frac{\eps^2}\gamma  \leq \frac\eps{\gamma^2} \leq \frac1{C \log d_1}$.
Thus, choosing $C>0$ large enough, we obtain
\begin{align*}
  \abs{\lambda - 1}
= O\left(\frac{\eps^2}{\gamma}\right)
\leq O\left(\frac{\eps \log d_1}{\gamma}\right).
\end{align*}
hence in particular $\lambda = O(1)$.
Since also $\norm{P_a - I_{d_a}}_{\op} = O(\eps\log d_1/\gamma)$ by \cref{thm:klr}, we conclude that
\begin{align*}
  \norm{\htheta_a - I_{d_a}}_{\op}
\leq \lambda \norm{P_a - I_{d_a}}_{\op} + \abs{\lambda - 1}
= O\Bigl(\frac{\eps \log d_1}\gamma\Bigr)
\end{align*}
for $a\in\{1,2\}$.
This completes the proof.
%
\end{proof}

\Cref{lem:fm20,thm:klr}, along with what we have shown so far, already imply a preliminary version of \CREFmain{thm:matrix-normal}{Theorem~1.11}.
Indeed, similarly to the proof of \CREFmain{thm:tensor-convexity}{Proposition~2.17}, one can use \CREFmain{prop:gradient-bound}{Proposition~2.11} and \cref{prp:xnorm} to show that under suitable assumptions on $n$, $t$, the completely positive map~$\Phi^{(12)}$ is a $(t \sqrt{{d_2}/{n d_1}}, \eta)$-quantum expander for some universal constant~$\eta\in(0,1)$ with failure probability
\[ e^{ - \Omega( d_2 t^2)} + \left( \frac{\sqrt{nD}}{d_2} \right)^{ - \Omega(d_1)}. \]
By \cref{thm:klr} and \cref{lem:rel op distances}, with the above failure probability the MLE satisfies
\begin{align*}
  \dop(\Theta'_a, \Theta_a) = O\left(t \sqrt{\frac{d_2}{n d_1}} \log d_1\right),
\end{align*}
which matches \CREFmain{thm:matrix-normal}{Theorem~1.11} for the larger Kronecker factor.

As in the proof of \CREFmain{thm:ball-convexity}{Proposition~2.18}, combining the failure probability bound of \CREFmain{thm:operator-cheeger}{Theorem~3.1} with \cref{lem:convexRobustness} yields the next corollary.

\begin{corollary}\label{cor:matrix-convexity}
There are universal constants $C, c > 0$ and $\lambda\in(0,1)$ such that the following holds.
For $d_1 \leq d_2$, let $x=(x_1,\dots,x_n)$ be independent standard Gaussian random variables in~$\R^{d_1d_2}$, where $n \geq C \frac{d_2}{d_1} \max\{\log d_2, t^2 \}$ and $t\geq1$.
Then, with probability at least $1 - e^{ - \Omega( d_2 t^2)}$, the function~$f_x$ is $\lambda$-strongly convex at any point $\Theta\in\P$ such that $\dop(\Theta, I_D) \leq c$.\!
\end{corollary}

We now use \CREFmain{thm:operator-cheeger}{Theorem~3.1} as well as some more refined concentration inequalities to prove \CREFmain{thm:matrix-normal}{Theorem~1.11}.
The additional concentration is required to obtain the tighter bounds on the smaller Kronecker factor.
Throughout this section, we still assume without loss of generality that $d_1 \leq d_2$.
We now implement the strategy discussed in \CREFmain{sec:matrix-normal}{Section~3}, beginning with the concentration bound after one step of flip-flop.

Let $X_1,\dots,X_n$ be random $d_1 \times d_2$ matrices with independent standard Gaussian entries.
Consider new random variables $Y_1,\dots,Y_n$ obtained by one step of the flip-flop algorithm applied to the second, larger Kronecker factor (cf.\ \CREFmain{alg:flip-flop matrix}{Algorithm 1}).
That is, for $i\in[n]$:
\begin{align}\label{eq:one-step}
  Y_i = X_i \left( \frac1{nd_1} \sum_{i=1}^n X_i^T X_i \right)^{-1/2}.
\end{align}
The completely positive map $\Phi^{(12)}$ corresponding to the ``renormalized'' samples $Y_1,\dots,Y_n$ is $\frac1{nD} \Phi_Y$.
By construction, it satisfies
\begin{align*}
  \frac1{n D} \Phi_Y(I_{d_2}) = \frac1{d_2} \sum_{i=1}^{n} X_i \left( \sum_{i=1}^n X_i^T X_i \right)^{-1} X_i^T
\quad\text{and}\quad
  \frac1{n D} \Phi^*_Y(I_{d_1}) = \frac{I_{d_2}}{d_2}.
\end{align*}
Note also that $\tr \Phi_Y(I_{d_2}) = \tr \Phi_Y^*(I_{d_1}) = \|Y\|^2 = nD$. Thus $\Phi_Y$ is $\delta$-doubly balanced if and only if $\|\frac{1}{nD} \Phi_Y(I_{d_2}) - \smash{\frac{I_{d_1}}{d_1}}\|_{\op} \leq \frac{\delta}{d_1}$.

\begin{prop}[Concentration after flip-flop]\label{lem:flipflop-concentration}
There is a universal constant $C'>0$ such that the following holds.
Let $X_1,\dots,X_n$ be random $d_1 \times d_2$ matrices with independent standard Gaussian entries, where $d_1 \leq d_2$.
If $n \geq \frac{d_2}{d_1}$ and $t\geq C'$, then for $\Phi_Y$ with $Y$ as in \cref{eq:one-step} we have, with probability at least $1 - e^{- \Omega( d_1 t^2)}$,
\begin{align*}
  \norm*{\frac1{n D} \Phi_Y(I_{d_2}) - \frac{I_{d_1}}{d_1}}_{\op} \leq t \sqrt{\frac1{nD}}.
\end{align*}
By the above remarks preceding the proposition, this implies $\Phi_Y$ is $t \sqrt{\frac{d_1}{nd_2}}$-doubly balanced.
\end{prop}

The proof of this proposition uses the following result, proved in \cite[Lemma III.5]{hayden2006aspects}, on the overlap of two random projections, as well as the subsequent lemma, \cite[Lemma~5.4]{vershynin2010introduction}, which allows us to employ a net argument.

\begin{theorem}\label{thm:hayden}
Let $P$ be a uniformly (Haar) random orthogonal projection of rank $a$ on $\R^{m}$, let $Q$ be a fixed orthogonal projection of rank $b$ on $\R^{m}$, and let $\eps>0$.
Then,
\begin{align*}
\Pr \left[ \langle P, Q \rangle \not\in (1 \pm \eps) \frac{ab}{m} \right] \leq 2 e^{ - \Omega( ab \eps^2 ) }.
\end{align*}
\end{theorem}

\begin{lemma}\label{lem:versh-net}
Let $A$ be a symmetric $d\times d$ matrix, and let $\mathcal{N}$ be an $\delta$-net of the unit sphere of $\R^d$ for some $\delta \in [0,1)$.
Then,
$$\|A\|_{\op} \leq (1 - 2 \delta)^{-1} \sup_{\xi \in \mathcal{N}} | \langle \xi, A \xi \rangle|.$$
\end{lemma}

With these tools in hand we can now prove \cref{lem:flipflop-concentration}.

\begin{proof}[Proof of \cref{lem:flipflop-concentration}]
For convenience, we consider the differently normalized random variable $Z = Y/\sqrt{nd_1}$.
Note that these satisfy
$Z_i = X_i \Phi_X^*(I_{d_1})^{-1/2} =  X_i (\sum_{i = 1}^n X_i^T X_i)^{-1/2}.$
Thus we wish to prove that
\begin{align}\label{eq:1marg goal}
  \norm*{\sum_{i=1}^n Z_i Z_i^T - \frac{d_2}{d_1} I_{d_1}}_{\op}
\leq t \sqrt{\frac{d_2}{nd_1}}.
\end{align}
Since we are interested in the spectral norm, we will consider the random variable $\langle \xi, \sum_{i = 1}^n Z_i Z_i^{T} \xi \rangle$ for a fixed unit vector $\xi \in \R^{d_1}$.
We will show that this variable is highly concentrated, and apply a union bound over a net of the unit vectors.
To show the concentration, we first cast $\langle \xi, \sum_{i = 1}^n Z_i Z_i^{T} \xi \rangle$ as the inner product between a random orthogonal projection and a fixed one.
Since each $Z_i$ is a $d_1 \times d_2$ matrix, we can consider $Z$ as an $n d_1 \times d_2$ matrix by vertically concatenating the $Z_i$.
By definition of the flip-flop step, $Z^T Z = \sum_{i=1}^n Z_i^T Z_i = I_{d_2}$, so $Z Z^T$ is an orthogonal projection onto a $d_2$-dimensional subspaces of~$\R^{n d_1}$.
In fact, $Z Z^T$ is a uniformly random such projection.
This is because $X$, considered as a $n d_1 \times d_2$ random matrix with i.i.d.\ Gaussian entries, is invariant under left multiplication $X \mapsto OX$ by orthogonal transformations~$O \in O(n d_1)$, hence the same is true for $Z = X (X^T X)^{-1/2}$.
We can now write
\begin{align*}
  \langle \xi,  \sum_{i = 1}^n Z_i Z_i^{T} \xi \rangle = \langle Z Z^T,  \xi \xi^T \ot I_{n} \rangle.
\end{align*}
The matrix $\xi \xi^T \otimes I_{n}$ is a fixed rank $n$ projection on $\R^{n d_1}$.
We now use \cref{thm:hayden} with $P = ZZ^T$, $Q = \xi \xi^T \otimes I_n$, $a = d_2$, $b = n$, and~$m = n d_1$ to obtain
\begin{align}\label{eq:fixed-concentration}
  \Pr\left[ \abs*{\langle \xi, \left( \sum_{i = 1}^n Z_i Z_i^{T} - \frac{d_2}{d_1} I_{d_1} \right) \xi \rangle} > \frac{d_2}{d_1} \eps \right]
\leq 2 e^{ - \Omega( n d_2 \eps^{2} ) }
\end{align}
for any fixed unit vector $\xi\in\R^{d_1}$.

Next we apply a standard net argument for the unit vectors over $\R^{nd_1}$.
We apply \cref{lem:versh-net} with $A = \sum_{i = 1}^n Z_i Z_i^{T} - \smash{\frac{d_2}{d_1}} I_{d_1}$, $d = d_1$, and a net~$\mathcal N$ for~$\delta = 1/4$.
By standard packing bounds (e.g., Lemma~4.2 in \cite{vershynin2010introduction}) we may take $\abs{\mathcal{N}} \leq 9^{d_1}$.
By \cref{eq:fixed-concentration} and the union bound, with failure probability $2 \cdot 9^{d_1} e^{- \Omega (n d_2 \eps^2)}$ we have that $|\langle \xi , A \xi \rangle| \leq \frac{d_2}{d_1} \eps$ for all $\xi \in \mathcal{N}$, and by \cref{lem:versh-net} this event implies $\norm{A}_{\op} \leq 2  \frac{d_2}{d_1} \eps$.
Setting
\[ \eps = t \sqrt{\frac{d_1}{4n d_2}}, \]
we obtain \cref{eq:1marg goal}, i.e.,
\begin{align*}
  \norm*{\sum_{i = 1}^n Z_i Z_i^{T} - \frac{d_2}{d_1} I_{d_1}}_{\op}
\leq 2 \frac{d_2}{d_1} t \sqrt{\frac{d_1}{4n d_2}}
= t \sqrt{\frac{d_2}{n d_1}},
\end{align*}
with failure probability at most $2 \cdot 9^{d_1} e^{- \Omega(d_1 t^2)}$, which is at most $e^{ - \Omega(d_1 t^2)}$, provided $t$ is bounded from below by a large enough constant~$C'>0$.
This concludes the proof.
\end{proof}

The final ingredient needed is the following robustness result for quantum expansion, \cite[Lemma~4.4]{FM20}, which will play a role analogous to our \cref{lem:convexRobustness}.

\begin{lemma}\label{lem:robust expansion}
There is a constant~$c>0$ with the following property:
let $X=(X_1,\dots,X_n)$, $Y=(Y_1,\dots,Y_n)$ be tuples of $d_1\times d_2$ matrices such that $Y_i = X_i R$ for some~$R \in \GL(d_2)$.
Let $0<\eps,\eta<1$.
If $\Phi_X$ is an $(\eps,\eta)$-quantum expander and $\norm{R^TR - I_{d_2}}_{\op} \leq \delta$ for some $\delta\leq c$, then $\Phi_Y$ is an $(\eps + O(\delta), \eta + O(\delta))$-quantum expander.
\end{lemma}

We may finally prove \CREFmain{thm:matrix-normal}{Theorem~1.11}.

\begin{proof}[Proof of \CREFmain{thm:matrix-normal}{Theorem~1.11}]
As discussed in \CREFmain{subsec:proof-sketch}{Section~2.2}, we may assume without loss of generality that~$\Theta_a = I_a$ for~$a\in\{1,2\}$.
We will also assume that $d_1 \leq d_2.$
Let $x=(x_1,\dots,x_n)$ be our tuple of samples, which we can identify with a tuple $X=(X_1,\dots,X_n)$ independent random $d_1\times d_2$ matrices with independent standard Gaussian entries.
Define $Y=(Y_1,\dots,Y_n)$ as in \cref{eq:one-step}.
Consider the following three events:
\begin{enumerate}
\item The operator $\Phi_X$ is a $(t\sqrt{d_2/nd_1},\eta)$-quantum expander for $\eta\in(0,1)$ as in \CREFmain{thm:operator-cheeger}{Theorem~3.1}.
\item The operator $\Phi_Y$ is $t\sqrt{d_1/nd_2}$-doubly balanced.
\item $\abs{\frac{\norm{x}_2^2}{nD} - 1} \leq t\sqrt{d_2/nd_1}$.
\end{enumerate}
By \CREFmain{thm:operator-cheeger}{Theorem~3.1} and our assumptions, the first event occurs with probability at least $1 - e^{ - \Omega(d_2 t^2)}$ provided we choose~$C$ large enough.
By \cref{lem:flipflop-concentration} and our assumptions, the second event occurs with probability at least $1 - e^{ - \Omega( d_1 t^2)}$ assuming $t \geq C'$.
Finally, the third event occurs with probability at least $1 - e^{-\Omega(d_2^2 t^2)}$ by \cref{prp:xnorm} and our assumptions.
By the union bound, all three events occur simultaneously with probability at least $1 - e^{-\Omega(d_1 t^2)}$, which is the desired success probability.

We now show that the three events together imply the desired properties.
We first want to use \cref{lem:robust expansion} to relate the quantum expansion of~$\Phi_X$ and~$\Phi_Y$.
By definition, $Y_i = X_i R$ for $R := ( \frac1{nd_1} \sum_{i=1}^n X_i^T X_i )^{-1/2} = R^T$.
Now note that
\begin{align*}
  R^{-2} - I_{d_2}
= \frac1{nd_1} \sum_{i=1}^n X_i^T X_i - I_{d_2}
= \frac{\norm{x}_2^2}{nD} \left( d_2 \frac{\Phi_X^*(I_{d_1})}{\tr \Phi_X^*(I_{d_1})} - I_{d_2} \right)
+ \left( \frac{\norm{x}_2^2}{nD} - 1 \right) I_{d_2}.
\end{align*}
Therefore, by the first and the third event,
\begin{align*}
  \norm{R^{-2} - I_{d_2}}_{\op}
= O\left(t\sqrt{\frac{d_2}{nd_1}}\right),
\end{align*}
noting that $t\sqrt{\frac{d_2}{nd_1}} \leq \frac1{\sqrt C}$ can be made smaller than any constant by choosing~$C$ large enough.
This also implies that
\begin{align}\label{eq:RR bound}
  \norm{R^TR - I_{d_2}}_{\op}
= \norm{R^2 - I_{d_2}}_{\op}
= O\left(t\sqrt{\frac{d_2}{nd_1}}\right).
\end{align}
Noting again that the right-hand side can be made smaller than any universal constant, we can now apply \cref{lem:robust expansion} to see that $\Phi_Y$ is a $(t\sqrt{d_1/nd_2},\eta')$-quantum expander for some universal constant~$\eta'\in(0,1)$ (the double balancedness follows from the second event!).
With this, \cref{lem:fm20} shows that $\Phi_Y$ has spectral gap~$\gamma$ for a universal constant~$\gamma\in(0,1)$.

Finally, noting that $\norm{Y}_2^2 = \sum_{i=1}^n \tr Y_i^T Y_i = nD$ and using our assumption on $n$, provided we choose $C$ large enough we may apply \cref{cor:klr} with $\eps = t\sqrt{d_1/nd_2}$.
We obtain:
\begin{align}\label{eq:MLE Y bound}
  \max \, \Bigl\{ \norm{\htheta_1(Y) - I_{d_1}}_{\op}, \norm{\htheta_2(Y) - I_{d_2}}_{\op} \Bigr\}
= O\Bigl(t\sqrt{\frac{d_1}{nd_2}} \log d_1\Bigr),
\end{align}
where $\htheta_a(Y)$ denotes components the MLE for the samples $Y=(Y_1,\dots,Y_n)$.
By equivariance, the components of the MLE for the samples $X=(X_1,\dots,X_n)$ are then given by~$\htheta_1(X) = \htheta_1(Y)$ and $\htheta_2(X) = R \, \htheta_2(Y) R$.
This immediately yields the bound
\begin{align*}
  \Dop(\htheta_1(X)\Vert\Theta_1)
= \Dop(\htheta_1(X)\Vert I_{d_1})
= O\Bigl(t\sqrt{\frac{d_1}{nd_2}} \log d_1\Bigr).
\end{align*}
To bound $\Dop(\htheta_2(X)\Vert\Theta_2)$, we use invariance of $\Dop$ and the approximate triangle inequality (\cref{lem:triangle-ineq}) to write
\begin{align*}
  \Dop(\htheta_2(X)\Vert\Theta_2)
&= \Dop(\htheta_2(X)\Vert I_{d_2})
= \Dop(R \, \htheta_2(Y) R\Vert I_{d_2})
= \Dop(\htheta_2(Y)\Vert R^{-2}) \\
&= O\left(\Dop(\htheta_2(Y)\Vert I_{d_2}) + \Dop(I_{d_2}\Vert R^{-2})\right) \\
&= O\left(t\sqrt{\frac{d_1}{nd_2}} \log d_1\right)
+ O\left(t\sqrt{\frac{d_2}{nd_1}}\right)
= O\left(t\sqrt{\frac{d_2}{nd_1}} \log d_1\right)
\end{align*}
using \cref{eq:MLE Y bound,eq:RR bound}; by choosing $C$ large enough we can ensure that the right-hand side is smaller than any universal constant, which justifies the application of \cref{lem:triangle-ineq}.
Reparametrizing $t$ by $t \leftarrow t/C'$ allows us to assume $t \geq 1$ rather than $t \geq C'$.
The bounds on $\dop$ follow from the above and from \cref{lem:rel op distances} by choosing large enough $C$.
\end{proof}

\ifdefined\ARXIV
\section{Proofs of results in Section~\ref{sec:lower}}\label{app:lower}
\else
\section{Proofs of results in \texorpdfstring{\CREFmain{sec:lower}{Section~4}}{Section 4}}\label{app:lower}
\fi

We first recall and, for completeness, prove well-known lower bounds on the accuracy of any estimator for the precision matrix in the Frobenius and operator error from independent samples of a Gaussian.
Informally, these bounds imply that no estimator for a $d\times d$ precision matrix from $n$ samples can have accuracy better than $\sqrt{d^{2}/n}$ in Frobenius error or $\sqrt{d/n}$ in operator norm error with probability more than $1/2$.

\begin{prop}[Frobenius and operator error]\label{prp:standard-lower}
There is $c > 0$ such that the following holds.
Let $x \in \R^{d \times n}$ denote $n$ independent random samples from a Gaussian with precision matrix $\Theta \in \PD(d)$.
Consider any estimator $\htheta = \htheta(x)$ for the precision matrix~$\Theta$, and let $B\subset \PD(d)$ denote the operator norm ball centered at $I_d$ of radius~$1/2$.
\begin{enumerate}
\item Let $\delta^2 = c \, \min \left\{1,d^2/n\right\}$. Then,
\begin{align}
\sup_{\Theta \in B} \Pr\left[ \| \htheta - \Theta\|_F \geq \delta\right] \geq \frac{1}{2}.\label{eq:frob-lower}
\end{align}
\item Let $\delta^2 = c \, \min \left\{1,d/n\right\}$. Then,
 \begin{align}
\sup_{\Theta \in B} \Pr\left[ \| \htheta - \Theta\|_{\op} \geq \delta\right] \geq \frac{1}{2}. \label{eq:op-lower}
\end{align}
\end{enumerate}
As a consequence, we have
\begin{align*}
  \sup_{\Theta \in B}\E[\| \htheta - \Theta\|_F^2] =\Omega\left( \min \left\{\frac{d^2}{n},1\right\}\right)
\text{ and } \sup_{\Theta \in B}\E[\| \htheta - \Theta\|_{\op}^2] = \Omega\left( \min \left\{\frac{d}{n},1\right\}\right).
\end{align*}
\end{prop}

The proof uses Fano's method with mutual information bounded by relative entropy, as in \cite{yang1999information}, and the relationship between the Frobenius error and the relative entropy (which is proportional to Stein's loss).

\begin{lemma}[Fano's inequality]\label{lem:fano}
Let $\{P_i\}_{i \in [m]}$ be a finite set of probability distributions over a set $\mathcal X$, and let $T: \mathcal X \to [m]$ be an estimator for $i$ from a sample of $P_i$. Then
\[ \max_{i\in [m]} \Pr_{X \sim P_i}[T(X) \neq i] \geq 1 - \frac{ \log 2 + \max_{i,j \in [m]} \DKL(P_i\Vert  P_j)}{\log m}. \]
\end{lemma}

\begin{proof}[Proof of \cref{prp:standard-lower}]
We first prove \cref{eq:frob-lower}, the lower bound on estimation in the Frobenius norm.
We begin by the standard reduction from estimation to testing.
Let $V_0$ be a $1$-separated set in the Frobenius ball $B_F$ of radius~$1$ in the $d\times d$ symmetric matrices, i.e., the set $B_F = \{A: A \text{ Symmetric}, \|A\|_F \leq 1\}$.

We may take $V_0$ to have cardinality $m \geq 2^{d(d+1)/2}$ because $B_F$ is a Euclidean ball of radius~$1$ in the linear subspace of $d\times d$ symmetric matrices, which has dimension $d(d+1)/2$, and hence any maximal Frobenius $1/2$-packing (collection of disjoint radius $1/2$ Frobenius balls) in $B_F$ has cardinality at least $2^{d(d+1)/2}$.
Let $0 \leq \delta \leq 1/2$, and let $V = I_d + \delta V_0 = \{I_d + \delta v: v \in V_0\}$.
Write $V = \{\Theta_1, \dots, \Theta_m\}$.
Note that $V$ is contained within the operator norm ball $B$.
Let $P_i =\mathcal{N}(0, \Theta^{-1}_i)^{\ot n}$ for $i\in[m]$, and define the estimator~$T$ by
\[ T(x) = \argmin_{i \in [m]} \|\Theta_i - \htheta(x)\|_F. \]
Then, because $V$ is $2\delta$-separated,
\begin{align}\label{eq:pr vs est}
  \Pr_{X \sim P_i} \left[T(X) = i\right] \geq \Pr\left[\|\htheta -  \Theta_i\|_F \leq \delta\right].
\end{align}
In order to apply Fano's inequality, we use the well-known fact that $\DKL(P_i\Vert  P_j) = n \DKL(\mathcal{N}(0, \Theta_i^{-1})\ \Vert  \ \mathcal{N}(0, \Theta_j^{-1})) = O(n\DF(\Theta_j \ \Vert  \ \Theta_i)^2)$ when $\Theta_i^{-1}\Theta_j$ has eigenvalues uniformly bounded away from zero by the proof of \cref{lem:rel op distances}.
This condition on the eigenvalues holds because $I_d/2 \preceq \Theta_j, \Theta_j \preceq 3I_d/2$ for $i,j \in [m]$ by our assumption that $\delta \leq 1/4$.

Moreover, for $i \in [m]$, we have $\kappa(\Theta_i) \leq 3$ and so $\DF(\Theta_j\Vert  \Theta_i) \asymp \|\Theta_i - \Theta_j\|_F = O(\delta)$ by \cref{prop:absRelation}.
Thus we have $\DKL(P_i\Vert  P_j) \leq Cn \delta^2$ for some absolute constant $C$.
Then, by \cref{lem:fano},
\begin{align*}
  \min_{i \in [m]} \Pr_{X \sim P_i}[T(X) = i] \leq \frac{  \log 2 + C n \delta^2}{d(d+1)(\log 2)/2 }.
\end{align*}
If $\delta^2 = c \, \min\{ \frac{d^2}{n}, 1\}$, the right-hand side of the inequality above is bounded by $\frac{1}{2}$ and the assumption $\delta \leq 1/4$ is satisfied provided $c$ is a small enough absolute constant.
In view of \cref{eq:pr vs est}, it follows that
\[ \min_{i \in [m]} \Pr\left[ \|\htheta - \Theta_i\|_F \leq \delta\right] \leq 1/2. \]
Because $V \subset B$, this proves \cref{eq:frob-lower}.

To obtain \cref{eq:op-lower}, the lower bound in operator norm, instead start with a packing $V_0$ of the unit operator norm ball of cardinality $m \geq 2^{d(d+1)/2}$ and define $V = \{ \Theta_1, \dots, \Theta_m \}$ as above.
We modify the proof by bounding $\DKL(P_i \Vert  P_j) = O(n \| \Theta_i - \Theta_j\|_F^2) = O(n d \|\Theta_i - \Theta_j\|_{\op}^2) \leq C nd \delta^2.$
Proceeding as before, we find that for $\delta = c \, \min \{\frac{d}{n}, 1\}$,
\[ \min_{i \in [m]} \Pr\left[ \|\htheta - \Theta_i\|_{\op} \leq \delta\right] \leq 1/2. \]
Again, we have $V \subset B$, so \cref{eq:op-lower} follows.
\end{proof}

The above proof shows the necessity of a scale-invariant dissimilarity measure to obtain error bounds that are independent of the ground truth precision matrix $\Theta$.
Indeed, replacing the packing $V$ by $C V$ for $C \to \infty$ in the proof shows that $\sup_{\Theta \in C B} \Pr[ \| \htheta - \Theta\|_F \geq C \delta ] \geq \frac{1}{2}$.
That is, no fixed bound can be obtained with probability $1/2$.
The result just obtained implies similar lower bounds on the Fisher-Rao and Thompson metrics.

\begin{proof}[Proof of \CREFmain{cor:relative-lower}{Proposition~4.1}]
Since $\kappa(\Theta) \leq 3$ for $\Theta \in B$, \cref{lem:rel op distances,prop:absRelation} imply $\|\Theta - \htheta\|_F \asymp \dFR(\htheta, \Theta)$ and $\|\Theta - \htheta\|_{\op} \asymp \dop(\htheta, \Theta)$.
Thus, the result follows from \cref{prp:standard-lower}.
\end{proof}


We finally give the proof of \CREFmain{lem:reduce-lower}{Lemma~4.3}.

\begin{proof}[Proof of \CREFmain{lem:reduce-lower}{Lemma~4.3}]
If $d_2 \leq nd_1$, then setting $\Theta_2 = I_{d_2}$ shows that $\htheta_1$ has access to precisely $n d_2$ samples from a Gaussian $\R^{d_1}$ with precision matrix $\Theta_1$.
Thus we may take $\tilde{\Theta} = \htheta_1$ in that case, completing the proof. The harder case is $d_2 > n d_1$.

For intuition, let $B$ be any $d_2\times d_2$ matrix such that the last $d_2 - nd_1$ columns are zero.
Consider $n$ samples $X_{i} := \sqrt{\Sigma_{1}} Z_i B^T$, where $Z_i$ are i.i.d standard Gaussian $d_1\times d_2$ matrices. Then any estimate for $\hat{\Theta}_{1}(X)$ has access to at most $n \cdot n d_1$ samples of the Gaussian on $\R^{d_1}$ with precision matrix $\Theta_1$ because $Z_i B^T$ depends only on the first $d_1$ columns of each $Z_i$.
Therefore \CREFmain{cor:relative-lower}{Proposition~4.1} applies to give lower bound $\dFR(\hat{\Theta}_{1}(X), \Theta_{1})^{2} \gtrsim \frac{d_{1}^{2}}{n^{2} d_{1}}$.

However, in order for this to be a well-defined input in the matrix normal model, we must supply \emph{invertible} $B$ with $\Theta_2 = (BB^T)^{-1}$.
For $\delta \geq 0$, let
the first $nd_1$ columns of $B_\delta$ be an orthonormal basis for a random $nd_{1}$-dimensional subspace of $\R^{d_{2}}$,
and let the remaining entries be i.i.d uniform in $[-\delta, \delta]$ (the precise distribution of the remaining entries does not matter as long as they are independent, continuous, and small).
Let $Y_\delta:=(\sqrt{\Sigma_1} Z_1 B_\delta^T, \dots, \sqrt{\Sigma_1} Z_n B_\delta^T)$ denote the resulting random variable with $B_\delta$ and $Z$ chosen independently.
If $\delta = 0$, then, by the argument above, with access to the random variable $Y_\delta:=(\sqrt{\Sigma_1} Z_1 B_\delta^T, \dots, \sqrt{\Sigma_1} Z_n B_\delta^T)$ the estimator $\widehat{\Theta}_1(Y)$ has access to at most $n^2d_1$ samples of a Gaussian on $\R^{d_1}$ with precision matrix $\Theta_1$.
We claim that as $\delta \to 0$, the distribution of $Y_\delta$ tends to that of $Y_0$ in total variation distance.
Thus the distribution of $\widehat{\Theta}_1(Y_\delta)$ converges to that of $\widehat{\Theta}_1(Y_0)$ in total variation.
Since $Y_0$ only depends on $n^2d_1$ samples to the Gaussian on $\R^{d_1}$ with precision matrix $\Theta_1$, which we call $Y$, defining $\tilde{\Theta}(Y) = \widehat{\Theta}_1(Y_0)$ proves the theorem.
\footnote{Actually, as $B$ has a probability zero chance of being singular, the final family of densities $Y'_\delta$ we will use is $Y_\delta$ conditioned on $B$ being invertible.
As $B$ is invertible with probability $1$ for $\delta > 0$, the total variation distance between $Y'_\delta, Y_\delta$ is zero for all $\delta > 0$ and hence $Y'_\delta$ converges to $Y_0$ in total variation distance provided $Y_\delta$ does.}

It remains to prove that $Y_\delta$ converges to $Y_0$ in total variation distance.
First note that $Y_\delta = Y_0 + \delta W$ where $W_i = \sqrt{\Theta_1} Z_i C^T$, where $C$ is a random matrix where the first $nd_1$ columns are zero and the last $d_2 - n d_1$ columns have entries i.i.d uniform on $[-1, 1]$.
Note that the random variables $Y_0$ and $W$ are independent, as the entries of $Z$ are i.i.d. and the supports of $B_{0}$ and $C$ are disjoint.
If we can show that $Y_0$ has a density with respect to the Lebesgue measure on $\R^{nd_1d_2}$, then $Y_0 + \delta W$ converges to $Y_0$ in total variation distance as $\delta \to 0$.
This follows because $Y_0 + \delta W$ has a density obtained by convolving the density of $Y_0$ with the law of $\delta W$, which is an $L_1$ function.
The density of $Y_0 + \delta W$ then converges to that of $Y_0$ in $L_1$ by the continuity of the convolution operator in $L_1$.\footnote{We thank Oliver Diaz for communicating a proof of this fact.}

By invertibility of $\Sigma_1$, it is enough to show that $Y_0$ has a density when $\Sigma_1 = I_{d_1}$.
Consider $Y_0 = ( B_0 Z_1^T, \dots, B_0 Z_n^T)$.
We may think of $Y_0$ as the $d_2 \times n d_1$ random matrix obtained by horizontally concatenating the matrices $B_0Z_i^T$.
\footnote{
Almost every matrix of these dimensions has rank $n d_1$, but if we had set even more of the columns of $B_0$ to zero then $Y_0$ would have rank \emph{less} than $n d_1$ with probability $1$ and hence would not have a density.
This is why we cannot push this argument any further.}

Now consider the $nd_1$ random vectors in $ \R^{d_2}$ that are the columns of the matrix $Y_{0}$.
Because $B_0$ is supported only in its first $nd_1$ columns, the joint distribution of these random vectors may be obtained by sampling $n d_1$ independent standard Gaussian vectors $v_j$ on $\R^{nd_1}$ and then multiplying them by the $d_2 \times nd_1$ matrix $B'$ that is the restriction of $B_0$ to its first $nd_1$ columns.
We have chosen $B'$ such that it is an isometry into a uniformly random subspace of $\R^{d_2}$ of dimension $nd_1$.
Thus $Bv_j/\|v_j\|$ are $nd_1$ many independent, random unit vectors in $\R^{d_2}$.
As the $\|v_j\|$ are also independent, $B v_j$ are thus independent. Each marginal $Bv_i$ has a density; one may sample it by choosing a uniformly random vector and then choosing the length $\|v_i\|$, hence the density is a product density in spherical coordinates. The joint density of the $Bv_j$ is then the product density of the marginal densities.
\end{proof}

\ifdefined\ARXIV
\section{Proofs of results in Section~\ref{sec:flip-flop}}\label{app:flip-flop}
\else
\section{Proofs of results in \texorpdfstring{\CREFmain{sec:flip-flop}{Section~5}}{Section 5}}\label{app:flip-flop}
\fi

We first record a structural property of the flip-flop algorithm in \CREFmain{alg:flip-flop matrix,alg:flip-flop}{Algorithms~1 and~2}.
Note that at the end of each iteration, we update only a single Kronecker factor~$\otheta_a$.
This update has the following property.

\begin{lemma}[Flip-flop update]\label{lemma:flip-flop-update}
  Let $t\in\{1,\dots,T-1\}$ and assume the flip-flop algorithm has not terminated before the $(t+1)$-st iteration.
  Then $\rho_{t+1}^{(a)} = \frac{I_{d_a}}{d_a}$, where $a\in[k]$ denotes the index chosen in the $t$-th iteration.
  As a consequence, $\tr\rho_t=1$ for $t=2,\dots,T$.
\end{lemma}

\begin{proof}
Let $\otheta$ denote the precision matrix at the beginning of the $t$-th iteration.
Then,
\begin{align*}
  \rho_{t+1}^{(a)}
&= \left( \frac1{d_a} \otheta_a^{1/2} \left( \rho_t^{(a)} \right)^{-1} \otheta_a^{1/2} \right)^{1/2} \otheta_a^{-1/2}
\rho_t^{(a)}
\otheta_a^{-1/2} \left( \frac1{d_a} \otheta_a^{1/2} \left( \rho_t^{(a)} \right)^{-1} \otheta_a^{1/2} \right)^{1/2} \\
&= \frac1{d_a} \left( \otheta_a^{1/2} \left( \rho_t^{(a)} \right)^{-1} \otheta_a^{1/2} \right)^{1/2}
\!\!\left( \otheta_a^{1/2} \left( \rho_t^{(a)} \right)^{-1} \otheta_a^{1/2} \right)^{-1}\!\!
\left( \otheta_a^{1/2} \left( \rho_t^{(a)} \right)^{-1} \otheta_a^{1/2} \right)^{1/2} \\
&= \frac1{d_a} I_{d_a}.
\end{align*}
\end{proof}

\noindent
In view of \CREFmain{lem:gradient}{Lemma 2.9} and \CREFmain{remark:gradient-everywhere}{Remark~2.10} and the assumption on the initial guess in \CREFmain{alg:flip-flop matrix}{Algorithm 1}, the above means that in each iteration $\nabla_0 f_x(\otheta) = 0$.
Moreover, from the second iteration onwards, $\nabla_a f_x(\otheta) = 0$ for the $a\in[k]$ chosen in the preceding iteration.
Thus the flip-flop algorithm can be understood as carrying out an alternating minimization or coordinate descent of the objective function~$f_x$.

Next, we discuss direct generalizations of standard convergence results for descent methods under strong convexity to the geodesically convex setting.
To prove that flip-flop converges once the initial conditions are satisfied, we need the following general lemma on strongly geodesically convex functions, which tells us that once the gradient is small then the point must be inside a sublevel set of our function which is contained in a ball where our function is strongly convex.
This result is stated in \cite[Lemma~4.7]{FM20} for the manifold of positive definite matrices of determinant one, but the proof uses no specific properties of this manifold beyond the fact that it is a Hadamard manifold.
Thus it holds for $\SPD$ as well.

\begin{lemma}\label{lem:gradient-strong-convexity-fm}
Let $f \colon \SPD \to \R$ be a geodesically convex function with optimizer $z \in \SPD$ (i.e. $\nabla f(z) = 0$), and further assume that $f$ is $\lambda$-strongly geodesically convex on the ball $B_r(z)$.
If $y\in\SPD$ is such that~$\norm{\nabla f(y)}_F < \lambda r/8$, then $y$ is contained in a sublevel set $S$ of $f$ which in turn is contained in~$B_r(z)$.
In particular, $f$ is $\lambda$-strongly geodesically convex on $S$.
\end{lemma}

The next lemma shows that any descent method which manages to significantly decrease the value of the function with respect to the gradient, if starting from a sublevel set where the function is strongly convex, will converge quickly to the optimum.
The proof of the lemma is a straightforward translation of the proof of \cite[Lemma~4.11]{FM20} which we give here for completeness.

\begin{lemma}\label{lem:descent-sublevel-set}
  Let $f \colon \SPD \rightarrow \R$ be $\lambda$-strongly geodesically convex on a sublevel set~$S$.
  Let $x_0 \in S$ and let $\alpha,\beta>0$ such that $\norm{\nabla f(x_0)}_F^2 \leq \beta$ and $\{x_t\}_{t\in[T]}$ be a sequence satisfying
  \begin{align}\label{eq:descent prop}
    f(x_t) \leq f(x_{t-1}) - \alpha \cdot \min \bigl\{\beta, \, \norm{\nabla f(x_{t-1})}^2_F \bigr\},
  \end{align}
  for $t\in[T]$.
  Then,
  \begin{align*}
    \min_{0\leq t\leq T} \norm{\nabla f(x_t)}^2_F \leq \norm{\nabla f(x_0)}_F^2 \cdot 2^{- T \alpha \lambda}.
  \end{align*}
\end{lemma}

\begin{proof}
  Let $f^*$ be the minimum value of the function $f$.
  Since $f$ is $\lambda$-strongly geodesically convex on~$S$, we have
  \begin{align}\label{eq:g convex standard estimate}
    f^* \geq f(x) - \frac{1}{2\lambda} \norm{\nabla f(x)}_F^2
  \end{align}
  for any $x \in S$.
  Since $\{x_t\}$ is a descent sequence, i.e., $f(x_t) \leq f(x_{t-1})$ for all $t\in[T]$, we know that each $x_t \in S$.
  Therefore, \cref{eq:g convex standard estimate} holds for any $x_t$, $0 \leq t \leq T$.

  We claim that for any $x_t$ such that $\eps := \norm{\nabla f(x_t)}_F^2 \leq \beta$, there exists $\ell \leq 1/\alpha \lambda$ such that $\norm{\nabla f(x_{t + \ell})}_F^2 \leq \eps/2$. This is enough to conclude the proof of the lemma, as with this claim we see that we halve the squared norm of the gradient at every sequence of $1/\alpha \lambda$ steps.

  To prove the claim, we assume that $\norm{\nabla f(x_{t+\ell})}_F^2 \geq \eps/2$ for all $\ell\in[m]$ (this is also true for $\ell=0$).
  We wish to show that $m \leq 1/\alpha\lambda$.
  To see this, note that from \cref{eq:descent prop} we have
  \begin{align*}
    f(x_{t+\ell})
  \leq f(x_{t+\ell-1}) - \alpha \cdot \min \bigl\{\beta, \, \norm{\nabla f(x_{t+\ell-1})}^2_F \bigr\}
  \leq f(x_{t+\ell-1}) - \frac{\alpha\eps}2
  \end{align*}
  for all $\ell\in[m]$, and therefore
  \begin{align*}
    f(x_{t+m}) \leq f(x_t) - \frac{\alpha\eps m}2.
  \end{align*}
  On the other hand, \cref{eq:g convex standard estimate} implies that
  \begin{align*}
    f(x_{t+m})
  \geq f^* \geq f(x_t) - \frac{1}{2\lambda} \norm{\nabla f(x_t)}_F^2
  \geq f(x_t) - \frac{\eps}{2\lambda}.
  \end{align*}
  Together, we find that $m \leq 1/\alpha\lambda$ as claimed.
  This concludes our proof.
\end{proof}

We now show that the flip-flop algorithm produces a descent sequence as in \cref{eq:descent prop}.

\begin{lemma}[Descent]\label{lem:tensor-descent-lemma}
Let $k \geq 2$ and $t\in\{2,\dots,T-1\}$.
Assume that the flip-flop algorithm has not terminated before the $(t+1)$-st iteration.
Let $\smash{\otheta^{(t)}}$, $\smash{\otheta^{(t+1)}}$ denote the precision matrices at the beginning of the $t$-th and the $(t+1)$-st iteration, respectively.
Then,
\begin{align*}
f_x(\otheta^{(t+1)}) \leq f_x(\otheta^{(t)}) - \frac1{6(k-1)} \min \left\{ \frac{k-1}{\dmax}, \norm{\nabla f_x(\otheta^{(t)})}_F^2 \right\}.
\end{align*}
\end{lemma}
\begin{proof}
  Recall that
  \begin{align*}
    f_x(\otheta^{(t)}) &= \tr\rho_t - \frac1D\log\det\otheta^{(t)}.
  \end{align*}
  and similarly for $f_x(\otheta^{(t+1)})$.
  By \cref{lemma:flip-flop-update}, we have $\tr\rho_t = \tr\rho_{t+1} = 1$.
  Moreover, by definition of the update step
  \begin{align*}
    \frac1D\log\det\otheta^{(t+1)} = \frac1D\log\det\otheta^{(t)} - \frac1{d_a}\log\det\left(d_a \rho_t^{(a)}\right).
  \end{align*}
  It follows that
  \begin{align*}
    f_x(\otheta^{(t+1)}) = f_x(\otheta^{(t)}) + \frac1{d_a}\log\det\left(d_a \rho_t^{(a)} \right).
  \end{align*}
  Lemma 5.1 in \cite{GGOW19} states that for any $d\times d$ positive semidefinite matrix~$Z$ of trace~$d$, the following inequality holds:
  \begin{align*}
    \log\det(Z) \leq -\frac16 \min \, \Bigl\{ \norm{Z - I_d}_F^2, 1 \Bigr\}.
  \end{align*}
  Applying this with $Z = d_a \rho_t^{(a)}$, we obtain
  \begin{align*}
    \frac1{d_a}\log\det\left(d_a \rho_t^{(a)} \right)
  &\leq -\frac16 \min\left\{ \norm{\rho_t^{(a)} - \frac{I_{d_a}}{d_a}}_F^2, \frac1{d_a} \right\} \\
  &\leq -\frac16 \min\left\{ \norm{\nabla_a f_x(\otheta^{(t)})}_F^2, \frac1{d_a} \right\} \\
  &\leq -\frac16 \min\left\{ \frac{\norm{\nabla f_x(\otheta^{(t)})}_F^2}{k-1}, \frac1{\dmax} \right\}.
  \end{align*}
  The equality follows from \CREFmain{lem:gradient}{Lemma 2.9} and \CREFmain{remark:gradient-everywhere}{Remark~2.10}.
  In the last inequality we used that $\nabla_0 f(\otheta^{(t)}) = 0$ and at least one other component of the gradient is zero, as follows from \cref{lemma:flip-flop-update}, and that $a\in[k]$ is the index where the gradient has largest norm.
\end{proof}

We can also use \cref{lem:tensor-descent-lemma} to show that the flip-flop algorithm reaches a point with small enough gradient relatively quickly.
This is given by the following lemma, which follows the analysis given by~\cite{GGOW19,burgisser2017alternating}:

\begin{lemma}[Flip-flop reduces gradient]\label{lem:flip-flop-sinkhorn}
For any $\gamma>0$, \CREFmain{alg:flip-flop}{Algorithm 2} with initial guess $\wttheta$ satisfying $\nabla_0 f_x(\wttheta) = 0$ reaches some~$\otheta$ such that $\norm{\nabla f_x(\otheta)}_F < \gamma$ within the first
\begin{align*}
T_0 = \left\lceil 3 + 6(k+1)
\cdot \dop(\wttheta, \htheta) \cdot
\max \left\{ \frac{\dmax}{k-1}, \frac1{\gamma^2} \right\} \right\rceil
\end{align*}
iterations, where $\displaystyle \htheta := \arg\inf_{\Theta\in\SPD} f_x(\Theta)$.

Consequently,
if $\wttheta := \dfrac{1}{f_x(I_D)} \cdot I_D$ and $f_x^* := \inf_{\Theta\in\SPD} f_x(\Theta)$, we have
\begin{align*}
T_0 = \left\lceil 3 + 6(k+1)
\bigl( 1 + \log f_x(I_D) - f_x^* \bigr)
\max \left\{ \frac{\dmax}{k-1}, \frac1{\gamma^2} \right\} \right\rceil
\end{align*}
\end{lemma}
\begin{proof}
We denote by $\otheta^{(t)}$ the precision matrices at the beginning of the $t$-th iteration of the flip-flop algorithm.
In particular, we have $\otheta^{(1)} = \wttheta$.

By \cref{lem:tensor-descent-lemma}, using that $\tr\rho_1 = 1$, we have that
\begin{align*}
  f_x^*
\leq f_x(\otheta^{(T_0)})
\leq f_x(\otheta^{(1)}) - \frac1{6(k-1)} \sum_{t=2}^{T_0-1} \min \left\{ \frac{k-1}{\dmax}, \norm{\nabla f_x(\otheta^{(t)})}_F^2 \right\}
\end{align*}
(we omit the summand for $t=1$).
Therefore, if $\norm{\nabla f_x(\otheta^{(t)})}_F \geq\gamma$ for $t=2,\dots,T_0-1$, then
\begin{align*}
  \frac{T_0-2}{6(k-1)} \min \left\{ \frac{k-1}{\dmax}, \gamma^2 \right\}
\leq f_x(\otheta^{(1)}) - f_x^*
= f_x(\wttheta) - f_x^* = \dop(\wttheta, \htheta)
\end{align*}
where the last equality follows since $\nabla_0 f_x(\wttheta) = 0$.
This implies the desired bound.
Now, when $\wttheta = \dfrac{1}{f_x(I_D)} \cdot I_D$, the right-hand side of the above inequality becomes $1 + \log f_x(I_D) - f_x^*$.
\end{proof}

We are now ready to prove fast convergence of flip-flop under suitable initial conditions.

\begin{proof}[Proof of \CREFmain{prop:meta}{Theorem~5.2}]
By the triangle inequality for $\dop$, the first and third assumptions imply that $f_x$ is $\lambda$-strongly geodesically convex at all $\Theta' \in \SPD$ such that $\dop(\Theta', \htheta) \leq \zeta/2$.
By \CREFmain{lem:op ball vs frob ball}{Lemma~2.19}, it is $\lambda$-strongly geodesically convex on the geodesic ball $B_r(\htheta)$ of radius
\begin{align*}
  r = \frac{\zeta}{2\sqrt{(k+1)\dmax}}.
\end{align*}

First note that our error bounds on the MLE follow if \CREFmain{alg:flip-flop}{Algorithm 2} reaches the stopping criterion within~$T$ iterations, that is, if we reach a precision matrix $\otheta=\otheta^{(t)}$ such that $\norm{\nabla f_x(\otheta)}_F \leq \delta$.
In fact, since $\delta < \frac{\lambda r}8$ by our assumption on~$\delta$, \cref{lem:gradient-strong-convexity-fm} applies (with~$x=\htheta$, $y=\otheta$) and shows that $\otheta \in B_r(\htheta)$.
Now \CREFmain{lem:convex-ball}{Lemma~2.7} applies, since in particular~$r > 2\delta/\lambda$, and shows that $\otheta \in B_{\delta/\lambda}(\htheta)$, that is,
\begin{align*}
  d(\otheta, \htheta) \leq \frac{\delta}{\lambda} \quad \Rightarrow \quad \dFR(\otheta_a, \htheta_a) \leq \sqrt{\frac{d_a}{2}} \cdot \frac{\delta}{\lambda}
\end{align*}
for all $a\in[k]$.
This is the desired distance to the MLE.

We will now analyze the iteration complexity of \CREFmain{alg:flip-flop}{Algorithm 2} with distinct initial guesses:

\textbf{Case 1:} initial guess $\wttheta$ s.t. $\nabla_0 f_x(\wttheta) = 0$.

We first reason about the number of steps required before strong convexity applies.
By \cref{lem:flip-flop-sinkhorn} with $\wttheta$ and $\gamma = \lambda r/8 = \lambda\zeta/16\sqrt{(k+1)\dmax} \leq \sqrt{(k-1)/\dmax}$, within at most
\begin{align}\label{eq:T_0 expr}
 T_0 = \left\lceil 3 + 6(k+1) \cdot \dop(\wttheta, \htheta) \cdot  \frac{64}{r^2 \lambda^2} \right\rceil
\end{align}
iterations, the algorithm reaches a point~$\otheta^{(t_0)}$ such that
\begin{align*}
  \norm{\nabla f_x(\otheta^{(t_0)})}_F < \frac{\lambda r}{8}.
\end{align*}
We can use \cref{lem:gradient-strong-convexity-fm} (with $z=\htheta$, $y=\otheta^{(t_0)}$) to see that $\otheta^{(t_0)}$ is contained in a sublevel set of~$f_x$ on which~$f_x$ is $\lambda$-strongly geodesically convex.

Note also that $\norm{\nabla f_x(\otheta^{(t_0)})}_F^2 \leq \beta := \frac{k-1}{\dmax}$ because $\lambda r/8 = \lambda\zeta/16\sqrt{(k+1)\dmax} \leq \sqrt{(k-1)/\dmax}$ by our assumption that $\zeta \leq 16 \sqrt{(k+1)(k-1)} / \lambda$.

Therefore, \cref{lem:tensor-descent-lemma} shows that each subsequent step of the algorithm will decrease the value of the objective function in accordance with the requirements of \cref{lem:descent-sublevel-set}, with parameters $\alpha = \frac1{6(k-1)}$ and $\beta$ as defined above.
Thus, for any $\delta > 0$, within at most
\begin{align*}
   T_1
&:= \left\lceil \frac{6(k-1)}{\lambda} \log_2 \frac{\norm{\nabla f_x(\otheta^{(t_0)})}_F^2}{\delta^2} \right\rceil
\leq \left\lceil \frac{12(k-1)}{\lambda} \log_2 \frac{\norm{\nabla f_x(\otheta^{(t_0)})}_F}{\delta} \right\rceil
\\ &\leq \left\lceil \frac{18(k-1)}{\lambda} \log \frac{\lambda r}{8\delta} \right\rceil
= \left\lceil \frac{18(k-1)}{\lambda} \log \frac{\lambda\zeta}{16\sqrt{(k+1)\dmax} \cdot \delta} \right\rceil
\end{align*}
further iterations we will encounter a point $\otheta=\otheta^{(t)}$ such that $\norm{\nabla f_x(\otheta)}_F \leq \delta$, i.e., such that the algorithm will stop.


As the above shows that the iteration complexity of \CREFmain{alg:flip-flop}{Algorithm 2} is bounded by $T_0 + T_1$, combining the above expressions yields our desired bound.

\textbf{Case 2:} initial guess $\wttheta$ s.t. $\dop(\wttheta, \htheta) \leq \frac{\lambda \zeta}{100 \dmax \sqrt{k(k+1)}}  = \frac{\lambda r}{50 \sqrt{k \dmax}}$ and $\nabla_0 f_x(\wttheta) = 0$.

Let $\nu := \frac{\lambda r}{50 \sqrt{k \dmax}}$.
In this case, there is $H = (0; H_1, \dots, H_k)$ such that $H_a \in \Mat(d_a)$ are symmetric matrices with $\norm{H_a}_{op} \leq \norm{H}_{op} \leq \dop(\wttheta, \htheta) \leq  \nu$ such that $\wttheta =  \htheta^{1/2} e^H \htheta^{1/2}$.
Thus,
\begin{align*}
    \norm{\nabla f_x(\wttheta)}_F^2
    &= |\nabla_0 f_x(\wttheta)|^2 + \sum_{a=1}^k \norm{\nabla_a f_x(\wttheta)}_F^2 = \sum_{a=1}^k \norm{\nabla_a f_x(\wttheta)}_F^2
    \leq \sum_{a=1}^k d_a \cdot \norm{\nabla_a f_x(\wttheta)}_{op}^2 \\
    &= \sum_{a=1}^k d_a \cdot \norm{\nabla_a f_{\htheta^{1/2} x}(e^H)}_{op}^2
    \leq \sum_{a=1}^k d_a \cdot \left(\frac{25}{4} \cdot \norm{H}_{op} \right)^2
\end{align*}
where in the last line above we used \CREFmain{remark:gradient-everywhere}{Remark~2.10} for the first equality and \cref{constantRobustness} for the second inequality.
Since $\norm{H}_{op} \leq \nu$, we have
$$ \norm{\nabla f_x(\wttheta)}_F^2
\leq \sum_{a=1}^k d_a \cdot \left(\frac{25}{4} \cdot \norm{H}_{op} \right)^2
\leq k \dmax \cdot \left(\frac{25}{4} \cdot \dop(\wttheta, \htheta) \right)^2
\leq \left( \frac{\lambda r}{8} \right)^2$$
and thus \cref{lem:gradient-strong-convexity-fm} tells us that $\wttheta$ is contained in a sublevel set of~$f_x$ on which~$f_x$ is $\lambda$-strongly geodesically convex.

Therefore, \cref{lem:tensor-descent-lemma} shows that each subsequent step of the algorithm will decrease the value of the objective function in accordance with the requirements of \cref{lem:descent-sublevel-set}, with parameters $\alpha = \frac1{6(k-1)}$ and $\beta$ as defined above.
Thus, for any $\delta > 0$, within at most
\begin{align*}
&\left\lceil \frac{6(k-1)}{\lambda} \log_2 \frac{\norm{\nabla f_x(\wttheta)}_F^2}{\delta^2} \right\rceil
\leq \left\lceil \frac{12(k-1)}{\lambda} \log_2 \frac{25 \sqrt{k \dmax} \cdot \dop(\wttheta, \htheta)}{4 \delta} \right\rceil
\end{align*}
further iterations we will encounter a point $\otheta=\otheta^{(t)}$ such that $\norm{\nabla f_x(\otheta)}_F \leq \delta$, i.e., such that the algorithm will stop.

\textbf{Case 3:} initial guess $\frac{1}{f_x(I_D)} \cdot I_D$.

We only need to bound our expression for~$T_0$ in \cref{eq:T_0 expr}.
By \cref{lem:flip-flop-sinkhorn}, it is enough to bound $1 + \log f_x(I_D) - f_x^*$, where $f_x^* := f_x(\htheta)$.
On the one hand,
\begin{align*}
  f_x^*
&= f_x(\htheta)
= \tr \htheta \rho_1 - \frac1D \log \det \htheta
= 1 - \frac1D \log \det \left( \Theta^{1/2} \Theta^{-1/2} \htheta \Theta^{-1/2} \Theta^{1/2} \right) \\
&\geq 1 - \frac\zeta2 - \frac1D \log \det \Theta
\geq 1 - \frac\zeta2 - \log \norm{\Theta}_{\op},
\end{align*}
where the third equality follows since $\nabla_0 f_x(\htheta) = \tr \htheta \rho_1 - 1 = 0$ at the MLE; the final inequality holds because $\dop(\htheta, \Theta) \leq \zeta/2$ by our third assumption, hence $\Theta^{-1/2} \htheta \Theta^{-1/2} \preceq e^{\zeta/2} I_D$.
On the other hand,
\begin{align*}
  f_x(I_D)
= \tr \rho_1
= \tr \Theta^{-1} \Theta \rho_1
\leq \norm{\Theta^{-1}}_{\op} \tr \Theta \rho_1
\leq \frac32 \norm{\Theta^{-1}}_{\op},
\end{align*}
using the second assumption, which states that $\abs{\nabla_0 f_x(\Theta)} = \abs{\tr \Theta \rho_1 - 1} \leq \frac12$.
Thus,
\begin{align*}
  \log f_x(I_D) - f_x^*
&\leq \log \frac32 + \log \norm{\Theta^{-1}}_{\op} - 1 + \frac\zeta2 + \log \norm{\Theta}_{\op} \leq
\log \kappa(\Theta),
\end{align*}
using the assumption that $\zeta\leq1$.
Finally, we obtain
\begin{align*}
 T_0
 &\leq \left\lceil 3 + 6 (k+1) \bigl( 1 + \log \kappa(\Theta) \bigr) \frac{64}{r^2 \lambda^2} \right\rceil
 = \left\lceil 3 + 1536 \bigl( 1 + \log \kappa(\Theta) \bigr) \frac{(k+1)^2 \dmax}{\zeta^2 \lambda^2} \right\rceil.
\end{align*}
\end{proof}

\end{appendix}

\else
\begin{supplement}
\stitle{}
\vspace{-10pt}
\sdescription{The full proofs of the claims mentioned above and a full discussion and comparison between this work and previous works can be found in~\cite{FORW25supp}.}
\end{supplement}
\fi

\bibliographystyle{imsart-number}
\bibliography{bibliography}

\end{document}